\documentclass[11pt,a4paper]{article}
\usepackage{geometry}
\usepackage{amsfonts}
\usepackage[latin1]{inputenc}
\usepackage{amssymb , amsmath, amsthm, amsopn, amstext, amscd}
\usepackage[T1]{fontenc}
\usepackage{latexsym, enumerate}
\usepackage{graphics}
\usepackage{graphicx}
\usepackage{a4wide}
\usepackage{multirow}

\newtheorem{definit}{Definition}[section]
\newtheorem{prp}{Proposition}

\newtheorem{thm}{Theorem}
\newtheorem{lm}{Lemma}
\newtheorem{rmq}{Remark}

\newtheorem{cor}{Corollary}

\DeclareMathOperator*{\argmin}{argmin}
\DeclareMathOperator*{\pen}{pen}

\newcommand{\abs}[1]{\left\vert#1\right\vert}

\title{The Smooth-Lasso and other $\ell_1+\ell_2$-penalized methods}
\author{\noindent{Mohamed Hebiri and Sara van de Geer}}

\date{}

\begin{document}

\maketitle

\begin{abstract}
We consider a linear regression problem in a high dimensional setting where the number of covariates $p$ can be much larger than the sample size $n$. In such a situation, one often assumes sparsity of the regression vector, {\it i.e.,} the regression vector contains many zero components. We propose a Lasso-type estimator $\hat{\beta}^{Quad}$ (where `$Quad$' stands for quadratic) which is based on two penalty terms. The first one is the $\ell_1$ norm of the regression coefficients used to exploit the sparsity of the regression as done by the Lasso estimator, whereas the second is a quadratic penalty term introduced to capture some additional information on the setting of the problem. We detail two special cases: the Elastic-Net $\hat{\beta}^{EN}$ introduced in \cite{Zou-E-Net}, which deals with sparse problems where correlations between variables may exist; and the Smooth-Lasso\footnote{The Smooth-Lasso estimator has initially been introduced in the paper titled {\it Regularization with the Smooth-Lasso procedure}, in \cite{Mo7SLasso}. Results can be found there for the this method which are not provided here, such as the theoretical performance when $p\leq n$ and a simulation study from a variable selection point of view.} $\hat{\beta}^{SL}$, which responds to sparse problems where successive regression coefficients are known to vary slowly (in some situations, this can also be interpreted in terms of correlations between successive variables).
From a theoretical point of view, we establish variable selection consistency results and show that $\hat{\beta}^{Quad}$ achieves a Sparsity Inequality, {\it i.e.,} a bound in terms of the number of non-zero components of the `true' regression vector. These results are provided under a weaker assumption on the Gram matrix than the one used by the Lasso. In some situations this guarantees a significant improvement over the Lasso.
Furthermore, a simulation study is conducted and shows that the S-Lasso $\hat{\beta}^{SL}$ performs better than known methods as the Lasso, the Elastic-Net $\hat{\beta}^{EN}$, and the Fused-Lasso (introduced in \cite{Rosset-Fused}) with respect to the estimation accuracy. This is especially the case when the regression vector is `smooth', {\it i.e.,} when the variations between successive coefficients of the unknown parameter of the regression are small. The study also reveals that the theoretical calibration of the tuning parameters and the one based on $10$ fold cross validation imply two S-Lasso solutions with close performance.
\end{abstract}

\section{Introduction}

We focus on the usual linear regression model
\begin{equation}
\label{ChapSLeq_depart}
	y_{i}= x_{i} \beta^{*} + \varepsilon_{i}, \quad \quad i=1,\ldots,n,
\end{equation}
where the design $x_{i}=(x_{i,1},\ldots,x_{i,p}) \in \mathbb{R}^p$ is deterministic, $\beta^*=(\beta^*_1,\ldots,\beta^*_p)' \in \mathbb{R}^p$ is the unknown parameter, and $\varepsilon_1,\ldots,\varepsilon_n,$ are independent, identically distributed (i.i.d.) centered Gaussian random variables with known variance $\sigma^{2}$. We aim on estimating $\beta^{*}$ in the sparse case, that is, when many of its unknown components are zero. Thus only a subset of the design covariates $(X_{j})_{j}$ is truly of interest where $X_{j} = (x_{1,j},\ldots,x_{n,j})',\,j=1,\ldots,p$. Moreover, we are interested in the high dimensional problem where $p\gg n$ and we consider $p$ depending on $n$.
In such a framework, two main problems arise: the interpretability of the prediction and the control of the variance in the estimation.
To tackle these problems we use regularized selection type procedures of the form
\begin{equation}
\label{ChapSLeq:penalized-risk}
	\tilde \beta=\underset{\beta \in \mathbb{R}^{p}}{\argmin}\left\lbrace \Vert Y- X\beta\Vert_n^2 + \pen(\beta) \right \rbrace,
\end{equation}
where $X=(x_{1}',\ldots,x_{n}')'$, $Y=(y_{1},\ldots,y_{n})'$ and $\pen:  \, \mathbb{R}^p \rightarrow \mathbb{R}$ is a positive convex function called the penalty. For any vector $a=(a_{1},\ldots,a_{n})'$, we have adopted the notation $\Vert a \Vert_n^2 = n^{-1} \sum_{i=1}^n |a_{i}|^{2}$ and we denote by $<\cdot,\cdot>_n$ the corresponding inner product in $\mathbb{R}^{n}$. The choice of the penalty appears to be crucial. On the one hand, although well-suited for variable selection purpose, concave-type penalties (see for example \cite{ArnakTsyb,FanLiScad,Tsyb-VanDeGeer-SquareRoot}) are often computationally hard to optimize. On the other hand, Lasso-type procedures (modifications of the $\ell_{1}$ penalized least square (Lasso) estimator introduced in \cite{Tibshirani-LASSO}) have been extensively studied during the last few years. See for example \cite{Lasso3,Bunea_consist,Lasso2,BiYuConsistLasso} and references therein. Such procedures are suitable for our purposes as they perform both regression parameters estimation and variable selection with low computational costs. We will explore this type of procedures in our study.

In this paper, we propose a novel estimator, denoted by $\hat{\beta}^{Quad}$, which is a modification of the Lasso. It is defined as the solution of the optimization problem~\eqref{ChapSLeq:penalized-risk} for a combination of the Lasso penalty ({\it i.e.,} $\sum_{j=1}^{p} |\beta_j|$) and the quadratic penalty $\beta' \mathbf{J}' \mathbf{J} \beta$ for some $m \times p$ matrix~$\mathbf{J}$ ($m\in \mathbb{N}^*$). 

The matrix $\mathbf{J}$ typically reflects some underlying geometry or structure in the true signal.
More generally, the matrix $\mathbf{J}$ can be chosen so that sparsity of $\beta^*$ translates to some other desired behavior depending on the context. There is a wide variety of interesting applications, and what we present below is not meant to be an exhaustive list but rather a small set of illustrative examples that motivated our work on this problem.
We add this second term to the Lasso procedure for two major issues. First, we exploit this second penalty to take into account some prior information on the data or the regression vector (such as correlation between variables or a specified structure on the regression vector).
Second, the quadratic penalty is introduced to overcome (or to reduce) theoretical problems observed by the Lasso estimator. Indeed, (see for example \cite{Lasso3,Bunea_consist,KarimNormSup,MeinshBulhmConsistLasso,WainSelection,NonNegativeGarotte,BiYuConsistLasso,AdapLassoZou}) strong conditions to guarantee good performance in prediction, estimation or variable selection for the Lasso procedure are required. See also \cite{VandeGeerConditionLasso09} for an overview of the conditions used to establish the theoretical results according to the Lasso. It was shown that the Lasso does not always ensure good performance when high correlations exist between the covariates.
In this paper, we establish theoretical results showing good performance of $\hat{\beta}^{Quad}$ under a weaker assumption than the Lasso estimator.
The improvement is especially observed when the Lasso achieves only poor results.\\
\noindent Two particular cases of the estimator $\hat{\beta}^{Quad}$ are mainly considered: the Elastic-Net introduced in \cite{Zou-E-Net} to deal with problems where correlations between variables exist. It is defined with the quadratic penalty term $\sum_{j=1}^p \beta_j^2$. The second and novel procedure is called the \textit{Smooth-Lasso} (\textit{S-Lasso}) estimator. It is defined with the $\ell_2$-fusion penalty, that is, $\sum_{j=2}^{p}\left(\beta_{j}-\beta_{j-1}\right)^{2}$. The $\ell_2$-fusion penalty was first introduced in \cite{Fusion}. This term helps to tackle situations where the regression vector is structured such that its coefficients vary slowly. Let us call the regression vector `smooth' in this case. Note, however, that our theoretical study takes into account a large amount of procedures such as the  closely related `Weighted Fusion' introduced in \cite{Daye09Referee}. This is detailed in Remark~\ref{rq:weightFusion}.

The main contribution of this paper is the introduction of the Smooth-Lasso estimator which significantly improves (both in theory and in practice) the performance of the Lasso and the Elastic-Net in some situations.
However, the method is a special case of the estimator $\hat{\beta}^{Quad}$.
This type of estimators aims on
\begin{itemize}
	\item capturing the sparsity and some other structure (smoothness in the case of the S-Lasso);
	\item reducing the assumptions on the Gram matrix and providing theoretical guarantees in situations that are not suitable for the Lasso (correlations between successive covariates in the case of the S-Lasso).
\end{itemize}

From a practical point of view, some problems are also encountered when we solve the Lasso criterion (for instance with the LARS algorithm \cite{Efron-LARS}). Indeed, this algorithm fails to select a complete group of correlated covariates.
We describe two disadvantages of the Lasso. First, the Lasso is not consistent neither in variable selection nor in estimation (bad reconstitution of $\beta^*$). In this paper, we focus on the estimation issue. We consider the case where the regression vector $\beta^*$ is structured. We invoke the \textit{S-Lasso} estimator to respond to such problems where the covariates are ranked so that the regression vector is `smooth' (that is, the vector $\beta^*$ has only small variations in its successive components). We will see with the help of simulations that such situations support the use of the \textit{S-Lasso} estimator. This estimator is inspired by the \textit{Fused-Lasso} \cite{Rosset-Fused}.
Both S-Lasso and Fused-Lasso combine a $\ell_{1}$-penalty with a fusion term \cite{Fusion}. The fusion term is designed to make successive coefficients as close as possible to each other.
The main difference between these two procedures is that we use the $\ell_{2}$ distance between the successive coefficients (that is, the $\ell_2$-fusion penalty: $\sum_{j=2}^{p}(\beta_{j}-\beta_{j-1})^2$) whereas the Fused-Lasso uses the $\ell_{1}$ distance (that is, the $\ell_1$-fusion penalty: $\sum_{j=2}^{p}|\beta_{j}-\beta_{j-1}|$). Hence, compared to the Fused-Lasso, we sacrifice sparsity  in changes between successive coefficients in the estimation of $\beta^*$ for an easier optimization due to the strict convexity of the $\ell_{2}$ distance.
This implies a large reduction of computational cost. However, sparsity is, nonetheless, ensured by the Lasso penalty. The $\ell_2$-fusion penalty helps to provide `smooth' solutions.
Consequently, even if there is no perfect match between successive coefficients, our results are still interpretable. From a theoretical point of view, the $\ell_2$ distance also helps us to provide theoretical properties for the S-Lasso which in some situations appears to outperform the Lasso and the Elastic-Net ({\it cf.} \cite{Zou-E-Net}). Let us mention that variable selection consistency of the Fused-Lasso and the corresponding Fused adaptive Lasso have also been studied in \cite{Rinaldo08FusedAdaptive} but in a different context from the one in the present paper. The results obtained in~\cite{Rinaldo08FusedAdaptive} are established not only under the sparsity assumption, but the model is also supposed to be {\it piecewise constant}, that is, the non-zero coefficients are represented in a block shape with equal values inside each block.

Many techniques have been proposed to address the weaknesses of the Lasso. The Fused-Lasso procedure is one of them.
Additionally we give here some of the most popular alternative methods.
The Adaptive Lasso was introduced by~\cite{AdapLassoZou}. It is similar to the Lasso but with adaptive weights used to penalize each regression coefficient separately. This procedure reaches under certain (strong) conditions {\it Oracles Properties} (that is, consistency in variable selection and asymptotic normality, see~\cite{AdapLassoZou}).
Another approach is the Relaxed Lasso (see \cite{RelaxedLasso}), which aims on double-controlling the Lasso estimate: one parameter to control variable selection and another to control the shrinkage of the selected coefficients.
To overcome the problem due to the correlation between covariates, group variable selection has been proposed in \cite{GpLasso1} with the Group-Lasso procedure which selects groups of correlated covariates instead of single covariates at each step. A first step to the variable selection consistency study has been proposed in \cite{BachGpLasso} and Sparsity Inequalities were given in \cite{ChriMo7GpLass,GpLassLogistic}.
In \cite{Zou-E-Net}, another choice of penalty has been proposed with the Elastic-Net. This penalty has also been studied for example in~\cite{BuneaEN,JY_EN_08,ZZ_AdaptEN_09}.

The rest of the paper is organized as follows.
In the next section, we introduce the estimator $\hat{\beta}^{Quad}$ defined with the Lasso penalty together with a quadratic penalty.
In particular, we define the S-Lasso estimator and a notion of smoothness.
We also provide a way to solve the $\hat{\beta}^{Quad}$ problem with the attractive property of piecewise linearity of its regularization path.
Consistency in estimation and variable selection in the high dimensional case are considered in Section~\ref{ChapSLsec:theoryPgrd}.
We moreover provide some examples in favor of the Elastic-Net and the S-Lasso in Sections~\ref{subsec:ElasticNet} and~\ref{subsec:SL} and technical issues in Section~\ref{sec:Techn}.
We finally give experimental results in Section~\ref{sec:simul} which show the S-Lasso performance compared to some popular methods.
All proofs are postponed to the Appendix.

\section{The S-Lasso procedure}
\label{ChapSLsolve}

In many applications for example in macroeconomics, financial time series analysis, and biological and medical sciences one often deals with data with given complex attributes and a `smooth' solution.
This is, for instance, the case in trend filtering (see~\cite{KimKBoyd_trending_09} for a nice survey).

As a start, let us provide a definition of a `smooth' vector:
\begin{definit}[Smoothness]
\label{Def:Smooth}
 Let $\alpha$ be some positive number. A vector $\beta\in \mathbb{R}^p$ is $\alpha$-smooth (or simply smooth) if
	$$
	\sum_{j=2}^{p} \left(\beta_{j}-\beta_{j-1}\right)^{2} \leq \alpha.
	$$
\end{definit}
In the applications mentioned above, the regression vector $\beta^*$ is smooth.
Hence, it is important to consider estimation methods which can reflect this aspect of the problem. It is often useful to assume that the regression vector is also sparse in order to be able to treat data such as spectrometry or some genomic data, where both smoothness and sparsity appear simultaneously.
For these reasons, it is worth introducing and analyzing a method which can reconstitute sparse and smooth regression vectors.
Hence, we define the S-Lasso estimator $\hat{\beta}^{SL}$ as the solution of the optimization problem~\eqref{ChapSLeq:penalized-risk} with the penalty
\begin{equation}
\label{ChapSLcritere_S-lasso}
	\pen(\beta) = \lambda | \beta |_{1} +
	\mu\sum_{j=2}^{p}\left(\beta_{j}-\beta_{j-1}\right)^{2},
\end{equation}
where $\lambda$ and $\mu$ are two positive parameters that control the sparsity of our estimator and its smoothness. For any vector $a=(a_{1},\ldots,a_{p})'$ and integer $q$, we have used the notation $| a |_{q}^q = \sum_{j=1}^{p}\abs{a_{j}}^q$. Note that the Lasso estimator is a special case of the S-Lasso with $\mu = 0$. More generally, we consider the following penalty
\begin{equation}
\label{ChapSLcritere_S-lassoGneneram}
	\pen(\beta) = \lambda | \beta |_{1} +
	\mu\beta' \mathbf{J}'\mathbf{J} \beta,
\end{equation}
where $\mathbf{J}$ is a given $m\times p$ matrix ($m\in \mathbb{N}^*$).
This penalty is a combination of the Lasso penalty and a quadratic penalty.
The matrix $\mathbf{J}$ typically reflects some underlying geometry or structure in the true signal (we refer to~\cite{TibJunior_Smooth_10} for similar ideas).
Let us call $\hat{\beta}^{Quad}$ the solution of the minimization problem~\eqref{ChapSLeq:penalized-risk} and \eqref{ChapSLcritere_S-lassoGneneram}.
The S-Lasso penalty is a particular case of the penalty~\eqref{ChapSLcritere_S-lassoGneneram} with $\mathbf{J}$ given by
	\begin{equation}
	\label{ChapSLJmatrix}
	\mathbf{J}=\begin{pmatrix}
	  0 & 0 & 0 & \ldots & 0 \\
	  1 & -1 & \ddots & \ddots & \vdots \\
	  0 & 1 & -1 & \ddots & 0 \\
	  \vdots & \ddots &\ddots & \ddots & 0 \\
	  0 & \ldots & 0 & 1 & -1
	\end{pmatrix}.
	\end{equation}
The Elastic-Net corresponds to the case where $\mathbf{J}$ is the identity matrix.
\begin{rmq}
\label{rq:weightFusion}
	For any $j,k\in\{1,\ldots,p\}$, denote by $s_{j,k} = sign\left(\frac{X_j'X_k}{n}\right)$ the sign of the sample correlation between predictor variables $j$ and $k$. Denote also by $w_{j,k}\geq 0$ some predictor correlation driven weights. Given this notation, the Weighted Fusion introduced in \cite{Daye09Referee} corresponds to the case where the $k$-th diagonal term of $\mathbf{J}$ equals $w_{k,k}$ and $(\mathbf{J})_{k,j} = (\mathbf{J})_{j,k} = -s_{j,k} w_{j,k}$ for $j\neq k$.
\end{rmq}

Now we deal with the solution $\hat{\beta}^{Quad}$ of~\eqref{ChapSLeq:penalized-risk} and \eqref{ChapSLcritere_S-lassoGneneram} and its computational costs. The following lemma shows that $\hat{\beta}^{Quad}$ can be expressed as a Lasso solution by expanding the data artificially.
\begin{lm}
\label{ChapSLequivalenceLasso}
	Given the dataset $(X,Y)$ and the tuning parameters $(\lambda,\mu)$, define the extended dataset
	$(\widetilde{X},\widetilde{Y})$ and $\widetilde{\varepsilon}$ by
	$$\widetilde{X}= \begin{pmatrix}
	X \\
	\sqrt{n \mu}\mathbf{J}
	\end{pmatrix}, \quad
	\text{and} \quad \widetilde{Y}=
	\begin{pmatrix}
	Y \\
	\mathbf{0}
	\end{pmatrix},
	 \quad
	\text{and} \quad \widetilde{\varepsilon}=
	\begin{pmatrix}
	\varepsilon \\
	- \sqrt{n \mu}\mathbf{J}\beta^*
	\end{pmatrix},$$
	where $\mathbf{0}$ is a vector of size $p$ containing only zeros, $\varepsilon = (\varepsilon_1,\ldots,\varepsilon_n)'$ is the noise vector and $\mathbf{J}$ is the $m\times p$ matrix given by the penalty~\eqref{ChapSLcritere_S-lassoGneneram}. Then, we have $\widetilde{Y} = \widetilde{X}\beta^* + \widetilde{\varepsilon}$, and the estimator $\hat{\beta}^{Quad}$, defined as the solution of the minimization problem~\eqref{ChapSLeq:penalized-risk} with the penalty given by~\eqref{ChapSLcritere_S-lassoGneneram}, is also the minimizer of the Lasso-criterion
	\begin{equation}
	\label{ChapSlassEq:CritAugme}
	\frac{1}{n}\left|\widetilde{Y}-\widetilde{X} \beta \right|_{2}^2+ \lambda | \beta |_{1}.
	\end{equation}
\end{lm}
This result is a consequence of simple algebra. It motivates the following comments on the estimator $\hat{\beta}^{Quad}$.

\begin{rmq}[{\it Regularization paths}]
\label{ChapSLrq:RegulPath}
	LARS is an iterative algorithm introduced in \cite{Efron-LARS}. A modification of LARS can be used to construct $\hat{\beta}^{Quad}$. For a fixed $\mu$, it constructs at each step an estimator based on the correlation between covariates and the current residual. Each step corresponds to a value of $\lambda$. Then, for a fixed $\mu$, we obtain the evolution of the coefficients values of $\hat{\beta}^{Quad}$ when $\lambda$ varies. This evolution describes the regularization paths of $\hat{\beta}^{Quad}$ which are piecewise linear (see \cite{Rosset-PeacewiseLin}). This property implies that (again for fixed $\mu$) the problem~\eqref{ChapSLeq:penalized-risk} and \eqref{ChapSLcritere_S-lassoGneneram} can be solved using the LARS algorithm with the same computational cost as the ordinary least square (OLS) estimate.
\end{rmq}


\section{Theoretical results in the high dimensional setting}
\label{ChapSLsec:theoryPgrd}

In this section, we study the performance of the estimator $\hat{\beta}^{Quad}$ in the high dimensional case. In particular, we provide a non-asymptotic bound for the squared risk. We also provide a bound for the $\ell_2$ estimation error of $\hat{\beta}^{Quad}$. Let
$$\widetilde{J} = \mathbf{J}'\mathbf{J},$$
be the $p \times p$ matrix where $\mathbf{J}$ is the matrix appearing in the quadratic penalty~\eqref{ChapSLcritere_S-lassoGneneram}. Since our main interest is the study of the S-Lasso estimator, we first focus on the case where the matrix $\widetilde{J}$ is sparse. We refer the reader to Section~\ref{sec:Techn} where we address several technical points, for example the study of the case where the matrix $\widetilde{J}$ is general. \\
\noindent All the results of this section are proved in Section~\ref{sec:Proof}. These theoretical contributions rely partly on Lemma~\ref{ChapSLequivalenceLasso}. Let us finally mention that the tuning parameters $\lambda$ and $\mu$ will actually be chosen depending on the sample size $n$. We emphasize this dependency by adding a subscript $n$ to these parameters.

\subsection{Sparsity Inequality when $\widetilde{J}$ is sparse}
\label{ChapSLsec:SOI}

Now we establish a Sparsity Inequality (SI) achieved by the estimator $\hat{\beta}^{Quad}$, that is, a bound on the squared risk that takes into account the sparsity of the regression vector $\displaystyle{\beta^*}$. More precisely, we prove that the rate of convergence of $\hat{\beta}^{Quad}$ is $\displaystyle{\max ( |\mathcal{A}^*| \log (n) / n ;   \mu_n^2 |\widetilde{J}\beta^*|^2) }$, where $\mathcal{A}^*$ is the sparsity set $\mathcal{A}^* = \{j:\, \beta_j^* \neq 0\}$. This rate depends not only on the sparsity index $|\mathcal{A}^*|$ but also on $|\widetilde{J}\beta^*|$. In the case of the S-Lasso, this last quantity is related to the smoothness of the vector $\beta^*$. 
Let us first present the assumptions needed, and the setup of this contribution. Let $\eta \in (0,1)$ be given ($(1-\eta)$ will be a confidence bound, see Theorem~\ref{ChapSLThm:doubkeSparsJ}). We define the tuning parameter $\lambda_n$ as
\begin{equation}
	\label{ChapSLeq:TuneLambda}
	\lambda_{n} = 4\sqrt{2} \sigma \sqrt{\frac{\log(p/\eta)}{n}}.
\end{equation}
For now, we leave the calibration of $\mu_n$ free. We discuss later (see Corollary~\ref{Cor:BorneL1Sparse} and Section~\ref{subsec:Mu} for example) the choice for this parameter.
Our assumption on the Gram matrix $\Psi^n := n^{-1}X'X$ involves the symmetric $p \times p$ matrix $K_n$ defined by
\begin{equation}
	\label{eq:matrixKn}
	K_n = \Psi^n + \mu_n \widetilde{J}.
\end{equation}	
Given the expanded dataset defined in Lemma~\ref{ChapSLequivalenceLasso}, we note that $K_n = n^{-1} \widetilde{X}'\widetilde{X}$ can be seen as an expanded Gram matrix. Let $\Theta \subset \{1,\ldots,p\} $ be a set of indices. Using this notation, we formulate the following assumption:
\begin{description}
	\item[Assumption~$B(\Theta)$:] {\it Let $K_n$ be the matrix given by~\eqref{eq:matrixKn} and let $ \varrho_n = 4 \sqrt{|\mathcal{A}^*|} +  \frac{4\mu_n}{\lambda_n} | \widetilde{J}\beta^*|_2 $. There is a constant $\phi_{\mu_n} > 0$ such that, for any $\Delta\in\mathbb{R}^{p}$ that satisfies
$ \sum_{j\notin \Theta} \left|\Delta_{j}\right| \leq  \varrho_n 
\sqrt{ \sum_{j\in \Theta} \Delta_{j}^2 }$, we have
	\begin{equation}
	\label{eq:CondHypEase}
		\Delta' K_n \Delta 
		\geq
		\phi_{\mu_n} \sum_{j\in \Theta} \Delta_{j}^{2} .
	\end{equation}
}
\end{description}
Here are some comments about this assumption:
\begin{itemize}
\item first of all, Assumption~$B(\Theta)$ is inspired by the Restricted Eigenvalue (RE) Assumption introduced in~\cite{Lasso3}. The RE assumption is widely used in the literature and requires somehow that the restriction of the matrix $K_n$ to the rows and columns in $\Theta$ is invertible (when $K_n$ is invertible, the condition~\eqref{eq:CondHypEase} is always satisfied with $\phi_{\mu_n}$ at least as large as the smallest eigenvalue of $K_n$) . We refer to~\cite{Lasso3,VandeGeerConditionLasso09} for more details on this assumption. The main difference with the assumption we use is that in~\cite{Lasso3} the authors consider the case where $K_n = \Psi^n $, which matches with the Lasso estimator (that is $\mu_n = 0$ in our setting).\\
\noindent In the sequel, let $\phi_0$ denote $\phi_{\mu_n}$ for $\mu_n = 0$, that is, the case of the Lasso estimator;

\item another difference to~\cite{Lasso3} is that the set on which the assumption should hold is larger in Assumption~$B(\Theta)$ than in the RE Assumption. Indeed, in Assumption~$B(\Theta)$, the considered vectors $\Delta$ should be such that $ \sum_{j\notin \Theta}        \left|\Delta_{j}\right| \leq \varrho_n  \sqrt{ \sum_{j\in \Theta} \Delta_{j}^2 }$, whereas in~\cite{Lasso3} the authors only need to consider vectors $\Delta$ such that $ \sum_{j\notin \Theta}  |\Delta_{j}| \leq cst \cdot \sum_{j \in \Theta} |\Delta_{j}|  $ (see also \cite{VandeGeerConditionLasso09}). We make this set larger to allow large values of the tuning parameter $\mu_n$. We will explain later why this is desirable;

\item in the case of the Elastic-Net, $\Theta = \mathcal{A}^*$ in Assumption~$B(\Theta)$. Hence, the assumption above is close to {\it Condition Stabil} in~\cite[page~4]{BuneaEN} for the Elastic-Net. We will consider precisely the difference between both assumptions in Section~\ref{subsec:ElasticNet}. However, let us mention here that in {\it Condition Stabil} the condition~\eqref{eq:CondHypEase} is replaced by
 \begin{equation*}
	\Delta'\Psi^n \Delta \ge (\phi_{\mu_n}^{CS} - \mu_n) |\Delta_{\mathcal{A}^*}|_2^2 
 \end{equation*}
for a constant $\phi_{\mu_n}^{CS}>\mu_n$;

\item only small subsets $\mathcal{B} $ of indices $\Theta$ will be considered in Assumption~$B(\Theta)$. More precisely, let $\mathcal{B} \subset \{1,\ldots,p\}$ be a set of indices such including the true sparsity set $\mathcal{A}^*$. We will consider a set depending on $\widetilde{J}$ and on $\mathcal{A}^*$, and the sparser $\widetilde{J}$, the smaller $\mathcal{B}$. For instance, in the case of the Elastic-Net, $\mathcal{B} = \mathcal{A}^*$, and in the case of the S-Lasso (that we will detail later), the set $\mathcal{B}$ is such that $|\mathcal{B} | \leq 3 |\mathcal{A}^* |$. Thanks to the sparsity of $\widetilde{J}$, we will see that we can assume that there exists a constant $c_{\widetilde{J}} \geq 1$ such that $|\mathcal{B} | \leq c_{\widetilde{J}}  |\mathcal{A}^* |$ (see Sections~\ref{subsec:ElasticNet} and~\ref{subsec:SL}).

\end{itemize} 

Theorem~\ref{ChapSLThm:doubkeSparsJ} below holds for general matrices $\widetilde{J}$. However we emphasized here the sparse case since Assumption~$B(\mathcal{B})$ with large sets $\mathcal{B}$ is more stringent (with $\phi_{\mu_n}$ close or equal to zero).
Hence in the general case, another assumption presented in Section~\ref{subsecGenMatrix} may be more attractive. We also mention that Theorem~\ref{ChapSLThm:doubkeSparsJ} is formulated as general as possible. We refer to Corollary~\ref{Cor:BorneL1Sparse} below for a special case illustrating the superiority of $\hat{\beta}^{Quad}$ compared to the Lasso.
\begin{thm} [$\widetilde{J}$ sparse]
\label{ChapSLThm:doubkeSparsJ}
	Let $\mathcal{A}^*$ be the sparsity set. Let the tuning parameter $\lambda_n$ be defined as in~\eqref{ChapSLeq:TuneLambda}. Suppose that Assumption~$B(\mathcal{B})$ is satisfied with a set $\mathcal{B}\supset \mathcal{A}^* $ such that $|\mathcal{B}| \leq c_{\widetilde{J}}  |\mathcal{A}^* |$ for a given constant $c_{\widetilde{J}}\geq 1$. Then, with probability greater than $1- \eta$, we have	
	\begin{equation}
	\label{th:DenoizzError}
		\left\|X \beta^* -X \hat{\beta}^{Quad} \right\|_{n}^{2}
		\leq
		\phi_{\mu_n}^{-1} ( 2 \lambda_n \sqrt{|\mathcal{A}^*| }  + 2 \mu_n |\widetilde{J}\beta^*|_2 )^2 ,
	\end{equation}
	\begin{equation}
	\label{th:SmoothError}
		(\beta^* - \hat{\beta}^{Quad})' \widetilde{J} (\beta^* - \hat{\beta}^{Quad})
		\leq
		\phi_{\mu_n}^{-1} \frac{( 2 \lambda_n \sqrt{|\mathcal{A}^*| }  + 2 \mu_n |\widetilde{J}\beta^*|_2 )^2}{\mu_n} ,
	\end{equation}
	and
	\begin{equation*}
		|\beta^* - \hat{\beta}^{Quad}|_1
		\leq
		2\phi_{\mu_n}^{-1} \frac{( 2 \lambda_n \sqrt{|\mathcal{A}^*| }  + 2 \mu_n |\widetilde{J}\beta^*|_2 )^2}{  \lambda_n}.
	\end{equation*}
\end{thm}
Theorem~\ref{ChapSLThm:doubkeSparsJ} states that $\hat{\beta}^{Quad}$ achieves a SI which also brings the quantity $|\widetilde{J}\beta^*|_2$ into play. 
A first glance at the bounds above would suggest that $\mu_n = 0$ provides the best rates.
However, it is worth noting that $\phi_{\mu_n}$, one of the main terms of the bounds, also depends on $\mu_n $ and increases with this parameter since $\widetilde{J}$ is positive semidefinite.
Calibration of $\mu_n$ captures the tradeoff between slowing down the rate of convergence and being able to address situations where the Lasso fails.
For instance, the Smooth-Lasso with a large $\mu_n$ is devoted to problems with large correlations between successive variables.
In Section~\ref{subsec:Mu}, we further discuss the importance of a good calibration of $\mu_n$ and the interest of using $\hat{\beta}^{Quad}$ (with $\mu_n$ different from zero) instead of the Lasso estimator. These considerations lead to the following Corollary~\ref{Cor:BorneL1Sparse}. It points out that the estimator $\hat{\beta}^{Quad}$ is particularly useful when the assumptions on the Gram matrix $\Psi^n$ are so restrictive that the Lasso error fails to be well controlled.

\begin{cor}
\label{Cor:BorneL1Sparse}
	Consider the same setting as in Theorem~\ref{ChapSLThm:doubkeSparsJ}.
	Let $\lambda_n = 4\sqrt{2} \sigma \sqrt{\frac{\log(p/\eta)}{n}}$ with $\eta\in (0,1)$ and $\mu_n = \frac{\lambda_{n} \sqrt{|\mathcal{A}^*|}}{2 | \widetilde{J}\beta^*|_2 }$. Then, $\varrho_n = 6\sqrt{|\mathcal{A}^*|}$ in Assumption~$B(\mathcal{B})$ and with probability greater than $1- \eta $ we have
	\begin{equation*}
		\left\|X \beta^* -X \hat{\beta}^{Quad} \right\|_{n}^{2}
		\leq
		\frac{288 \sigma^2}{ \phi_{\mu_n}} \frac{\log(p/\eta)}{n} |\mathcal{A}^*| ,
	\end{equation*}
and
	\begin{equation*}
		|\beta^* - \hat{\beta}^{Quad}|_1
		\leq
		\frac{ 72 \sqrt{2} \sigma}{ \phi_{\mu_n}} \sqrt{\frac{\log(p/\eta)}{n}}|\mathcal{A}^*| .
	\end{equation*}
	Assume furthermore that the Gram matrix $\Psi^n$ is such that $\phi_0 < \lambda_n^2 | \mathcal{A}^*|$ and that the extended Gram matrix $K_n$ is such that $\phi_{\mu_n} \ge \mu_n$. Then the bound on the Lasso (obtained setting $\mu_n=0$ above) does not guaranty any control on the errors. In contrast, $\hat{\beta}^{Quad}$ satisfies
	\begin{equation*}
		\left\|X \beta^* -X \hat{\beta}^{Quad} \right\|_{n}^{2}
		\leq
		72\sqrt{2}\sigma|\widetilde{J}\beta^*|_2  \sqrt{\frac{\log(p/\eta)}{n} |\mathcal{A}^*| },
	\end{equation*}
and if $\phi_{\mu_n} \ge \sqrt{\mu_n}$
	\begin{equation*}
		|\beta^* - \hat{\beta}^{Quad}|_1
		\leq
		36 \sqrt{\sigma |\widetilde{J}\beta^*|_2}  \left(\frac{\log(p/\eta)}{n}\right)^{1/4} |\mathcal{A}^*|^{3/4} .
	\end{equation*}
with probability greater $1- \eta$.
\end{cor}
The above bounds are even better when $|\widetilde{J} \beta^*|_2$ is small. One illustration of this corollary can be found in the example included in Section~\ref{subsec:SL}. Moreover, we refer to Section~\ref{sec:Techn} for other choices of $\mu_n$ which are more suitable when we deal with a general (non sparse) matrix~$\widetilde{J}$.\\
In our simulation study we focus on the particular choice of $\mu_n$ given in the first part of Corollary~\ref{Cor:BorneL1Sparse}.
However, in real applications, since the parameters $\lambda_n $ and $\mu_n$ depend on the unknown regression vector $\beta^*$, we tune them with the help of a 2D ten fold cross validation over a grid.

\subsubsection{Discussion around $\mu_n$ and the rate of convergence}
\label{subsec:Mu}

In this paragraph, we highlight the cases when using $\hat{\beta}^{Quad}$ is useful in the sense of Theorem~\ref{ChapSLThm:doubkeSparsJ}.
We mainly consider two aspects.
The first one deals with situations (or conditions on the Gram matrix $\Psi^n$) where $\phi_{\mu_n}$ is much larger than $\phi_0$, that is, the settings where the introduction of the additional penalty enables the estimator $\hat{\beta}^{Quad}$ to consider problems that cannot be treated by the Lasso.
The second one is the fact that $ \mu_n |\widetilde{J}\beta^*|_2$ should be dominated by $ \lambda_n \sqrt{|\mathcal{A}^*| }$. 

For the first point, and to make things more understandable, let us restrict ourselves to the above prediction error bound~\eqref{th:DenoizzError} and consider the particular case of the Elastic-Net where the matrix $\widetilde{J}$ is the identity.\\
Because of the definition of $\phi_{\mu_n}$ (in the particular case of the Elastic-Net), we have $\phi_{\mu_n} \geq \mu_n$.
We now discuss the rates of convergence of the Lasso (with $\phi_0$) and the Elastic-Net (with $\phi_{\mu_n}$) in different situations.
We present the cases in an asymptotic setting with $n $ tending to infinity. The results provided in Theorem~\ref{ChapSLThm:doubkeSparsJ} suggest essentially three regimes:
\begin{itemize}
	\item \underline{when $\phi_0$ is a constant:} in this case, the rate $|\mathcal{A}^*|\log(p) / n $ is optimal (up to a logarithmic factor; {\it cf.} \cite[Theorem 5.1]{BTWAggSOI}). This rate is reached by the Lasso (set $\mu_n = 0$ in the above Theorem~\ref{ChapSLThm:doubkeSparsJ}) and as a consequence the Elastic-Net (and more generally $\hat{\beta}^{Quad}$) does not help a lot. Indeed, whatever $\mu_n>0$, the value of $\phi_{\mu_n}$ does not significantly vary from $\phi_0$ (although $\phi_{\mu_n} > \phi_0$);
	
	\item \underline{when $\phi_0$ depends on $n$ but with $ \mu_n \leq \phi_0 < 1$:} in this case, $\phi_{\mu_n}$ (and $\phi_0$ as well) is an influencing term that should be taken into account in the rate of convergence. The rate of the Lasso is worse than $|\mathcal{A}^*|\log(p) / n $. But, since $ \mu_n < \phi_0 $, the Elastic-Net does not cause a big improvement in this case neither;

	\item \underline{when $\phi_0$ depends on $n$ and $ \mu_n > \phi_0$:} clearly here, $\phi_{\mu_n} > \phi_0$. Then when $\phi_0$ is small (or even very small), the rate of convergence of the Lasso is bad (or even the Lasso error is not controlled when $\phi_0 < \lambda_n^2 |\mathcal{A}^*|$), whereas the Elastic-Net is guarantied to reach the worst case rate $\phi_{\mu_n}^{-1} |\mathcal{A}^*|\log(p) / n $ ({\it cf.} Corollary~\ref{Cor:BorneL1Sparse} for a bound independent on the second term in the LHS of~\eqref{th:DenoizzError}). This can lead to a big improvement.
For instance, Section~\ref{subsec:SL} gives an illustrating example pointing out the advantage of using the Smooth-Lasso estimator.
	
\end{itemize}

The above remarks recommend large values of $\mu_n$ due to the fact that $\phi_{\mu_n}$ grows with $\mu_n$.
However the RHS of~\eqref{th:DenoizzError} depends on $\mu_n$ also through $\mu_n |\widetilde{J}\beta^*|_2$. Then one may choose the largest $\mu_n$ such that the second term in the RHS of~\eqref{th:DenoizzError} remains reasonable compared to the first one. That is the choice of $\mu_n$ should make a tradeoff between increasing $\phi_{\mu_n}$ and increasing $\mu_n |\widetilde{J}\beta^*|_2$ in the bound.\\
To make things clearer, let us focus on the prediction error (the same reasoning is true for the other errors). The rate of convergence is
$$
\left\{
    \begin{array}{ll}
         \frac{\lambda_n}{\phi_{\mu_n}} |\mathcal{A}^*|  &\quad  \mbox{if }  \mu_n |\widetilde{J}\beta^*|_2 = \mathcal{O}( \lambda_n \sqrt{|\mathcal{A}^*| } ) \mbox{ or even smaller in order,}
\\         
\\
        \frac{\mu_n^2}{\phi_{\mu_n} \lambda_n} |\widetilde{J}\beta^*|_2^2 &\quad \mbox{otherwise.}
    \end{array}
\right.
$$
Then, the term $\mu_n |\widetilde{J}\beta^*|_2 $ induces an alteration on the rate of convergence when $ \mu_n |\widetilde{J}\beta^*|_2 \gg \lambda_n \sqrt{|\mathcal{A}^*| }$.
In other words, the rate of convergence is worse when we add the quadratic penalty unless if $ \mu_n |\widetilde{J}\beta^*|_2 \leq \lambda_n \sqrt{|\mathcal{A}^*| }$.

All these explanations encourage the compromise stated in Corollary~\ref{Cor:BorneL1Sparse} above for the calibration of $\mu_n$.
In the next two paragraphs we provide a more detailed study in the special cases of the Elastic-Net and the S-Lasso estimators.
%
%

\subsubsection{Elastic-Net}
\label{subsec:ElasticNet}

The Elastic-Net corresponds to the case where $\widetilde{J}$ equals the identity matrix. Then $\mathcal{B} = \mathcal{A}^*$ in the above theorem and corollary.
The theoretical performance of the Elastic-Net has already been considered for example in~\cite{BuneaEN,JY_EN_08}.
In~\cite{JY_EN_08}, the authors considered a version of the Irrepresentable Condition to establish their consistency results.
This necessary and (almost) sufficient assumption for the variable selection task is harder to interpret than ours.
The result in the present paper (and particularly those in Section~\ref{subsecGenMatrix}) about the Elastic-Net are quite close to those in~\cite{BuneaEN}. A comparison between the results obtained here and those stated in~\cite{BuneaEN} is postponed to Section~\ref{subsecGenMatrix}.

When compared to the Lasso, we essentially note two differences:
first, as mentioned before Theorem~\ref{ChapSLThm:doubkeSparsJ}, the Lasso brings into play a set of linear inequalities (that is, vectors $\Delta\in\mathbb{R}^p$ such that $ \sum_{j\notin \mathcal{A}^*}  |\Delta_{j}| \leq 4 \, \sum_{j \in \mathcal{A}^*} |\Delta_{j}|$, see for instance~\cite{Lasso3,VandeGeerConditionLasso09}), whereas we need in Theorem~\ref{ChapSLThm:doubkeSparsJ} a bigger set induced by a quadratic set of inequalities (that is, $\Delta$ such that $ \sum_{j\notin \mathcal{A}^*}        \left|\Delta_{j}\right| \leq \varrho_n  \sqrt{ \sum_{j\in\mathcal{A}^*} \Delta_{j}^2 }$ with $\varrho_n > 4\sqrt{|\mathcal{A}^*|}$).
Even though this difference is small, let us mention that we will establish in Section~\ref{subsecGenMatrix} theoretical guaranties which also require the same linear set as in the Lasso case; second, the main difference pertains to the values of $\phi_{\mu_n}$ and $ \phi_0$. Since $\phi_{\mu_n} > \phi_0$, the Elastic-Net is useful in situations that preclude the use of the Lasso because $\phi_0$ is close to zero. This was discussed in Section~\ref{subsec:Mu}. For instance, when the correlations are high between variables, the Lasso fails, whereas the Elastic-Net achieves satisfying performance (see Corollary~\ref{Cor:BorneL1Sparse}).

Finally, we observe that in the case of the Elastic-Net, Equation~\eqref{th:SmoothError} is nothing but a SI on the $\ell_2$ estimation error $|\beta^*-\hat{\beta}^{Quad}|_2^2$.
Note, however, that the rate $ \lambda_n  \sqrt{|\mathcal{A}^*|}$, when $\mu_n$ is defined as in Corollary~\ref{Cor:BorneL1Sparse}, is not optimal (it can be sharper with more restrictive assumptions) but has the advantage of only requiring Assumption~$B(\mathcal{A}^*)$. Imposing Assumption~$B(\mathcal{B})$ with $ \mathcal{B}$ larger $\mathcal{A}^*$, a better rate of convergence can be reached (see Proposition~\ref{ChapSLprop:SupNormm}). We refer to~\cite[Theorem~1]{YeZhang10}) for lower bounds on the $\ell_q$ estimation error of order $|\mathcal{A}^*|^{1/q} \sqrt{\frac{\log(1+p/|\mathcal{A}^*| )}{n}}$. See also~\cite{RWY09,RT10}. 
%
%

\subsubsection{Smooth-Lasso}
\label{subsec:SL}

The S-Lasso corresponds to the case where $\widetilde{J}=\mathbf{J}'\mathbf{J}$ with $\mathbf{J}$ given by~\eqref{ChapSLJmatrix}. This estimator can deal with problems where the regression vector is expected to be $\alpha$-smooth in the sense of Definition~\ref{Def:Smooth}. As a consequence, we have the worst case relation $|\widetilde{J} \beta^{*}|_2 \leq  7 |\mathbf{J} \beta^{*}|_2$ (the constant $7$ comes from some rough computations and is not accurate).
Note also that in this case Assumption~$B(\Theta)$ is satisfied with a set $\Theta = \mathcal{B}$ whose size is less than $3 |\mathcal{A}^*|$. This set can be expressed as
\begin{equation*}
	\mathcal{B} = \{ j\in \{ 2,\ldots,p-1\} : \,\,\beta^*_j \neq 0,  \quad \beta^*_{j-1} \neq 0 \quad or \quad \beta^*_{j+1} \neq 0\},
\end{equation*}
and Theorem~\ref{ChapSLThm:doubkeSparsJ} holds with $c_{\widetilde{J} }= 3$. Moreover, Equation~\eqref{th:SmoothError} can be seen as a control on the `smoothness error' $ \sum_{j=2}^p (\delta_j - \delta_{j-1})^2$, where $ \delta_j$ is the components difference $\beta_j^{*} - \hat{\beta}_j^{Quad}$.

The S-Lasso is designed to provide a smooth and sparse solution.
This is true whatever the correlations between variables.
However, it is interesting to remark that the smoothness has quite close interactions with correlations between successive variables.
Indeed, when we deal with the S-Lasso estimator, the matrix $\widetilde{J}$ is tridiagonal with its off-diagonal terms equal to $\text{-}1$.
If we do not consider the diagonal terms, we remark that $\Psi^n$ and $K_n$ differ only in the terms on the second diagonals (that is, $(K_n)_{j-1,j} \neq (\Psi^n)_{j-1,j}$ for $j=2,\ldots,p$ as soon as $\mu_n \neq 0$).
Terms in the second diagonals of $\Psi^n$ correspond to correlations between successive covariates.

When high correlations exist between successive covariates, a suitable choice of $\mu_n$ fulfills Assumption~$B(\mathcal{B})$.
Hence, the S-Lasso estimator is particularly useful in situations where we expect that the variables are ranked, such that not only the regression vector is `smooth', but also successive covariates are highly correlated.
Indeed, on the one hand Assumption $B(\mathcal{B})$ is a weaker assumption for `smooth' regression vector. On the other hand, this `smoothness' makes the prediction and the estimation errors sharper (as $\phi_{\mu_n}$ depends on $|\mathbf{J}\beta^*|_2$).

\vspace{2mm}

In the next paragraph, we present an illustrating example of Corollary~\ref{Cor:BorneL1Sparse} (or Theorem~\ref{ChapSLThm:doubkeSparsJ}) where we show the importance of using the Smooth-Lasso in certain situations where the Lasso and the Elastic-Net do not provide good control on the different errors.
In particular, we present a case where correlations between variables exist (and where the Lasso is not suitable). Moreover, since the influence of the quadratic penalty in the definition of $\hat{\beta}^{Quad}$ reduces when $|\widetilde{J}\beta^*|_2$ is large (see the definition of $\mu_n$ in Corollary~\ref{Cor:BorneL1Sparse}), we consider a smooth regression vector with large singular coefficient values such that $|\widetilde{J}\beta^*|_2$ is small when $\widetilde{J}$ is the matrix corresponding to the Smooth-Lasso, and large when $\widetilde{J}$ is the identity matrix associated to the Elastic-Net. 
Due to this difference on the value of $|\widetilde{J}\beta^*|_2$, the Smooth-Lasso outperforms the Elastic-Net.

\vspace{5mm}

\noindent {\bf Example.}
Let $\widetilde{J}$ be the matrix defined on~\eqref{ChapSLJmatrix}.
Assume that $n/4$ is an integer.
First of all, let us define a smooth regression vector $\beta^*$ with $n/2$ non-zero components such that
\begin{equation*}
\beta_j^* = 1 \quad \text{for}\,\, j=1,\ldots ,n/4 - 1,
\quad \quad \text{and}\quad
\beta_j^* = 1- \frac{4}{n}\left(j-\frac{n}{4}\right) \quad \text{for}\,\, j=n/4,\ldots ,n/2.
\end{equation*}
This regression vector is chosen piecewise linear (a particular case of smoothness) to clarify the idea and for simplicity of computations. The vector $\beta^{*}$ is such that
\begin{equation*}
|\beta^{*}|_2  = \sqrt{\frac{n}{3} - \frac{1}{2} + \frac{2}{3n}} = \mathcal{O}(\sqrt{n}),
\quad   \quad \text{and}\quad
|\mathbf{J} \beta^{*}|_2  = \sqrt{\frac{4}{n} - \frac{16}{n^2}} = \mathcal{O}(1/\sqrt{n}).
\end{equation*}
Then, we can set the smoothness parameter $\alpha = 4/\sqrt{n}$ in Definition~\ref{Def:Smooth}.

Let us now consider the design matrix $\Psi^n$.
Let $\epsilon>0$ be a real number.
Let $\Psi^n$ be a tridiagonal Gram matrix with diagonal elements equal to $1$ (that is, normalized) and such that $\Psi_{j,j-1}^n = \Psi_{k,k+1}^n = \epsilon$ for $j=2,\ldots, p$ and $k=1,\ldots,p-1$.
In such a case, the spectrum of the Gram matrix lies in $[1-2 \epsilon, 1+2\epsilon]$.
Then, $\phi_0 \ge 1-2\epsilon$ (the $\phi_{\mu_n}$ corresponding to the Lasso estimate, that is, when $\mu_n = 0$).
However, we do not know how far $\phi_0$ is from $1-2\epsilon$ so that we can only say the the prediction error of the Lasso $\hat\beta^L$ is such that with high probability
\begin{equation*}
	\left\|X \beta^* -X \hat\beta^L \right\|_{n}^{2}
	\leq
	\frac{16\sqrt{2} \sigma^2}{ 1-2\epsilon } \frac{\log(p/\eta)}{n} |\mathcal{A}^*|
	=
	\mathcal{O}(\sigma^2 |\mathcal{A}^*|),
\end{equation*}
with the choice $\epsilon = \frac{1}{2} - \frac{\log(p/\eta)}{2n}$. Actually, the above bound does not provide any control on the prediction error of the Lasso estimator.
\\
Let us now focus on the Elastic-Net estimate $\hat\beta^{EN}$.
According to Assumption~$B(\mathcal{A}^*)$, we have to consider the spectrum of the matrix $K_n^{EN} = \Psi^n + \mu_n I_p$, where $I_p$ is the identity matrix in $\mathbb{R}^p$.
This spectrum lies in $[1-2 \epsilon + \mu_n, 1+2\epsilon + \mu_n]$.
Given the values of $\epsilon$ and of $|\beta^*|_2$, we get the control
\begin{equation*}
	\left\|X \beta^* -X \hat{\beta}^{EN} \right\|_{n}^{2}
	\leq
	\frac{1}{1-2 \epsilon + \mu_n} ( 2 \lambda_n \sqrt{|\mathcal{A}^*| }  + 2 \mu_n |\beta^*|_2 )^2
	=
	\mathcal{O}(\sigma \min \{ \sqrt{\log (p) |\mathcal{A}^*| }, |\mathcal{A}^*|\}  ),
\end{equation*}
where we used the definition of $\mu_n$ provided in Corollary~\ref{Cor:BorneL1Sparse}.
Let us mention that choosing a different value for $\mu_n$ does not imply an improvement in the bound.
Hence, in this case the Elastic-Net estimator does not control the prediction error neither.\\
Next, in the case of the S-Lasso $\hat{\beta}^{SL} $ the eigenvalues of the matrix $K_n^{SL} = \Psi^n + \mu_n \widetilde{J}$ lie in $[1+\mu_n - 2|\epsilon - \mu_n|,1+ 2\mu_n + 2|\epsilon - \mu_n|]$.
We refer to~\cite{YuehValeurPropre} for more details on the eigenvalues of tridiagonal matrices.
This interval is of the same order as the one of the Elastic-Net.
By the sequel, we have the following control for the S-Lasso estimator (when $\epsilon>\mu_n$, otherwise the control is even better)
\begin{equation*}
	\left\|X \beta^* -X \hat{\beta}^{SL} \right\|_{n}^{2}
	\leq
	\frac{1}{1-2 \epsilon + 3\mu_n}  ( 2 \lambda_n \sqrt{|\mathcal{A}^*| }  + 2 \mu_n |\widetilde{J}\beta^*|_2 )^2
	 =
	\mathcal{O}\left(\sigma \frac{\sqrt{\log (p) |\mathcal{A}^*| }}{n}\right),
\end{equation*}
where here again, we considered the value of $\mu_n$ given in Corollary~\ref{Cor:BorneL1Sparse}.
In this `smooth context', the S-Lasso is obviously the best method (compared with the Lasso and the Elastic-Net).
Note that this last rate is better than the minimax rate under the sparsity assumption $\frac{\log (p/|\mathcal{A}^*| + 1) |\mathcal{A}^*| }{n}$ ({\it cf.} \cite[Theorem 5.1]{BTWAggSOI}).
This is due to the fact that we also imposed a smoothness assumption which is nicely exploited by the S-Lasso estimator.
Thus, the above minimax rates cannot be applied anymore.

Let us conclude with the following remarks: in the above situation, we assume that the regression vector is smooth also that the successive covariates are correlated.
This is the best context for the Smooth-Lasso.\\
In the case where the regression vector is smooth, but we do not have a particular structure in the Gram matrix (say the variables are independent and $\phi_0$ is a fixed positive constant), the Lasso and the Elastic-Net (for instance with the value of $\mu_n$ given in Corollary~\ref{Cor:BorneL1Sparse}) reach the rate $\sigma^2\frac{\log (p) |\mathcal{A}^*| }{n}$.
Compared to the bounds for the Elastic-Net, there are improved bounds for the S-Lasso and for suitable values of $\mu_n$ (note that $\mu_n$ depends on $\alpha$).
Here again, if we consider the same regression vector $\beta^*$ as in the above example, the rate is of order $ \mathcal{O}\left(\sigma \frac{\sqrt{\log (p) |\mathcal{A}^*| }}{n}\right)$.
Consequently, we get better performance than the Elastic-Net and the Lasso.\\
Finally, when the regression vector is not smooth (say, $|\beta^*|_2$ and $|\mathbf{J}\beta^*|_2$ are constants) and the design matrix is for instance as in the above example, the Lasso is not suitable.
In this case, both the Elastic-Net and the S-Lasso have comparable performance and their bound is in order $\mathcal{O}(\sqrt{\log (p) |\mathcal{A}^*| /n})$, which is much better than the bounds for the Lasso (even if not optimal).

The above discussion dealt with the prediction and the estimation performance. In the next section we consider the variable selection power of $\hat{\beta}^{Quad}$.

\subsection{Variable selection}
\label{ChapSLsec:supNorm}

Let us first mention that the estimator $\hat{\beta}^{Quad}$, with the Smooth-Lasso as a particular case, has not been introduced for such an objective.
Indeed, it is designed to deal with the estimation criterion or, more precisely, with structural questions.
However, in some problems $\hat{\beta}^{Quad}$ may induce better variable selection properties than the Lasso.

A large amount of work has been done on the topic of variable selection for Lasso-type methods.
One important observation is that one has to make a compromise between not identifying a low signal level (that is, small coefficients $\beta_j^*,\,j\in\mathcal{A}^*$, in absolute value) and imposing a strong restriction on the Gram matrix $\Psi^n$ which sometimes seems to be not realistic.
Moreover, the question of the identifiability of $\beta^*$ has also to be considered.
Since we tackle problems where we expect correlations between variables, we take the middle path, that is, we impose less restrictive assumptions on the Gram matrix that permit us to recover a reasonably low signal level.
For this purpose, we first provide a bound on the sup-norm $|\beta^{*} - \hat{\beta}^{Quad}|_{\infty}$, based on a control on the $\ell_2$ estimation error.

To this end, we use Assumption~$B(\Theta)$ on the Gram matrix.
However the set $\Theta$ should be larger than the one required in Theorem~\ref{ChapSLThm:doubkeSparsJ}.
To define it, let us denote by $\mathcal{C}$ the index-set of the $m$ largest components in absolute value of $\beta^{*} - \hat{\beta}^{Quad}$ outside $\mathcal{B}$.
Here $\mathcal{B}$ is the set introduced in Theorem~\ref{ChapSLThm:doubkeSparsJ}.
In this setting $m$ is an integer such that $m + |\mathcal{B}| < p$.
\begin{description}
	\item[Assumption~$B'(\mathcal{B} \cup  \mathcal{C})$:] {\it Let $K_n$ be the matrix given by~\eqref{eq:matrixKn} and let $ \varrho_n = 4 \sqrt{|\mathcal{A}^*|} +  \frac{4\mu_n}{\lambda_n} | \widetilde{J}\beta^*|_2 $.
	There is a constant $\phi_{\mu_n} > 0$ such that, for any $\Delta\in\mathbb{R}^{p}$ that satisfies
$ \sum_{j\notin \mathcal{B}} \left|\Delta_{j}\right| \leq  \varrho_n 
\sqrt{ \sum_{j\in\mathcal{B}} \Delta_{j}^2 }$, we have
	\begin{equation}
	\label{eq:CondHypEase2}
		\Delta' K_n \Delta 
		\geq
		\phi_{\mu_n} \sum_{j\in \mathcal{B} \cup  \mathcal{C}} \Delta_{j}^{2} .
	\end{equation}
}
\end{description}
The above assumption differs from Assumption~$B(\Theta)$ only in that we restrict $\mathbb{R}^{p}$ in a different set than the one used in Condition~\eqref{eq:CondHypEase2}.
Obviously, Assumption~$B'(\mathcal{B} \cup  \mathcal{C})$ implies Assumption~$B(\mathcal{B})$.
\begin{prp}
\label{ChapSLprop:SupNormm}
	Let us consider the same setting as in Theorem~\ref{ChapSLThm:doubkeSparsJ} with the only difference that $\lambda_n = 2\sqrt{2} \sigma \sqrt{\frac{\log (p/\eta)}{ n}}$ with $0 <\eta < 1 $.
	Under Assumption~$B'(\mathcal{B} \cup  \mathcal{C})$ and with probability $1 - \eta$
	\begin{equation*}
	|\hat{\beta}^{Quad} - \beta^*|_{\infty}
	\leq
	|\hat{\beta}^{Quad} - \beta^*|_2
	\leq 
	 \tilde{c}  ( \lambda_n \sqrt{|\mathcal{A}^*| } 
	+ \mu_n |\widetilde{J}\beta^*|_2 ),
	\end{equation*}
where $\tilde{c} = 2 \phi_{\mu_n}^{-1} (1+ \frac{\varrho_n}{\sqrt{m}} )$.
\end{prp}
One can exploit the control provided in Proposition~\ref{ChapSLprop:SupNormm} to construct a hard-thresholded version of $\hat{\beta}^{Quad}$ which is consistent in variable selection.
Such a construction has already been considered is several papers for the Lasso estimate.
The methodology closest to ours is the one developed in~\cite{MeinYuSelect}.

Consider $\hat{\beta}^{Th-Quad} = (\hat{\beta}_1^{Th-Quad},\ldots,\hat{\beta}_p^{Th-Quad})'$, the thresholded $\hat{\beta}^{Quad}$ estimator defined by
$$
\hat{\beta}_j^{Th-Quad} 
=
\hat{\beta}_j^{Quad} \quad \quad \text{if} \quad |\hat{\beta}_j^{Quad} |\geq \tilde{c} ( \lambda_n \sqrt{|\mathcal{A}^*| } 
	+ \mu_n |\widetilde{J}\beta^*|_2 )
$$
and zero otherwise, where $\tilde{c}$ is given in Proposition~\ref{ChapSLprop:SupNormm}.
This estimator consists of $\hat{\beta}^{Quad}$ with its small coefficients reduced to zero.
We then enforce the selection property of $\hat{\beta}^{Quad}$.
Variable selection consistency of this estimator is established under one more restriction on the regression vector given now.
\begin{description}
	\item[Assumption $C$:] {\it The true regression vector $\beta^*$ is such that 
	\begin{equation*}
		\min_{j\in \mathcal{A}^*} |\beta_j^*|  > 2 \tilde{c}  ( \lambda_n \sqrt{|\mathcal{A}^*| } 
	+ \mu_n |\widetilde{J}\beta^*|_2 ),
	\end{equation*}
where $ \tilde{c} =2 \phi_{\mu_n}^{-1} (1+ \frac{\varrho_n}{\sqrt{m}} )$ is from Proposition~\ref{ChapSLprop:SupNormm}, and $\phi_{\mu_n}$ is the term appearing in Assumption~$B'(\mathcal{B}\cup \mathcal{C})$.
}
\end{description}
Here again, we observe how important the quantity $\phi_{\mu_n}$ is.
We want it to be as large as possible.

This assumption bounds from below the smallest regression coefficient in $\beta^*$.
This is a common assumption to provide sign consistency in the high dimensional case.
See for example \cite{Bunea_consist,KarimNormSup,MeinYuSelect,WainSelection,ZhangHuangConsist,BiYuConsistLasso}.
We refer to~\cite{KarimNormSup} for a longer discussion on how these works are related in terms of restrictions related to the threshold or the assumption on the Gram matrix. 
Now, we can state the following sign consistency result.
\begin{thm}
\label{ChapSLthm:VarSelec}
	Let us consider the thresholded estimator $\hat{\beta}^{Th-Quad}$ as described above.
	In the same setting as in Proposition~\ref{ChapSLprop:SupNormm}, and under Assumption~$B'(\mathcal{B} \cup  \mathcal{C})$ and Assumption~$C$
	\begin{equation*}
		\mathbb{P}\left(\mathop{\rm sign}(\hat{\beta}^{Th-Quad}) = \mathop{\rm sign}(\beta^{*}) \right) \geq 1- \eta.
	\end{equation*}
\end{thm}
Note that all the remarks established in Sections~\ref{subsec:ElasticNet} and~\ref{subsec:SL} remain valid also for this variable selection result.

\subsection{Technical advances}
\label{sec:Techn}

We devote this paragraph to several technical considerations.
First, we consider the case of a general matrix $\widetilde{J}$.
Then, we establish the variable selection consistency of a non-thresholded version of $\hat\beta^{Quad}$.
Finally, we provide a relaxation of the assumption on the noise.
The reader who is not interested in these studies can skip them without consequences for the readability of the paper.
%

\subsubsection{General matrices $\widetilde{J}$}
\label{subsecGenMatrix}

Theorem~\ref{ChapSLThm:doubkeSparsJ} is particularly interesting when $\widetilde{J} = \mathbf{J}'\mathbf{J}$ is sparse. In that statement, Assumption~$B(\mathcal{B})$ was needed with a set $ \mathcal{B} \supset \mathcal{A}^*$ which depends on $\widetilde{J}$.
More precisely, $\mathcal{B}$ contains the indices of components which interfere in the sparse product $\beta^{*'} \widetilde{J} u$ for a given $u\in\mathbb{R}^p$ (see the proof for more details).
This set is not too large compared to $\mathcal{A}^*$ when we consider the case where $\widetilde{J}$ is sparse.
This way to solve the problem allows us to choose $\mu_n \sim \lambda_n \frac{\sqrt{|\mathcal{A}^*|}}{|\widetilde{J}\beta^*|_2}$ ({\it cf.} Corollary~\ref{Cor:BorneL1Sparse}). In what follows, we consider $p \times p$ matrices $\widetilde{J}$ (including the sparse case) for which we only need an (adapted) RE Assumption. Contrary to the results provided in Section~\ref{ChapSLsec:SOI}, $\mu_n$ is here, for technical reasons, not a free parameter anymore and is fixed in advance (see~\eqref{ChapSLeq:weight1} below). This value is smaller than the one given in Corollary~\ref{Cor:BorneL1Sparse}.

Let us first establish the assumptions needed and the setup of this contribution. Let $\eta \in (0,1)$. We define the regularization parameters $\lambda_n$ and $\mu_n$ in the following way:
\begin{eqnarray}
\label{ChapSLeq:weight1}
	\lambda_{n} = 8\sqrt{2} \sigma \sqrt{\frac{\log(p/\eta)}{n}} \quad \text{and} \quad \mu_{n} = \lambda_{n} \frac{1}{ 8 |\widetilde{J}\beta^*|_{\infty} }.
\end{eqnarray}
We now state the adapted RE Assumption which differs from the usual one introduced in~\cite{Lasso3} only by the matrix to which we apply the assumption ($K_n$ instead of $\Psi^n$):
\begin{description}
	\item[Assumption~$RE$:] {\it There is a constant $\phi_{\mu_n} > 0$ such that, for any $\Delta\in\mathbb{R}^{p}$ that satisfies
$ \sum_{j\notin \mathcal{A}^*} \left|\Delta_{j}\right| \leq 4  \sum_{j\in  \mathcal{A}^*} |\Delta_{j}| $, we have
	\begin{equation*}
	\label{eq:CondHypEaseRE}
		\Delta' K_n \Delta 
		\geq
		\phi_{\mu_n} \sum_{j\in  \mathcal{A}^*} \Delta_{j}^{2} .
	\end{equation*}
}
\end{description}
This assumption involves a set of linear inequalities. Then, we clearly have $\phi_{\mu_n} \ge \phi_0$ (the $\phi_{\mu_n}$ corresponding to the Lasso, that is, when $\mu_n = 0$). With this setting, we obtain the following result for a general matrix $\widetilde{J}$.
\begin{thm}[General $\widetilde{J}$]
\label{ChapSLThm:doubke}
	Let $\mathcal{A}^*$ be the sparsity set and let the tuning parameters $(\lambda_n,\mu_n)$ be defined as in~\eqref{ChapSLeq:weight1}. If Assumption~$RE$ holds, then with probability greater than $1- \eta$ we have
	\begin{equation*}
		\left\|X \beta^* -X \hat{\beta}^{Quad} \right\|_{n}^{2}
		\leq
		4\phi_{\mu_n}^{-1} \lambda_n^2  |\mathcal{A}^*| ,
	\end{equation*}
	\begin{equation*}
		(\beta^* - \hat{\beta}^{Quad})' \widetilde{J} (\beta^* - \hat{\beta}^{Quad})
		\leq
		4 \frac{| \widetilde{J}\beta^*|_{\infty}}{ \phi_{\mu_n}} \lambda_n  |\mathcal{A}^*| ,
	\end{equation*}
	and
	\begin{equation*}
		|\beta^* - \hat{\beta}^{Quad}|_1
		\leq
		8\phi_{\mu_n}^{-1}  \lambda_n  |\mathcal{A}^*|.
	\end{equation*}
\end{thm}
Similar bounds were provided for the Lasso estimator in \cite{Lasso3}. Let us mention that the constants are not optimal. We focused our attention on the dependency on $n$ (and thus on $p$ and $|\mathcal{A}^*|$). It turns out that our results are near optimal. For instance, for the $\ell_2$ risk, the S-Lasso estimator reaches nearly the optimal rate $\frac{|\mathcal{A}^*|}{n}\, \log (\frac{p}{|\mathcal{A}^*|}+1)$ up to a logarithmic factor (see \cite[Theorem 5.1]{BTWAggSOI}). Moreover, Theorem~\ref{ChapSLThm:doubke} states a control on an error which is linked to the expected prior information which suggested the use of the estimator $\hat{\beta}^{Quad}$.

The results provided in Theorem~\ref{ChapSLThm:doubkeSparsJ} and more precisely Corollary~\ref{Cor:BorneL1Sparse}, differ from those established in Theorem~\ref{ChapSLThm:doubke} in a few points. First, the value of $\mu_n$ is larger in the sparse case. Indeed, $\mu_n$ equals $\lambda_{n} \frac{\sqrt{\mathcal{A}^*|}}{ 2 |\widetilde{J}\beta^*|_{2} }$ and $\lambda_{n} \frac{1}{ 4 |\widetilde{J}\beta^*|_{\infty} }$ in Corollary~\ref{Cor:BorneL1Sparse} and Theorem~\ref{ChapSLThm:doubke} respectively. The former value can be much larger for some regression vector $\beta^*$. Second, these values of $\mu_n$ have an influence on the error bounds through $\phi_{\mu_n}$. As a consequence, the bounds in Corollary~\ref{Cor:BorneL1Sparse} are better than those in Theorem~\ref{ChapSLThm:doubke}. Finally, apart from the considerations on the quantity $\phi_{\mu_n}$, we observe a modification of the bound of $(\beta^* - \hat{\beta}^{Quad})' \widetilde{J} (\beta^* - \hat{\beta}^{Quad})$. Indeed, in Theorem~\ref{ChapSLThm:doubkeSparsJ}, it involves the term $| \widetilde{J}\beta^*|_{2} \sqrt{|\mathcal{A}^*|}$, whereas in Theorem~\ref{ChapSLThm:doubke}, $| \widetilde{J}\beta^*|_{\infty} |\mathcal{A}^*|$ appears, which is obviously larger. We then have a better control on this error using the sparsity of the matrix $\widetilde{J}$. Finally, we remark that the constant factor in the definition of the tuning parameter $\lambda_n$ in Corollary~\ref{Cor:BorneL1Sparse} is smaller than the corresponding constant in Theorem~\ref{ChapSLThm:doubke}. 
One should however mention that for a fixed $\phi_{\mu_n}$ (that is a fixed $\mu_n$), the set of feasible vectors $\Delta$ in Assumption~RE is larger than the one in Assumption~$B(\mathcal{B})$. In this sense, Assumption~RE is less restrictive than Assumption~$B(\mathcal{B})$. Nevertheless, this difference does not clearly mean that the $\phi_{\mu_n}$ resulting from the Assumption~RE is larger than the one arising from Assumption~$B(\mathcal{B})$. Indeed, when $\Delta$ is in the feasible set of both assumptions, $\phi_{\mu_n}$ is the same in both conditions.

\vspace{2mm}

A close result to Theorem~\ref{ChapSLThm:doubke} has been established by Bunea in~\cite{BuneaEN} in the particular case of the Elastic-Net. It is worth briefly pointing out here the differences and the similarities of our work and~\cite{BuneaEN} when we deal with the Elastic-Net. For any vector $b\in\mathbb{R}^p$ and subset $\Theta \subset \{1,\ldots,p\}$, let $b_{\Theta}$ be the vector in $\mathbb{R}^p$ such that $(b_{\Theta})_j = b_j$ if $j\in\Theta$ and zero otherwise.
In~\cite{BuneaEN}, Bunea provided a SI close to the one established in Theorem~\ref{ChapSLThm:doubke}. This inequality holds under the {\it Condition Stabil} defined in~\cite[page~4]{BuneaEN} by
 \begin{equation*}
	\Delta'\Psi^n \Delta \ge (\phi_{\mu_n}^{CS} - \mu_n) |\Delta_{\mathcal{A}^*}|_2^2 ,
 \end{equation*}
where $\phi_{\mu_n}^{CS} > \mu_n$, and similarly to vectors in Assumption~$RE$, $\Delta$ is such that
$ \sum_{j\notin \mathcal{A}^*} \left|\Delta_{j}\right| \leq 4  \sum_{j\in  \mathcal{A}^*} |\Delta_{j}| $.
The above equation is the analogous of the condition~\eqref{eq:CondHypEase} in Assumption~$RE$, and to make the comparison easier, let us write~\eqref{eq:CondHypEase} as follows
 \begin{equation}
\label{eq:ELasticNetCondit}
	\Delta'\Psi^n \Delta \ge (\phi_{\mu_n} - \mu_n) |\Delta_{\mathcal{A}^*}|_2^2 - \mu_n |\Delta_{(\mathcal{A}^*)^c}|_2^2.
 \end{equation}
Since the bounds in the Sparsity Inequalities stated in~\cite{BuneaEN} and in the present paper are up to constants the same, it seems that the only difference is the value of $\phi_{\mu_n}$. Indeed, according to Inquality~\eqref{eq:ELasticNetCondit}, $\phi_{\mu_n} $ can be much larger than $\phi_{\mu_n}^{CS}$ (given in {\it Condition Stabil}), as we subtract the term $\mu_n |\Delta_{(\mathcal{A}^*)^c}|_2^2$ in~\eqref{eq:ELasticNetCondit}, which can be large thanks to $\mu_n$ (we expect however $|\Delta_{(\mathcal{A}^*)^c}|_2^2$ to be small).
It is worth adding that the Elastic-Net corresponds to a case where the matrix $\widetilde{J}$ is sparse (as $\widetilde{J}$ is the identity). Therefore, it is more convenient to use the setting of Section~\ref{ChapSLsec:SOI} since the value of $\mu_n$ is larger there.

\subsubsection{Non-thresholded variable selector}

In Section~\ref{ChapSLsec:supNorm}, we established variable selection consistency for a thresholded version of $\hat{\beta}^{Quad}$ when $\widetilde{J}$ is sparse. In this section, we state a comparable result for a non-thresholded version. Indeed, paying the price of a more restrictive assumption, we provide in Theorem~\ref{ChapSLprp:ConsistAvecSign} below a variable selection consistency result directly for $\hat{\beta}^{Quad}$ when using a different calibration of the tuning parameters.
This result can be applied to general matrices $\widetilde{J}$.
The approach to prove the result is also different.
We first provide a bound on the sup-norm $|\beta_{\mathcal{A}^*}^{*} - \hat{\beta}_{\mathcal{A}^*}^{Quad}|_{\infty}$.
This can be done easily using the theorem stated in Section~\ref{subsecGenMatrix} for the $\ell_1$ estimation error $|\beta^{*} - \hat{\beta}^{Quad}|_{1}$.
However, this would imply that only `high' levels of the signal can be reconstituted, that is, coefficients $\beta_j^*,\,j\in\mathcal{A}^*$ such that $|\beta_j^*| \geq cst\cdot \lambda_n |\mathcal{A}^* |$.
Therefore, we favor to exploit here again a control on the $\ell_2$ estimation error $|\beta^{*} - \hat{\beta}^{Quad}|_{2}$ instead, which in the sequel enables us to recover signals with $|\beta_j^*| \geq cst\cdot \lambda_n \sqrt{|\mathcal{A}^* |}$ with the same assumption on the matrix $K_n$.
Let us mention that $\lambda_n \sqrt{|\mathcal{A}^* |} $ is not the best level which can be recovered.
One can also get rid of the extra term $\sqrt{|\mathcal{A}^* |}$ through a quite restrictive assumption on the correlations between variables (see Remark~\ref{RqBuneaVS}).\\
Proposition~\ref{ChapSLprp:supNormParl2} below is a first step to a variable selection result. It states that $\hat{\beta}^{Quad}$ enables us at least to detect the relevant variables (and maybe also some noise variables):
\begin{prp}
\label{ChapSLprp:supNormParl2}
	Let us consider the same setting as in Theorem~\ref{ChapSLThm:doubke} with the only difference that $\lambda_n = 4\sqrt{2} \sigma\sqrt{\log (p/\eta)/n}$ and $\mu_n = \lambda_n/(4 |\widetilde{J}\beta^{*}|_{\infty})$ with $0<\eta<1$. Under Assumption~$RE$, and with probability larger than $1- \eta $, we have
	\begin{eqnarray*}
		| \beta_{\mathcal{A}^*}^{*} - \hat{\beta}_{\mathcal{A}^*}^{Quad} |_{\infty}
		\leq
		|\beta_{\mathcal{A}^*}^{*} - \hat{\beta}_{\mathcal{A}^*}^{Quad}|_{2}
		\leq
		2\phi_{\mu_n}^{-1} \lambda_n \sqrt{|\mathcal{A}^*|},
	\end{eqnarray*}
	where $\phi_{\mu_n}$ is the constant appearing in Assumption~$RE$. Moreover, if $\displaystyle{\min_{j\in\mathcal{A}^*} |\beta^*_j|  > 2\phi_{\mu_n}^{-1} \lambda_n \sqrt{|\mathcal{A}^*|} }$, we have
	\begin{equation*}
		\mathbb{P}\left(\mathop{\rm sign}(\hat{\beta}_{\mathcal{A}^*}^{Quad}) = \mathop{\rm sign}(\beta_{\mathcal{A}^*}^*)\right)
		\geq 
		1- \eta.
	\end{equation*}
\end{prp}
Proposition~\ref{ChapSLprp:supNormParl2} is a trivial consequence of Theorem~\ref{ChapSLThm:doubke}. A short proof is given in the Appendix section. This proposition emphasizes directly that under Assumption~$RE$ all non-zero components of $\beta^*$ are detected by $\hat{\beta}^{Quad}$ with high probability. Actually, in the setting of Proposition~\ref{ChapSLprp:supNormParl2}, $\hat{\beta}^{Quad}$ may contain too many non-zero components. More restrictions are needed in order to ensure the variable selection consistency of $\hat{\beta}^{Quad}$. Here is an additional assumption on the Gram matrix which controls the correlations between the truly relevant variables and those which are not.
\begin{description}
	\item[Assumption $D$:] {\it We assume that
	\begin{equation*}
	\max_{ j \in \mathcal{A}^* } \max_{ k \notin \mathcal{A}^*} |(K_n)_{j,k}| 
	\leq 
	\frac{t}{|\mathcal{A^*}|},
	\end{equation*}
where $t$ is a positive term smaller than $\frac{\phi_{\mu_n}}{64}$.}
\end{description}
This assumption is quite close to the Mutual Coherence assumption which involves the Gram matrix $\Psi^n$ instead of $ K_n$. In addition, the Mutual Coherence assumption restricts correlations between all covariates.
\begin{thm}
\label{ChapSLprp:ConsistAvecSign}
	Let consider the linear regression model~\eqref{ChapSLeq_depart}. Let $\lambda_n = 16 \sigma \sqrt{ \frac{\log ( p / \sqrt{\eta p / (1+p)} )}{n} }$ and $\mu_n = \lambda_n/(4 |\widetilde{J}\beta^{*}|_{\infty})$. Under Assumptions~$RE$ and $C$, and also Assumption~$D$, we have
	\begin{equation*}
		\mathbb{P}(\hat{\mathcal{A}} \nsubseteq \mathcal{A}^*)
		\leq
		\eta,
	\end{equation*}
	and
	\begin{equation*}
		\mathbb{P}\left(\mathop{\rm sign}(\hat{\beta}^{Quad}) = \mathop{\rm sign}(\beta^*)\right)
		\geq 
		1 - \eta.
	\end{equation*}
\end{thm}
To prove the first claim, we use some arguments from~\cite{BuneaEN}.
The second point is a consequence of the first one and of Proposition~\ref{ChapSLprp:supNormParl2}. There are essentially two differences between the settings in Theorem~\ref{ChapSLprp:ConsistAvecSign} and Proposition~\ref{ChapSLprp:supNormParl2}. First, we need for this last result a more restrictive assumption on the correlations between variables. However, this restriction is only between relevant variables and irrelevant covariates. This is `quite' a reasonable assumption to identify the relevant variables, that is, the non-zero components of the vector $\beta^*$. Second, the minimal value of $\lambda_n$ is larger in this last theorem. This suggests that we need a larger value of this tuning parameter to set to zero the irrelevant components. Note that we established the variable selection consistency of $\hat{\beta}^{Quad}$ but with a value of the tuning parameter $\mu_n$ smaller than the one used in the thresholded version.
\begin{rmq}\label{RqBuneaVS}
	The results of Theorem~\ref{ChapSLprp:ConsistAvecSign} can also be obtained under the more restrictive Mutual Coherence assumption: $\max_{ j \in \mathcal{A}^* } \max_{
	\underset{k \not = j}{k \in \{1,\ldots,p \} }} |(K_n)_{j,k}| 
	\leq 
	\frac{\tilde{t}}{|\mathcal{A^*}|}$, where $\tilde{t}$ is a small positive constant. Here, even the correlations between relevant variables are restricted but this restriction makes possible to recover even smaller signal. That is, we can detect coefficients of $\beta^*$ such that $|\beta_j^*| \geq cst\cdot \sqrt{\log (p) /n}$. See for instance~\cite{BuneaEN} in case of the Elastic-Net.
\end{rmq}

\subsubsection{Non Gaussian noise with finite variance}

Most of the results established for Lasso-type methods assume Gaussian or sub-Gaussian type noise~\cite{Lasso3,BuneaEN,JY_EN_08,WainSelection,ZhangHuangConsist}.
Noise with exponential moment is studied in~\cite{Bunea_consist,MeinYuSelect}.
Only a few references consider other type of noise.
Noise with moment of order $2k$, where $k \geq 1$ is an integer, is considered in~\cite{BiYuConsistLasso}, whereas in the paper~\cite{KarimNormSup}, the author presents the case where the noise admits zero mean and finite variance.
It is in the same spirit as that in this last reference that we consider this relaxation on the noise.
According to the Elastic-Net, noise with moment of order $2k + \delta$, where $k \geq 1$ is an integer and $\delta$ is a positive constant is considered in~\cite{ZZ_AdaptEN_09}, but the authors treated only the case where $p = \mathcal{O}(n)$. 

We assume that the noise random variables $\varepsilon_1,\ldots,\varepsilon_n$ are independent and admit zero mean and finite variance.
That is $\mathbb{E} \varepsilon_i = 0$ and $\mathbb{E} \varepsilon_i^2\leq \sigma^2$ for $i=1,\ldots,n$ with $\sigma^2 < \infty$.
In this generalization we also use a revisited version of Nemirovski's Inequality established in~\cite{DvGVW10nemirovkiInequality}.
One more restriction is needed on the sample points.
\begin{description}
	\item[Assumption $E$:] {\it There exists a positive constant $L < \infty$ such that
	\begin{equation*}
	n^{-1}\sum_{i=1}^n \max_{j=1,\ldots,p} x_{i,j}^2 \leq L.
	\end{equation*}
}
\end{description}
Theorem~\ref{ChapSLthm:BruitNonGaussian} below extends the results in Corollary~\ref{Cor:BorneL1Sparse} of Section~\ref{ChapSLsec:theoryPgrd} to the non-Gaussian noise case.
However, one is able to generalize all the results of that section in the same way.
\begin{thm}
\label{ChapSLthm:BruitNonGaussian}
	Let consider the linear regression model~\eqref{ChapSLeq_depart} where the $\varepsilon_i$'s are independent random variables with zero mean such that $\mathbb{E} \varepsilon_i^2\leq \sigma^2$ for $i=1,\ldots,n$ with $\sigma^2 < \infty$.
	Denote by $K_{Nem}$ the quantity $K_{Nem} = \inf_{q \in [2,\infty) } (q-1) p^{2/q}$, and let $\lambda_n= 4\sigma  \sqrt{\frac{K_{Nem} L}{n\eta} }$ with $ 0<\eta <1$.
	Let $\mu_n = \frac{\lambda_{n} \sqrt{|\mathcal{A}^*|}}{2 | \widetilde{J}\beta^*|_2 }$.
	Assume also that Assumption~$B(\mathcal{B})$ (where $\varrho_n = 6\sqrt{|\mathcal{A}^*|}$) and Assumption~$E$ hold.
	Then, with probability greater than $1 - \eta $ we have
	\begin{equation*}
		|\beta^* - \hat{\beta}^{Quad}|_1
		\leq
		\frac{72 \sigma}{ \phi_{\mu_n}}   \sqrt{\frac{K_{Nem} L}{n\eta} } |\mathcal{A}^*| .
	\end{equation*}
\end{thm}
Let us mention that $2 e \log (p) -3e < K_{Nem} < 2 e \log (p) -e$.
As a consequence, the rate of convergence in Theorems~\ref{ChapSLthm:BruitNonGaussian} is of the same order as in Corollary~\ref{Cor:BorneL1Sparse}.
However, the constant factor seems to be worse in the non-Gaussian case since it brings into play the constant $L$ which can be large.
This is the price to pay to adapt to the non-Gaussian noise.
\begin{rmq}
In the above theorem, $\eta$ is fixed. However, one can set $\eta$ depending on $p$ (or on $n$) in such way that it decreases to zero as $p \rightarrow \infty$ (or $n \rightarrow \infty$).
It is interesting to note that in this case, we loose a small power of $\log (p)$ (or $\log(n)$) in the rate of convergence when we consider non-Gaussian noise compared to the Gaussian case. 
\end{rmq}
Using similar reasoning as in Theorem~\ref{ChapSLthm:BruitNonGaussian} ({\it cf.} proof of Theorem~\ref{ChapSLthm:BruitNonGaussian}), there is no major difficulty to extend the variable selection results established in Section~\ref{ChapSLsec:supNorm} with Gaussian noise to the case where the noise is defined as above.
This can be done using Lemma~\ref{ChpSLConcentNOnGauss} instead of Lemma~\ref{ChapSLprobalm} of Section~\ref{sec:Proof} in all the proofs.

\section{Experimental Results}
\label{sec:simul}

In this section, we present the experimental performance of the estimator $\hat{\beta}^{Quad}$. In particular, we focus on two special cases: the Elastic-Net and the S-Lasso defined respectively with the penalties $\pen^{EN}(\beta) = \lambda | \beta |_{1} + \mu |\beta |_{2}^2$ and $\pen^{SL}(\beta) = \lambda |\beta |_{1} + \mu \sum_{j=2}^p (\beta_j -\beta_{j-1} )^2$. The Elastic-Net is useful when high correlations between variables appear, whereas the S-Lasso is devoted to problems where the regression vector $\beta^*$ is `smooth' (small variations in the values of the successive components of $\beta^*$). We are essentially interested in the performance of these estimators w.r.t. their estimation accuracy, that is, in terms of the estimation error $|\hat{\beta}-\beta^*|_2$, when $\beta^*$ is known (simulated data). Indeed, the introduction of $\hat{\beta}^{Quad}$ is motivated by a priori knowledge on the structure of the parameter $\beta^*$, or on the correlation between variables, and the purpose here is to see how this information can be taken into account to improve the reconstitution of the vector $\beta^*$.
As benchmarks, we use the Lasso and the Fused-Lasso estimators, since the first is the reference method and the second is close in spirit to the S-Lasso estimator. Indeed, the Fused-Lasso produces solutions with equal successive components (`piecewise linear')~\cite{Rosset-Fused}. Note also that in the pioneer paper of the Elastic-Net, a `corrected' version of this estimator is proposed \cite{Zou-E-Net}. There is as yet no theoretical support for this method. Moreover, it outperforms the `non-corrected' Elastic-Net (this `non-corrected' Elastic-Net is denoted by naive in \cite{Zou-E-Net}) in only a very few of the situations we consider in this paper. We omitted the results for these `corrected' versions to avoid digressions.\\
Except for the Fused-Lasso solution, all of the Lasso, the S-Lasso and the Elastic-Net solutions can be computed with the LARS algorithm ({\it cf.} Lemma~\ref{ChapSLequivalenceLasso}). However, we will not use the LARS algorithm in this study. In order to be fair with all the methods, we used the same algorithm for all of them. We use an algorithm provided by J. Mairal\footnote{http://www.di.ens.fr/$\sim$mairal/index.php} which is an implementation of a general algorithm given in \cite{NesterovAlgo07}.\\
In all our experiments, the tuning parameters are chosen based on the $10$ fold cross validation criterion (for the Fused-Lasso, the Elastic-Net and the S-Lasso, the cross validation is performed on a $2d$ Grid), but we also display the results obtained based on the theoretical values. Note that for the Fused-Lasso, we consider the same theoretical values of the tuning parameters as for the S-Lasso as they are both motivated by similar applications (this choice seems arbitrary, but to our knowledge no precise study has been made for the Fused-Lasso in the context we consider). On the other hand, both the Elastic-Net and the S-Lasso involve a sparse matrix $\widetilde{J}$ in the definition of the estimator $\hat{\beta}^{Quad}$.
Then, the theoretical values of the tuning parameters are $\lambda = 2\sqrt{2}\sigma \sqrt{ \log (p)/n }$ and $\mu =\lambda \sqrt{\mathcal{A}^*}/ 2|\widetilde{J}\beta^*|_{2}$, in accordance with Corollary~\ref{Cor:BorneL1Sparse} and Proposition~\ref{ChapSLprop:SupNormm}. These quantities depend on unknown parameters. They can be used only in the simulation study, otherwise one needs to estimate $|\widetilde{J}\beta^*|_{2}$.\\
\noindent The different methods are applied to several simulation examples. They also have been applied to a {\it pseudo-real dataset} generated from the riboflavin dataset.

\subsection{Synthetic data}\label{subsec:Sim}

There are several parameters: the dimension $p$, the sample size $n$ and the level of noise $\sigma$.
They will be specified in the experimental settings (that is, in the different tables and figures captions).
The first one is classical and has been introduced in the original paper of the Lasso \cite{Tibshirani-LASSO}.
The second simulation, where we are interested in observing the performance of the procedures when groups of variables appear, comes from \cite{Zou-E-Net}.
The last two studies aim on determining the behavior of the methods when the regression vector is `smooth'.
\begin{description}
	\item[{\it Example}~(a)]$[\sigma/\rho]$: {\it No particularities.} We fix $p=8$ and $n=20$. Here only $\beta_1^*$, $\beta_2^* $ and $\beta_5^*$ are nonzero and equal respectively $3,\,1.5$ and $2$. Moreover, for $j,\,k\in \{1,\ldots,8\}$ the design correlation matrix $\Psi$ is defined by $\Psi_{j,k} = \rho^{-|j-k|}$ where $\rho\in ]0,1[ $. 
	\item[{\it Example}~(b)]$[p/n/\sigma]$: {\it Groups.} We have $\beta_{j}^* =3$ for $j\in\{1,\ldots,15\}$ and zero otherwise. We construct three groups of correlated variables: $\Psi_{j,j}= 1$ for every $j\in \{1,\ldots,p\}$; for $j \neq k$, $\Psi_{j,k} \approx 1$ (actually $\Psi_{j,k} = \frac{1}{1+ 0.01}$, due to an extra noise variable) when $(j,k)$ belongs to $\{1,\ldots,5\}^2$, $\{6,\ldots,10\}^2$ and $\{11,\ldots,15\}^2$ and zero otherwise.
	\item[{\it Example}~(c)]$[p/n/\sigma]$: {\it Smooth regression vector.} The regression vector is given by $\beta_j^* = (3- 0.2j)^2$ for $j=1,\ldots,15$ and zero otherwise. Moreover, the correlations are described by $\Psi_{j,k} = \exp(-|j-k|)$ for $(j,k)\in\{1,\ldots,p\}^{2}$.
	\item[{\it Example}~(d)]$[p/n/\sigma]$: {\it High sparsity index and smooth regression vector.} The regression vector is such that $\beta_{j}^* = (4+0.1j)^2$ for $j\in\{1,\ldots,40\}$ and zero otherwise, and the correlations are the same as in {\it Example}~(c).
\end{description}
Except when $p=500$ where we run only $100$ replications, we based all the experiments on $500$ replications.

\begin{figure}[t]
\vskip -0.2in
\includegraphics[height=1.9in,width=3in] {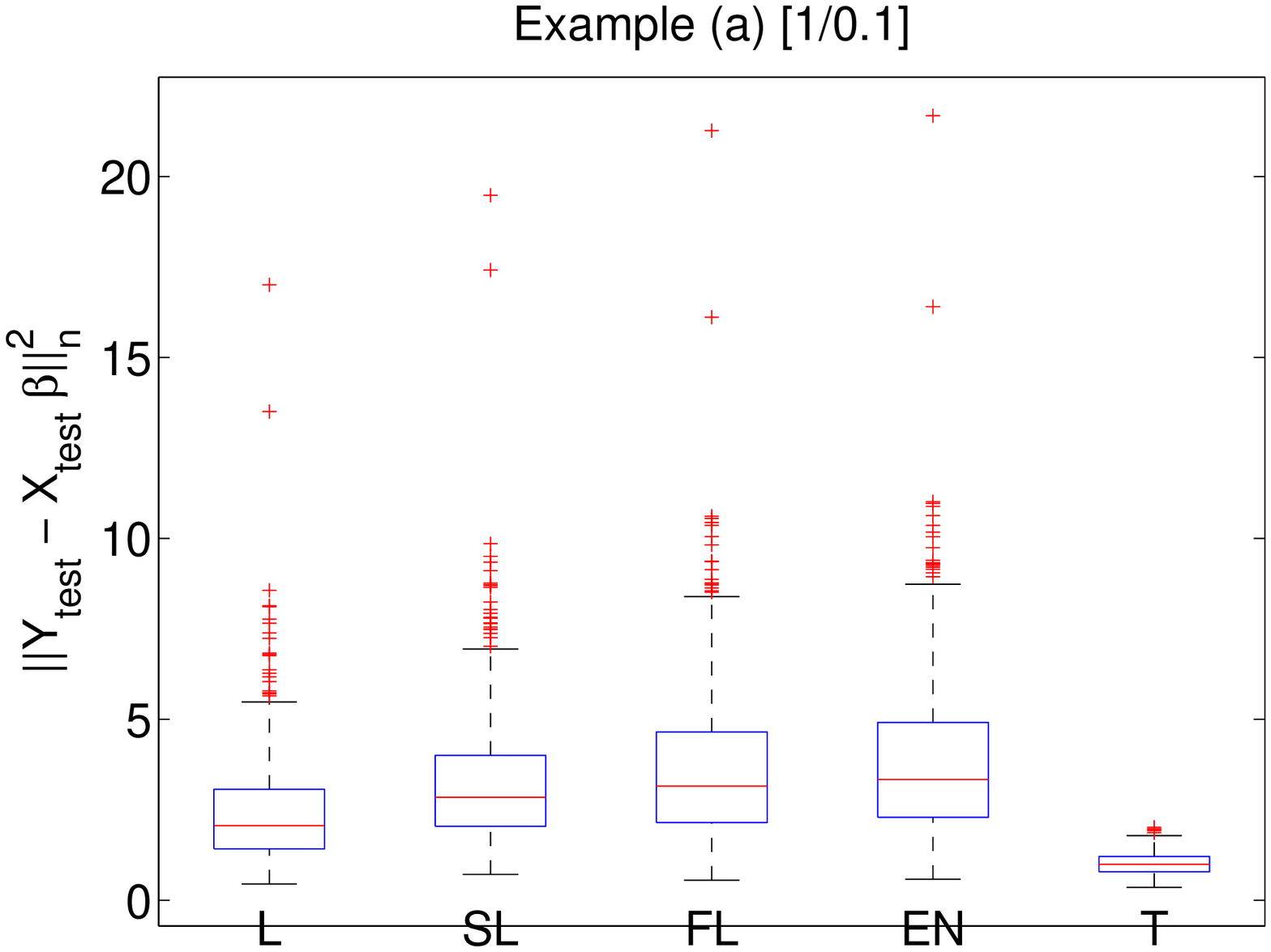}
\includegraphics[height=1.9in,width=3in] {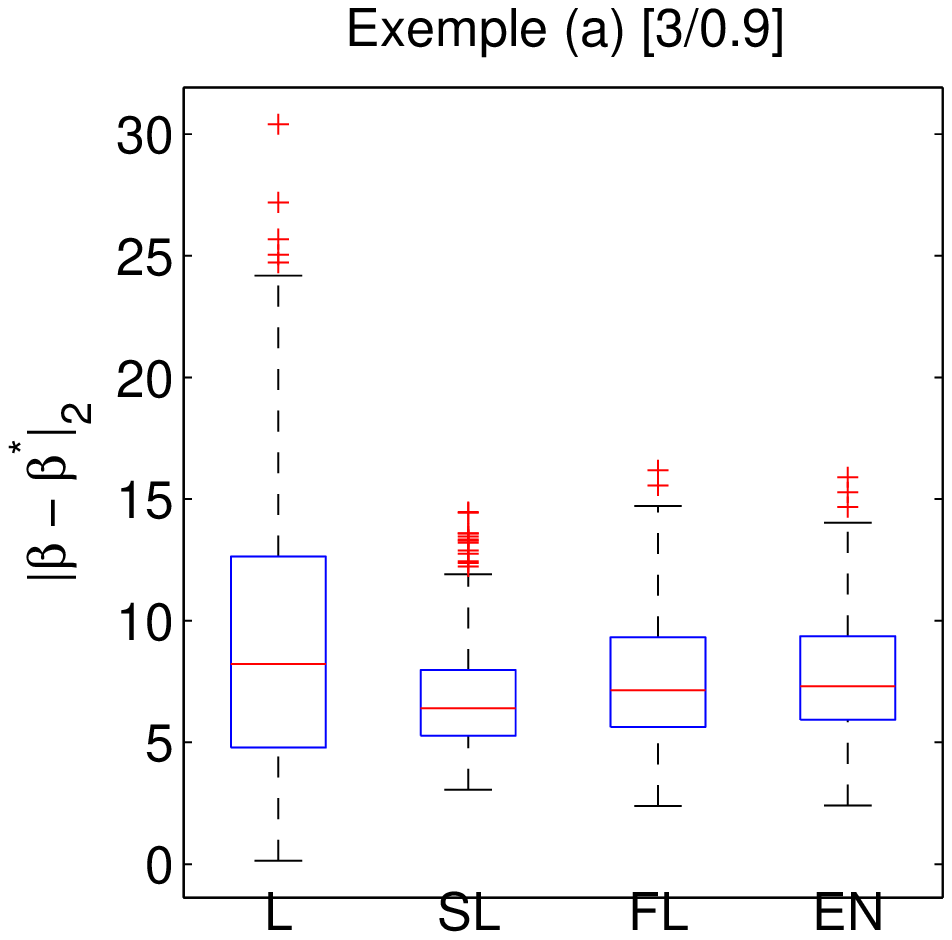}
\vglue-40pt
~
\begin{center}
\begin{minipage}[t]{0.90\textwidth}
\caption{\label{fig:BoxplotErrors1}\footnotesize Performance of the Lasso (L), the S-Lasso (SL), the Fused-Lasso (FL), and the Elastic-Net (EN) applied to {\it Example (a)} and based on $500$ replications. The tuning parameters are chosen based on the theoretical study. 
{\it Left}: Evaluation of the prediction error $\|Y_{test} - X_{test}\hat{\beta}\|_n^2$ in comparison with the performance of the truth (T), that is, $\|Y_{test} - X_{test}\beta^*\|_n^2$. {\it Right}: Evaluation of the $\ell_2$ estimation error $|\hat{\beta}-\beta^*|_2$.}
\end{minipage}
\end{center}
\end{figure}

\vspace{1cm}

\noindent {\bf Results.} The performance of the estimator $\hat\beta$ (which can be the Lasso, the S-Lasso, the Elastic-Net or the Fused-Lasso) in terms of the prediction error $\|Y_{test} - X_{test}\hat{\beta}\|_n^2$ (on a test set $(Y_{test},X_{test})$ of size $n$, that is, a set with the same size as the training set) and the $\ell_2$ estimation error $|\hat{\beta}-\beta^*|_2$ are illustrated by boxplots in Figures~\ref{fig:BoxplotErrors1} to~\ref{fig:BoxplotErrors4}. For some of these experiments, the corresponding computational costs (in seconds) of each method is reported in Table~\ref{tab:ComputCost}. In what follows, we first compare the methods to each other in terms of their accuracy. Then, we compare them in terms of their computational costs. Finally, we provide some numerical justifications to the theoretical calibration of the tuning parameters of the S-Lasso procedure.

\vspace{0.4cm}

\noindent \underline{{\it Methods comparison in terms of performance:}} Let us consider the different examples separately.\\
\noindent $-$ {\it Example}~(a): when we consider the procedures induced by the cross validation criterion (for the choice of the tuning parameter), we notice that none of them outperforms the others even when $\rho = 0.9$ (quite large correlation between successive variables).
This is observed for both prediction and estimation errors.
This is essentially due to the good behavior of the Lasso in such a situation where the regression vector is sparse but without any particular structure.
Actually, this conclusion holds in almost all the cases even when the tuning parameters are chosen based on the theoretical study.
However, two observations can be made.
First, when both of $\rho $ and $\sigma$ are small, the Lasso estimator performs slightly better than the other methods.
Moreover, when $\rho $ is large a small improvement can be observed using the Fused-Lasso, the Elastic-Net and the S-Lasso methods when we care about the estimation error.
This is illustrated in Figure~\ref{fig:BoxplotErrors1} (left and right respectively) where we display the performance of the methods in terms of the prediction error in {\it Example}~(a)~[$1/0.1$] (left) and in terms of the estimation error in {\it Example}~(a)~[$3/0.9$] (right).
For this example, the Lasso seems to be the best method since it involves only one tuning parameter.
It moreover has a lower (mean) computational cost equal to $0.18$ seconds (based on the cross validation criterion) as displayed in Table~\ref{tab:ComputCost}.
The S-Lasso, the Elastic-Net and the Fused-Lasso computational costs are respectively $3.7$, $3.6$ and $4.2$ seconds.
\begin{figure}[t]
\vskip -0.2in
\includegraphics[height=1.65in,width=2in] {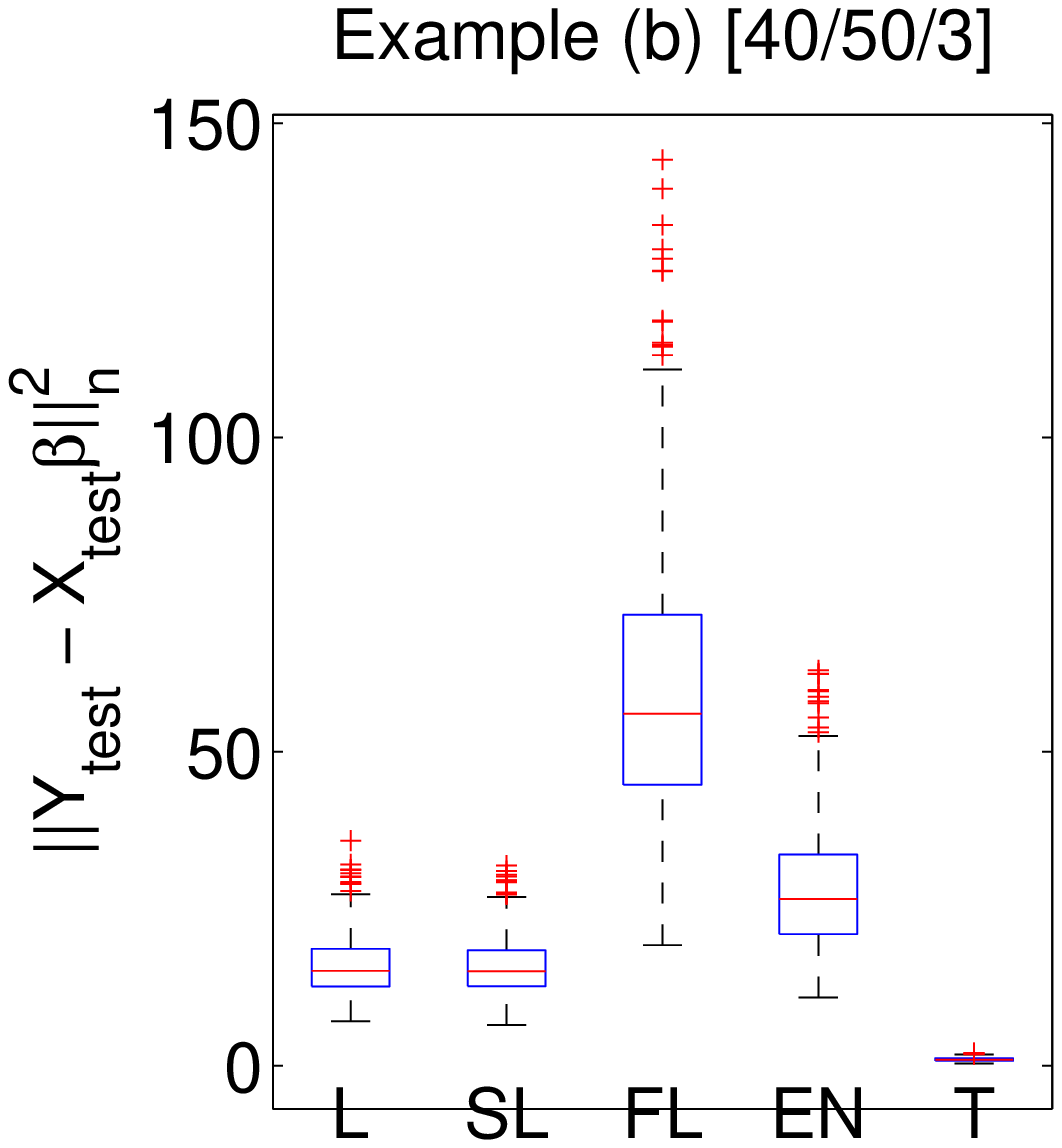}
\includegraphics[height=1.65in,width=2in] {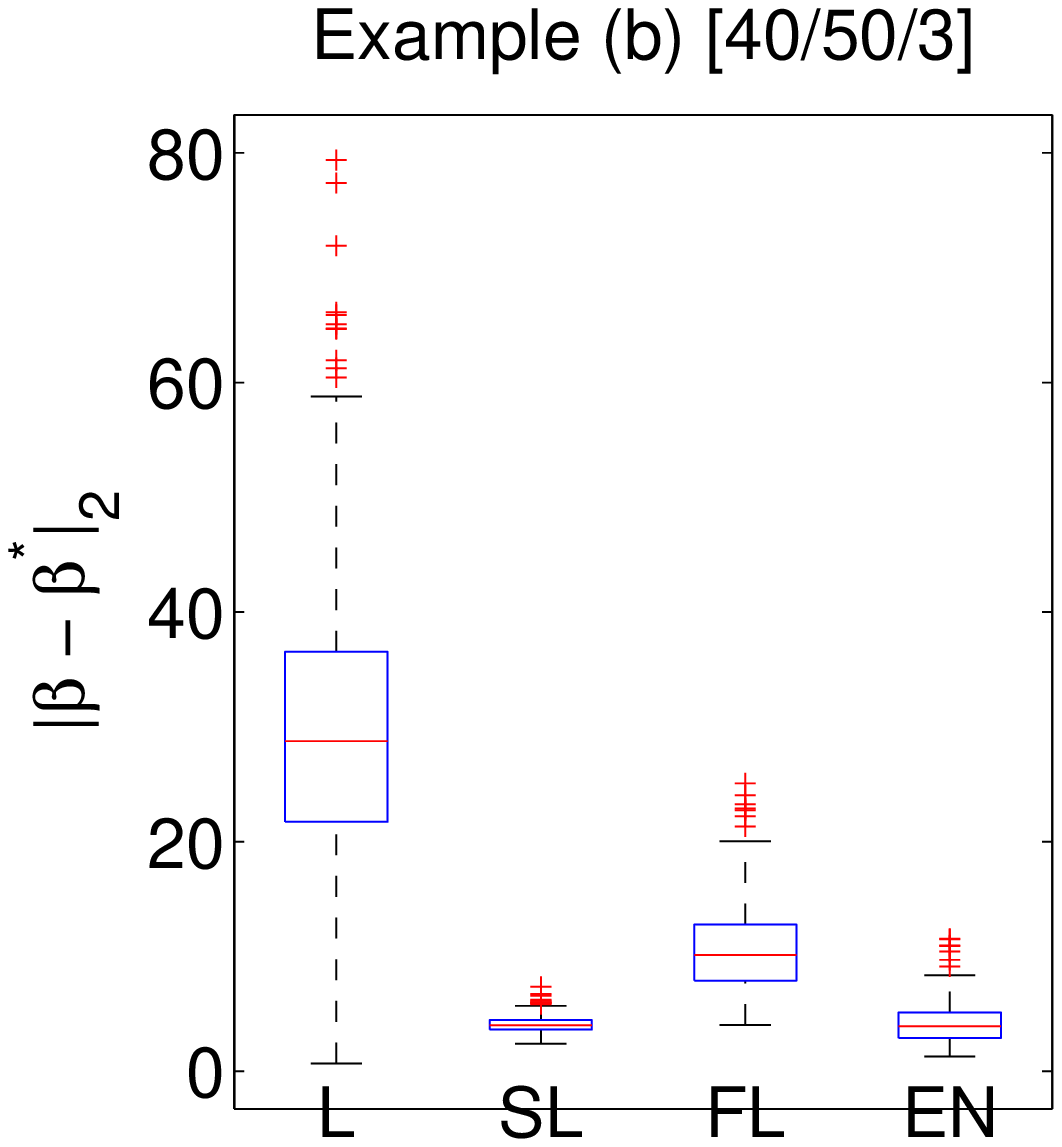}
\includegraphics[height=1.65in,width=2in] {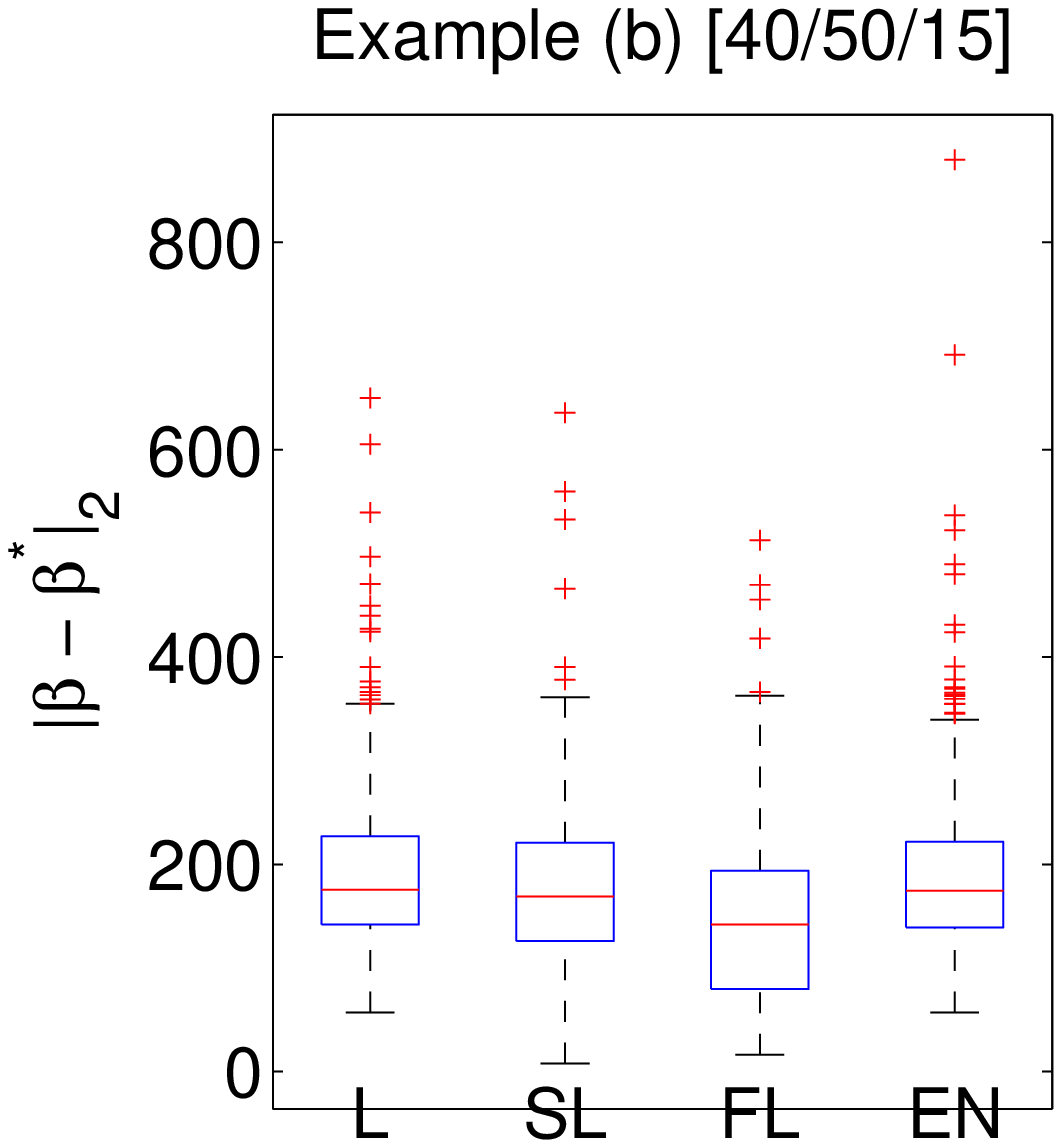}
\vglue-40pt
~
\begin{center}
\begin{minipage}[t]{0.90\textwidth}
\caption{\label{fig:BoxplotErrors2}\footnotesize Performance of the Lasso (L), the S-Lasso (SL), the Fused-Lasso (FL) and the Elastic-Net (EN) applied to {\it Example (b)} and based on $500$ replications. The tuning parameters are chosen based on the theoretical study in the first two plots
and by $10$ fold cross validation in the third.
{\it Left}: Evaluation of the prediction error $\|Y_{test} - X_{test}\hat{\beta}\|_n^2$, in comparison with the performance of the truth (T), {\it i.e.,} $\|Y_{test} - X_{test}\beta^*\|_n^2$. {\it Center and Right}: Evaluation of the $\ell_2$ estimation error $|\hat{\beta}-\beta^*|_2$.}
\end{minipage}
\end{center}
\end{figure}

\noindent $-$ {\it Example}~(b): with {\it Example}~(a), this is the least favorable example for the S-Lasso. Indeed, here the fifteen first coefficients equal $3$. Then the value of the coefficients drops down directly to $0$. There is a breakpoint in the `smoothness' in the true regression vector. Figure~\ref{fig:RegrVectSL} displays the best reconstitution of the regression vector $\beta^*$ using the S-Lasso solution (which minimizes the $\ell_2$ estimation error since $\beta^*$ is known). We observe the edge effects (breakpoint in the `smoothness') that the S-Lasso cannot solve due to the $\ell_2$ fusion penalty term. However, even in this case, it seems that all the procedures perform in a similar way when the tuning parameters are chosen by cross validation. When the noise level is large ($\sigma = 15$), let us nevertheless mention a (very) small improvement using the corrected versions of the S-Lasso and the Elastic-Net. Figure~\ref{fig:BoxplotErrors2} (right) illustrates the performance of the methods in terms of the estimation error when they are applied to {\it Example}~(b)~[$40/50/15$]. The Fused-Lasso outperforms the other methods slightly in this example (with $\sigma = 15$) when we deal with the estimation performance.\\
\noindent On the other hand, when the methods are based on the theoretical calibration of the tuning parameters, two observations can be made regardless of the noise level ($1 \leq \sigma \leq 15$): the S-Lasso and the Lasso perform better than the other methods in terms of the prediction error; the S-Lasso and the Elastic-Net provide good results whereas the Lasso has poor performance in terms of estimation error. This is illustrated in Figure~\ref{fig:BoxplotErrors2} (left and center respectively) when the methods are applied to {\it Example}~(b)~[$40/50/3$]. Note moreover that a similar results are also obtained when $p =100$ and $n=40$.
In this case, the behavior of the different methods seems to be stable with the parameters $p$, $n$ and $\sigma$. This example is quite interesting since it corroborates that a good method for the prediction objective can be less efficient for the estimation objective (see the performance of the Lasso and the Elastic-Net).

\begin{figure}[t]
\vskip -0.2in
\includegraphics[height=1.9in,width=3in] {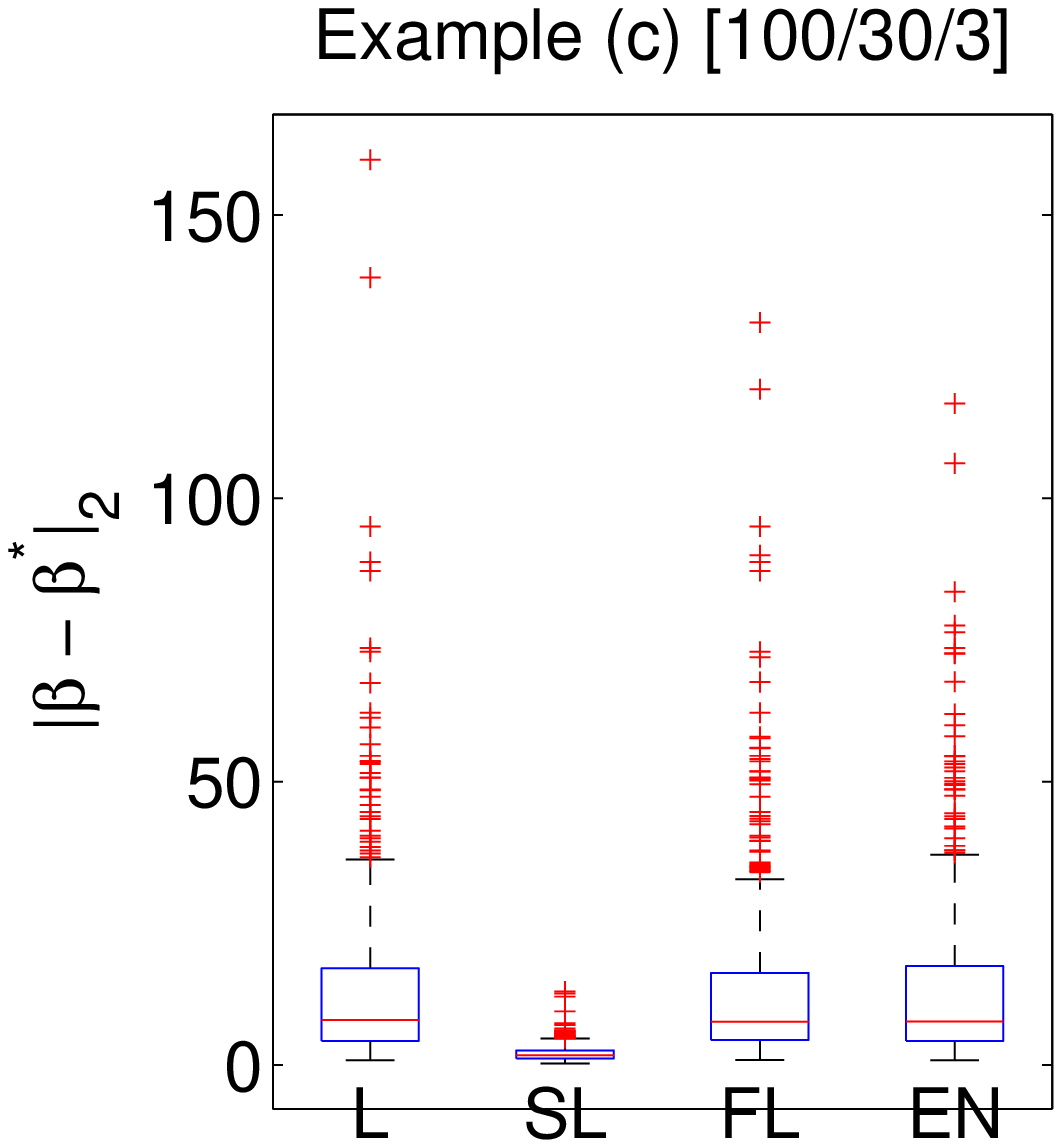}
\includegraphics[height=1.9in,width=3in] {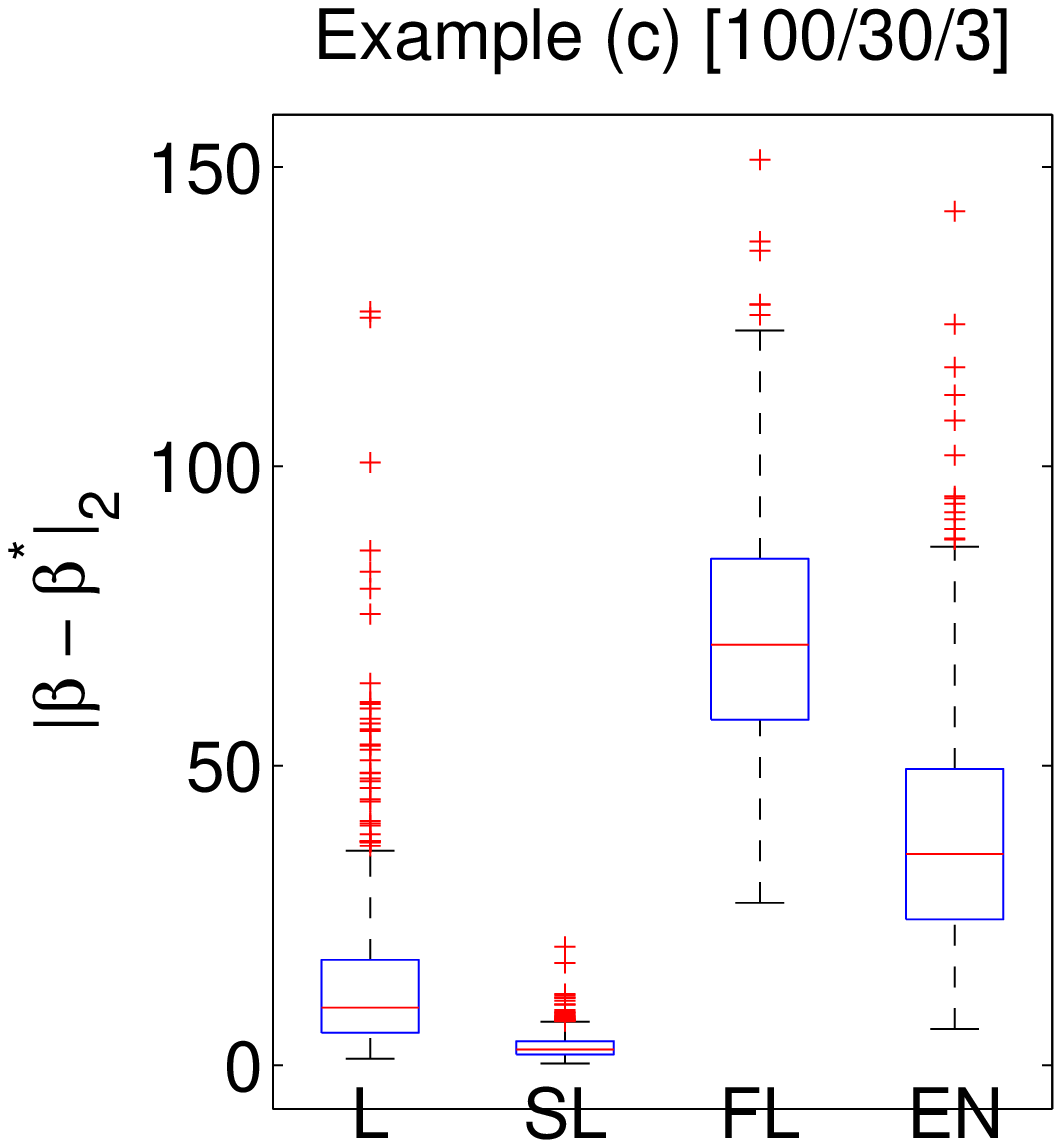}
\vglue-40pt
~
\begin{center}
\begin{minipage}[t]{0.90\textwidth}
\caption{\label{fig:BoxplotErrors3}\footnotesize Evaluation of the $\ell_2$ estimation error $|\hat{\beta}-\beta^*|_2$ of the Lasso (L), the S-Lasso (SL), the Fused-Lasso (FL) and the Elastic-Net (EN) applied to {\it Example (c)} and based on $500$ replications. {\it Left}: The tuning parameters are chosen by $10$ fold cross validation.
{\it Right}: The tuning parameters are chosen based on the theoretical study.
}
\end{minipage}
\end{center}
\end{figure}

\noindent $-$ {\it Example}~(c): we consider several values of the sample size $n$ and the dimension $p$. It turns out that here again, when $p<n$, all the methods behave in the same way when the tuning parameters are chosen by cross validation (the S-Lasso induces just a small improvement). However, when $p>n$ the S-Lasso is by far better than the other methods. This is illustrated by Figure~\ref{fig:BoxplotErrors3} (left) where the $\ell_2$ estimation error of each method applied to {\it Example}~(c)~[$100/30/3$] is displayed. The same plot is obtained for the prediction error.\\
\noindent Moreover, when the tuning parameters are calibrated according to the theoretical study, the S-Lasso performs the best and the Fused-Lasso the worst. This appears to be true whatever the values of the parameters $p$, $n$ and $\sigma$. See for instance Figure~\ref{fig:BoxplotErrors3} (right) where the different methods are applied to {\it Example}~(c)~[$100/30/3$] and for the estimation task (the same is obtained for the prediction objective).\\
Note that in this example, the Fused-Lasso and the Elastic-Net appear to be useless.\\
\noindent $-$ {\it Example}~(d): this is with {\it Example}~(c) the most favorable situation for the S-Lasso estimator where the regression vector is `smooth' with a large amount of non-zero components. The S-Lasso estimator seems to dominate its opponents in all the cases and regardless of the sample size $n$, the dimension $p$, or the noise level $\sigma$. This observation holds for the $\ell_2$ estimation and the prediction errors. Note that when the tuning parameters are chosen by cross validation, the Lasso, the Fused-Lasso and the Elastic-Net have quite close performance. Figure~\ref{fig:BoxplotErrors4} illustrates this fact when $p<n$ for the estimation error (left: cross validation; center-left: theory). Moreover, Figure~\ref{fig:BoxplotErrors4} (center-right and right) displays the performance of the methods when $p>n$ in case where the tuning parameters are based on the theoretical study (note that ranking of the methods does not change from the case $p<n$ when the tuning parameters are chosen by cross validation).
In addition, an interesting observation follows from the experiments on {\it Example}~(d)~$[100/30/3]$ (Figure~\ref{fig:BoxplotErrors4}-left) . Indeed, here the sparsity index $|\mathcal{A}^*| = 40$ and it is then larger than the sample size $n=30$. In this case, the Lasso has poor performance. However, the S-Lasso is still good. Moreover, there even exists a pair $(\lambda,\mu)$ (the pair minimizing the $\ell_2$ estimation error since $\beta^*$ is known) such that we have a good reconstitution on the regression vector $\beta^*$ (see Figure~\ref{fig:RegrVectSL}-right).

\begin{figure}[t]
\vskip -0.2in
\includegraphics[height=1.65in,width=1.5in] {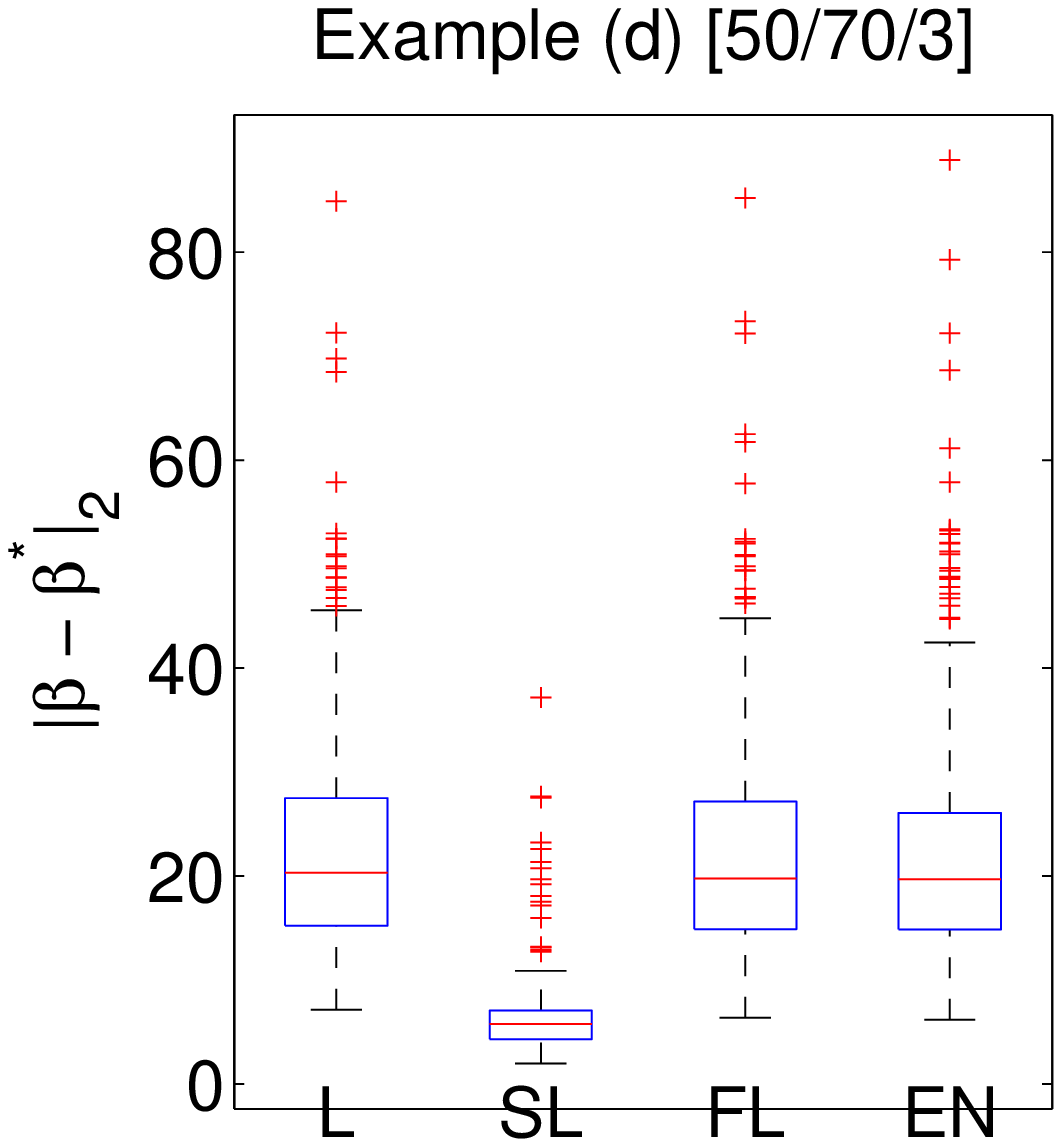}
\includegraphics[height=1.65in,width=1.5in] {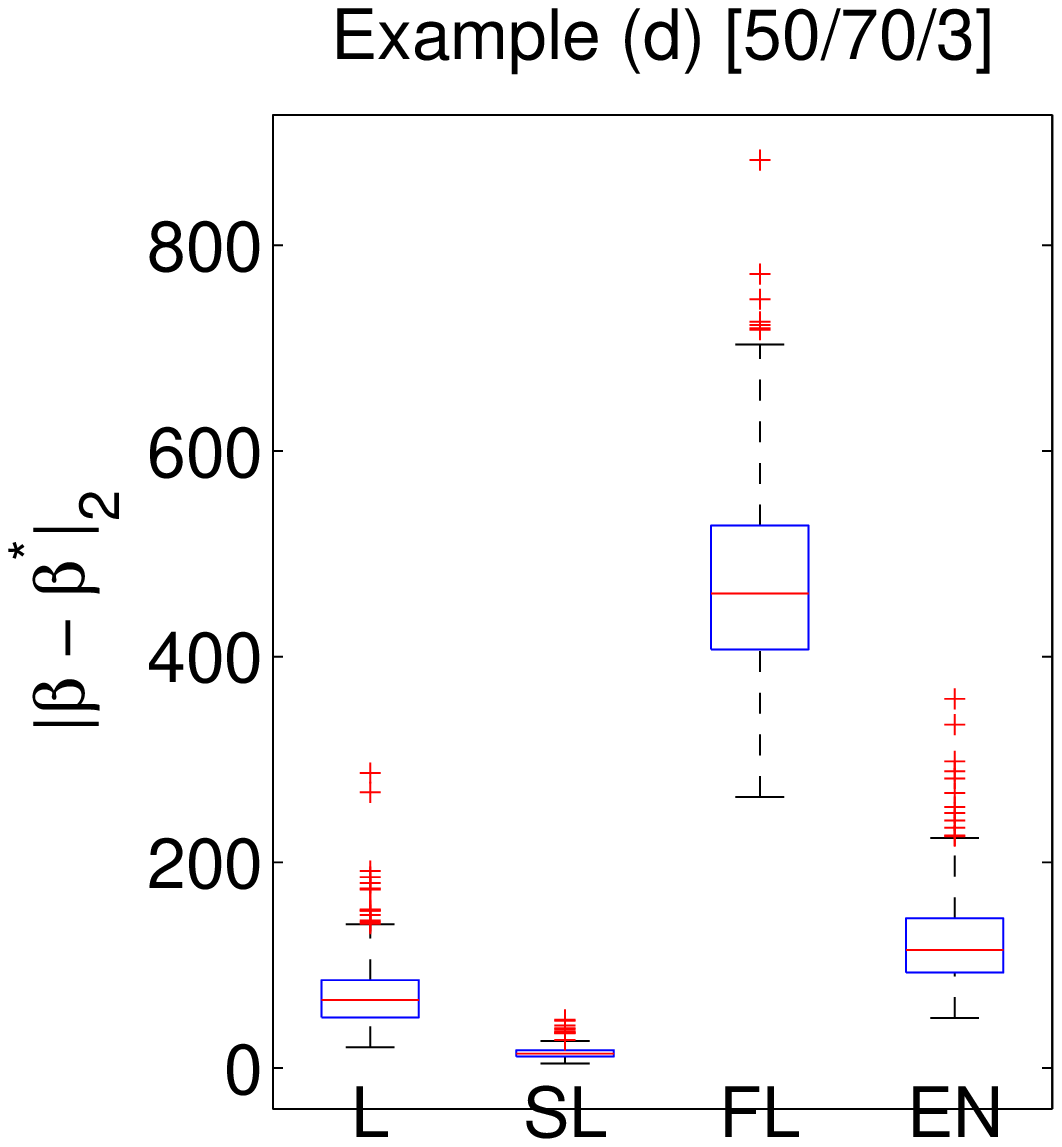}
\includegraphics[height=1.65in,width=1.5in] {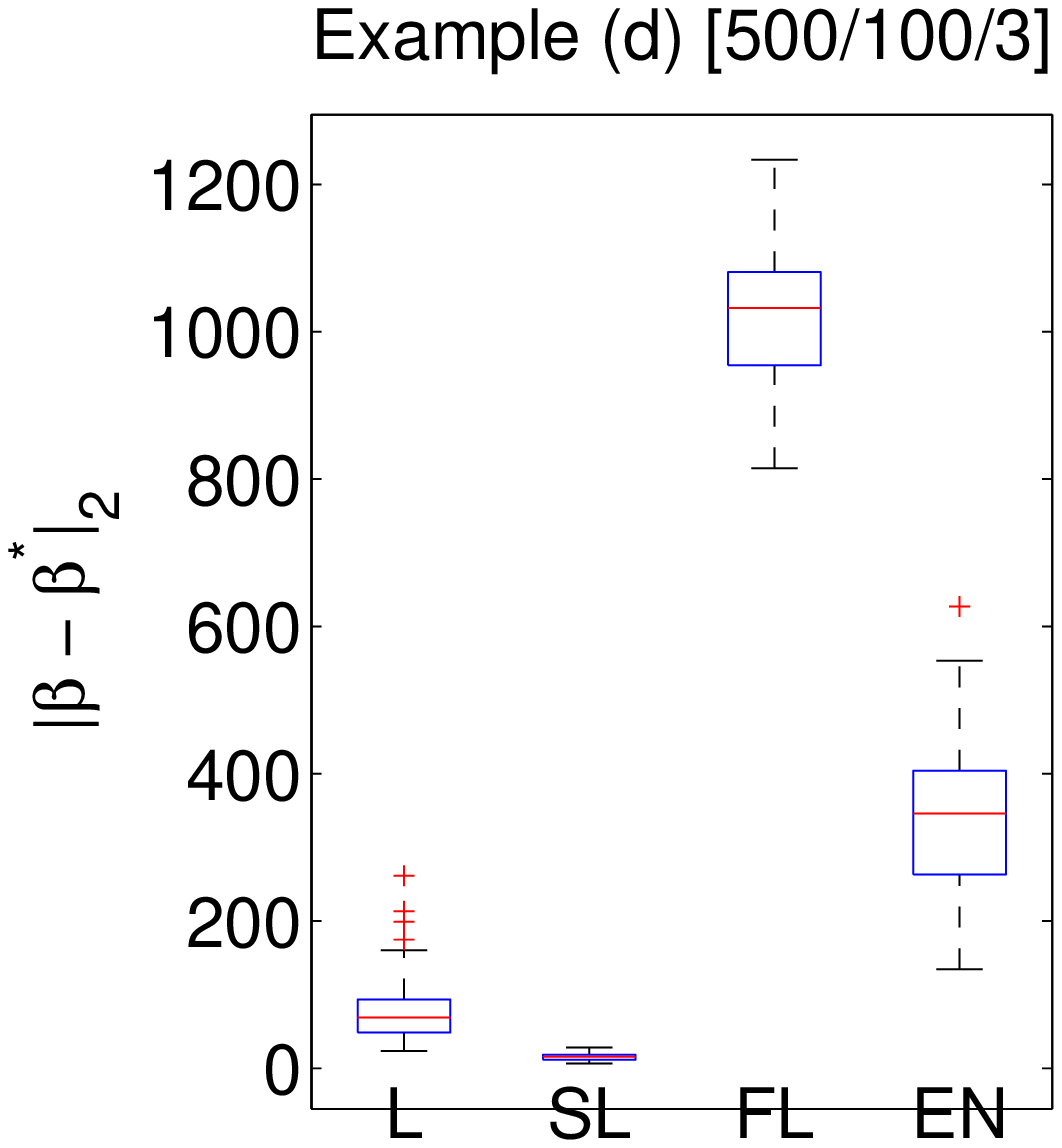}
\includegraphics[height=1.65in,width=1.5in] {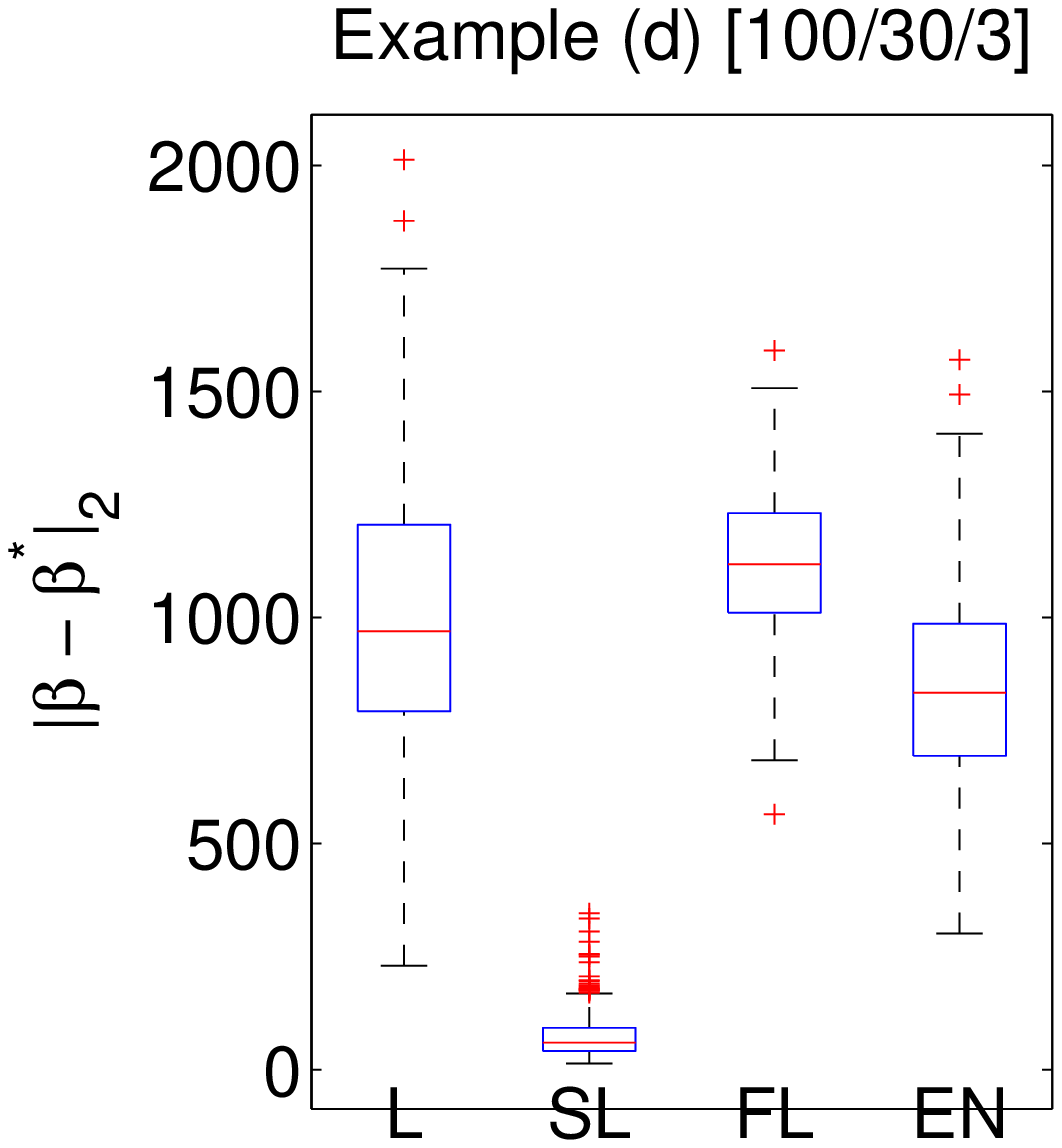}
\vglue-40pt
~
\begin{center}
\begin{minipage}[t]{0.90\textwidth}
\caption{\label{fig:BoxplotErrors4}\footnotesize Evaluation of the $\ell_2$ estimation error $|\hat{\beta}-\beta^*|_2$ of the Lasso (L), the S-Lasso (SL), the Fused-Lasso (FL) and the Elastic-Net (EN) applied to {\it Example (d)} and based on $500$ replications. {\it Left}: The tuning parameters are chosen by $10$ fold cross validation.
{\it Center-left; Center-right; Right}: The tuning parameters are chosen based on the theoretical study.
}
\end{minipage}
\end{center}
\end{figure}
\begin{figure}[t]
\vskip -0.2in
\includegraphics[height=1.9in,width=3in] {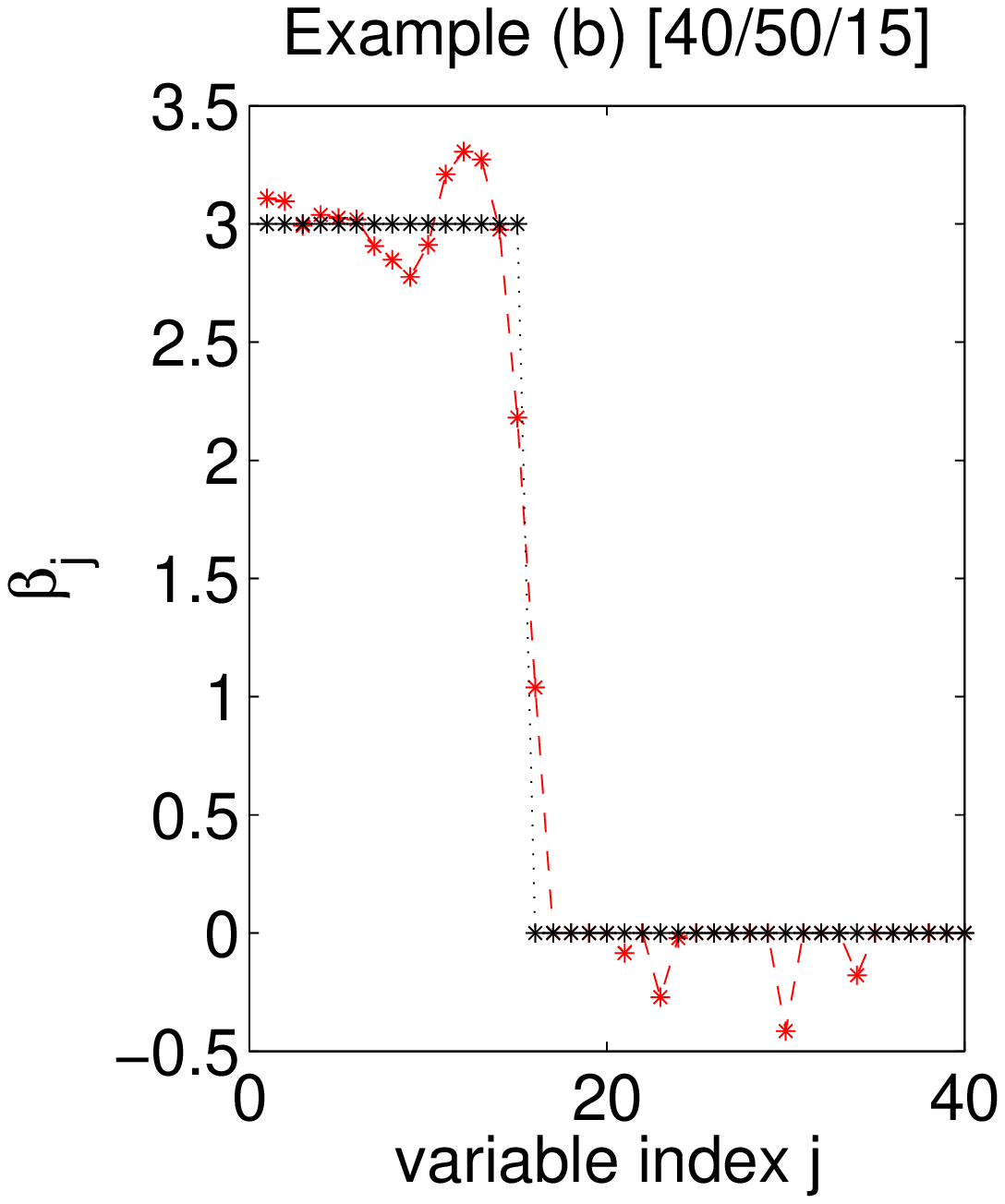}
\includegraphics[height=1.9in,width=3in] {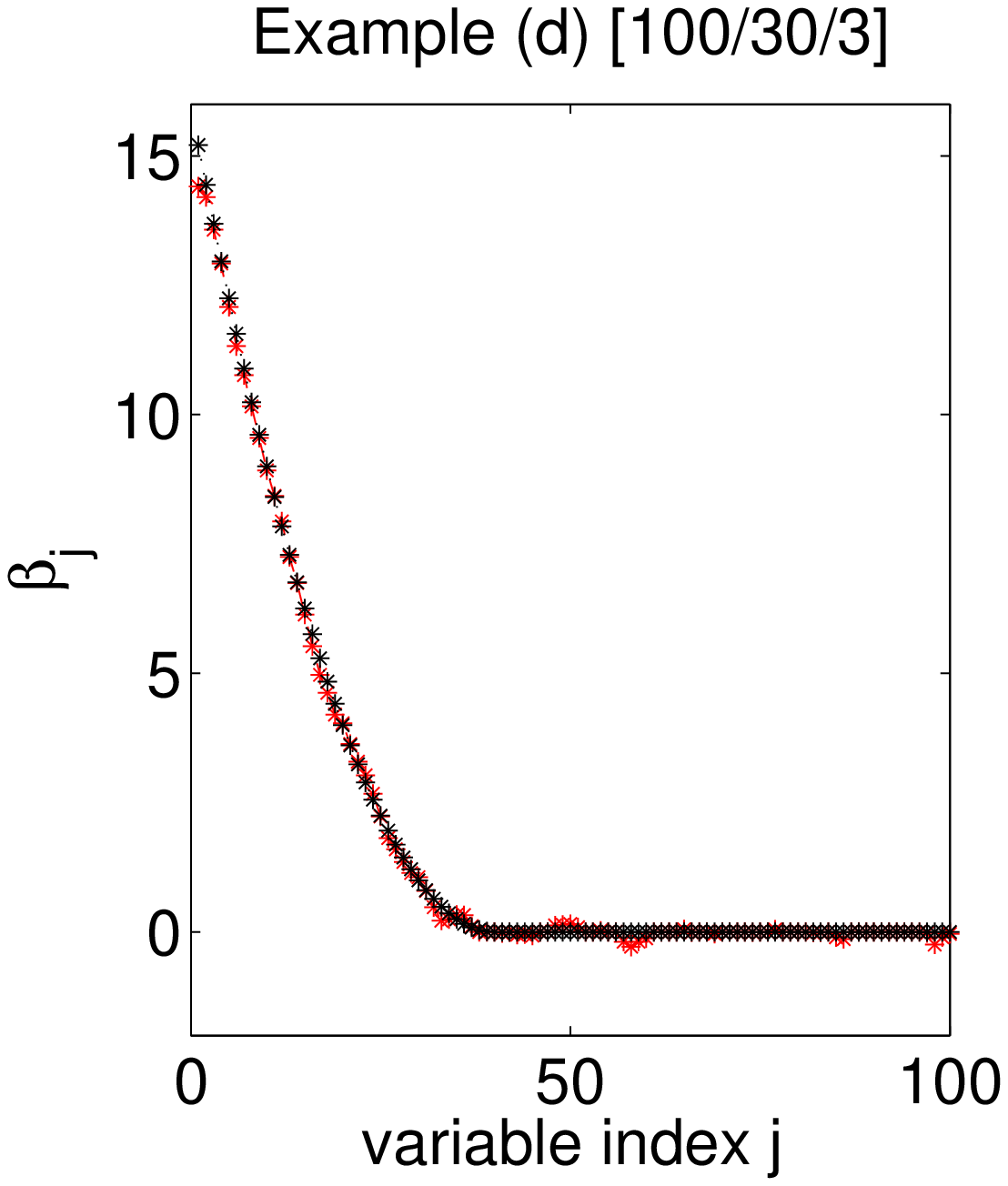}
\vglue-40pt
~
\begin{center}
\begin{minipage}[t]{0.90\textwidth}
\caption{\label{fig:RegrVectSL}\footnotesize Best reconstitution of the regression vector $\beta^*$ (black curve) by the SL-Lasso estimator (red curve). {\it Left}: Application to {\it Example}~(b)~[$40/50/15$]. {\it Right}: Application to {\it Example}~(d)~[$100/30/3$].}
\end{minipage}
\end{center}
\end{figure}

\vspace{0.4cm}

\noindent \underline{{\it Methods comparison in terms of computational costs:}} Table~\ref{tab:ComputCost} displays the computational cost (in seconds) of each method on several examples. First note that the Fused-Lasso has the largest computational cost in all the simulations whereas the Lasso has the smallest. The Elastic-Net and the S-Lasso have intermediate computational costs but are still reasonable compared to the Fused-Lasso.
More precisely, when the tuning parameters are chosen by cross validation, we remark that the computational costs for the S-Lasso and the Elastic-Net are about $30$ times larger than for the Lasso. This can partly be explained by the number of values explored for the tuning parameter $\mu$ (a grid with $20$ elements).
Actually, since the S-Lasso and the Elastic-Net are obtained with a Lasso program applied to expanded data ({\it cf.} Lemma~\ref{ChapSLequivalenceLasso}), it turns out that even for fixed $\lambda$ and $\mu$, the computation costs of the Lasso is (a bit) smaller than the computation costs of the S-Lasso and the Elastic-Net.
This is observed for example when we consider the solutions computed when the tuning parameters are chosen based on the theoretical study.
Except {\it Example}~(a), where the increase of computational cost using the S-Lasso and the Elastic-Net is not justified (since the improvement using the Lasso-type methods is quite small), in most of the considered situations it is quite interesting to use the Elastic-Net and even more interesting to use the S-Lasso estimator. This is due to the `smoothness' of the true regression vector. \\
\noindent Finally, the Fused-Lasso has a large computation cost due to the $\ell_1$-fusion penalty which admits a singularity.
Moreover, it does not improve significantly the Lasso estimator in the situations we considered in this paper (as observed in the previous part).\\
In view of the computational costs related to {\it Example}~(a) (the first two columns in Table~\ref{tab:ComputCost}), let us finally remark that these costs increase with $\rho$, the correlation level between variables, and $\sigma$, the noise level.
We observe for example that the mean computational cost of the Lasso estimator (when the tuning parameter is chosen by cross validation) is $1.1$ seconds when $\rho=0.1$ and $\sigma=1$ and increases to $8$ seconds when $\rho=0.9$ and $\sigma=3$.

\begin{table}[t]
\caption{Computational costs in seconds for the Lasso (L), the S-Lasso (SL), the Fused-Lasso (FL) and the Elastic-Net (EN) in several examples illustrated in the above figures. We chose either $Tuning =Th$ or $Tuning =Cv$, depending on whether we consider the methods with the tuning parameters based on the theoretical issue or on the $10$ fold cross validation.}
\label{tab:ComputCost}
\begin{center}
\begin{tiny}
\begin{sc}
\begin{tabular}{|l|c||c||c||c||c||c|}
\hline
Meth.     & Tuning & {\it Ex.}(a)~$[1/0.1]$  & {\it Ex.}(a)~$[3/0.9]$ & {\it Ex.}(b)~$[40/50/15]$ & {\it Ex.}(c)~$[30/50/3]$  & {\it Ex.}(d)~$[500/100/3]$ \\
\hline
\hline
\multirow{2}{*}{L}  & $ Th \,\, \cdot 10^{-4}$  & $ 1.1 \pm 0.1$   & $8 \pm  41 $ & $ 5 \pm 2 $  & $ 33 \pm  64$  & $ 457 \pm  243 $ \\
\cline{2-7}
        & $Cv$   & $ 0.18 \pm 0.01 $ & $ 0.5 \pm 0.2$   &  $ 0.5 \pm 0.1 $  & $ 1.1 \pm 0.3 $   & $ 12.3 \pm 4.9  $ \\
\hline
\multirow{2}{*}{SL}       & $Th \,\, \cdot 10^{-4}$ & $5.1 \pm 6.4$    & $ 8 \pm 28 $ &  $ 6 \pm 6 $ & $ 48 \pm 81$   & $  967 \pm 441  $ \\
\cline{2-7}
        & $Cv$ & $ 3.7 \pm 0.1 $ & $ 11.1 \pm 1.3 $ &  $ 10.2 \pm 2.0$ & $ 36.2 \pm 9.1$  & $ 648.3 \pm  219.2 $ \\
\hline
\multirow{2}{*}{FL}      & $Th \,\, \cdot 10^{-4}$    &  $ 2.6 \pm 0.3 $  &    $ 10.0 \pm 30.0 $ & $ 20 \pm 12 $ &  $ 518 \pm 271$  & $5996 \pm  2019 $ \\
\cline{2-7}
       & $Cv$  & $ 4.2 \pm 0.2$ & $ 14.1  \pm 1.6 $ &  $38.3 \pm 5.8 $ &  $ 245.6 \pm  64.3$ & $ \simeq 3\cdot 10^3   $\\
\hline
\multirow{2}{*}{EN}       & $Th \,\, \cdot 10^{-4}$   &  $ 4.7 \pm 3.5 $   & $ 9 \pm 43 $ & $5 \pm 3 $  & $ 41 \pm 60 $ & $1022 \pm 432  $ \\
\cline{2-7}
        & $Cv$    & $  3.6 \pm 0.2  $ & $ 11.0 \pm 1.3 $ &  $ 10.2 \pm 2.0 $ & $ 35.2 \pm 8.9$ & $637.3 \pm  214.0 $ \\
\hline
\end{tabular}
\end{sc}
\end{tiny}
\end{center}
\vskip -0.1in
\end{table}

\vspace{0.4cm}

\noindent \underline{{\it S-Lasso; theory vs. cross validation:}}
in what follows, we compare both of the version of the S-Lasso.
That is, we compare the S-Lasso when the tuning parameters are chosen by cross validation and when the tuning parameters are chosen based on the theoretical study:

\vspace{0.2cm}

\noindent $\bullet$ first, we compare these two methods in terms of their {\it performance}. Figure~\ref{fig:BoxplotErrors5} summarizes the comparison between the S-Lasso based on a theoretical choice of the tuning parameters (denoted in this part by S-Lasso$^{Th}$) and the S-Lasso where the tuning parameters are based on $10$ fold cross validation (denoted here by S-Lasso$^{Cv}$).
First we can observe that the performance of both S-Lasso$^{Th}$ and S-Lasso$^{Cv}$ are close.
Moreover, given the results in the part `{\it Methods comparison in terms of performance}', they both perform in a good way. However, it seems that S-Lasso$^{Cv}$ outperforms S-Lasso$^{Th}$ when we deal with the prediction task. This seems quite intuitive since by definition, the cross validation criterion attempts to provide good estimator for the prediction objective. According to the $\ell_2$ estimation goal, we cannot conclude the superiority of one of the estimators on the other. Nevertheless, in the high dimensional setting {\it Example}~(d)~[$500/100/\sigma$], it seems that S-Lasso$^{Cv}$ begins to become better.\\
At least, the theoretical choice for $\mu$ ($\mu = \frac{\lambda\sqrt{|\mathcal{A}^*|}}{2|\widetilde{J}\beta^*|_{2}} $) provides good performance both in terms of $\ell_2$ estimation error and test error. They are often close to the performance of the S-Lasso estimator based on the cross validation criterion. This is quite interesting since the computational cost of S-Lasso$^{Th}$ is much smaller than S-Lasso$^{Cv}$. This study is actually more a verification of our theoretical choices of the tuning parameters than a rule to apply in practice. Indeed, since the theoretical choice of $\mu$ depends on $\beta^*$, the corresponding estimator S-Lasso$^{Th}$ is unusable in real data problems;

\begin{figure}[t]
\vskip -0.2in
\includegraphics[height=1.3in,width=1.5in] {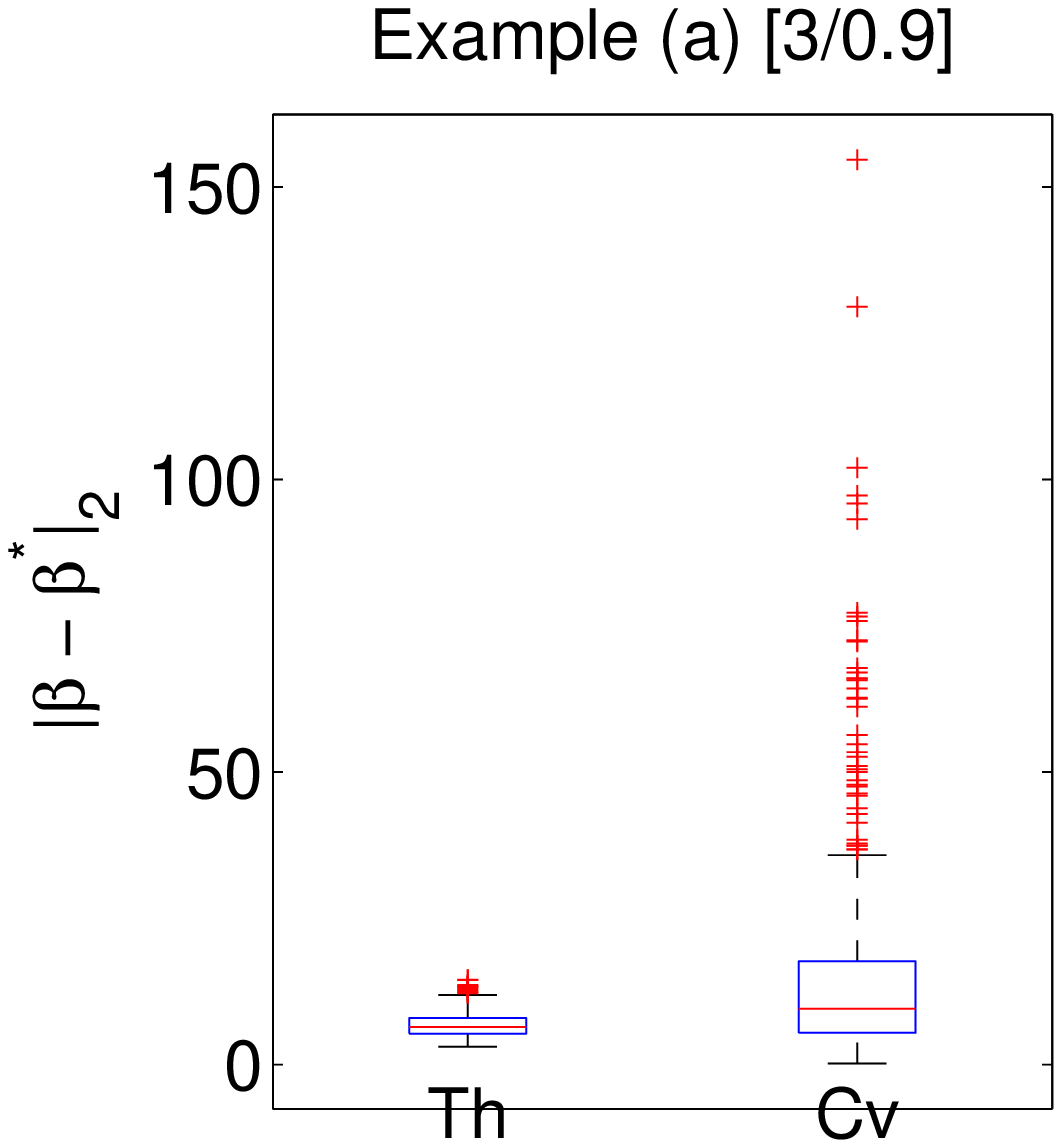}
\includegraphics[height=1.3in,width=1.5in] {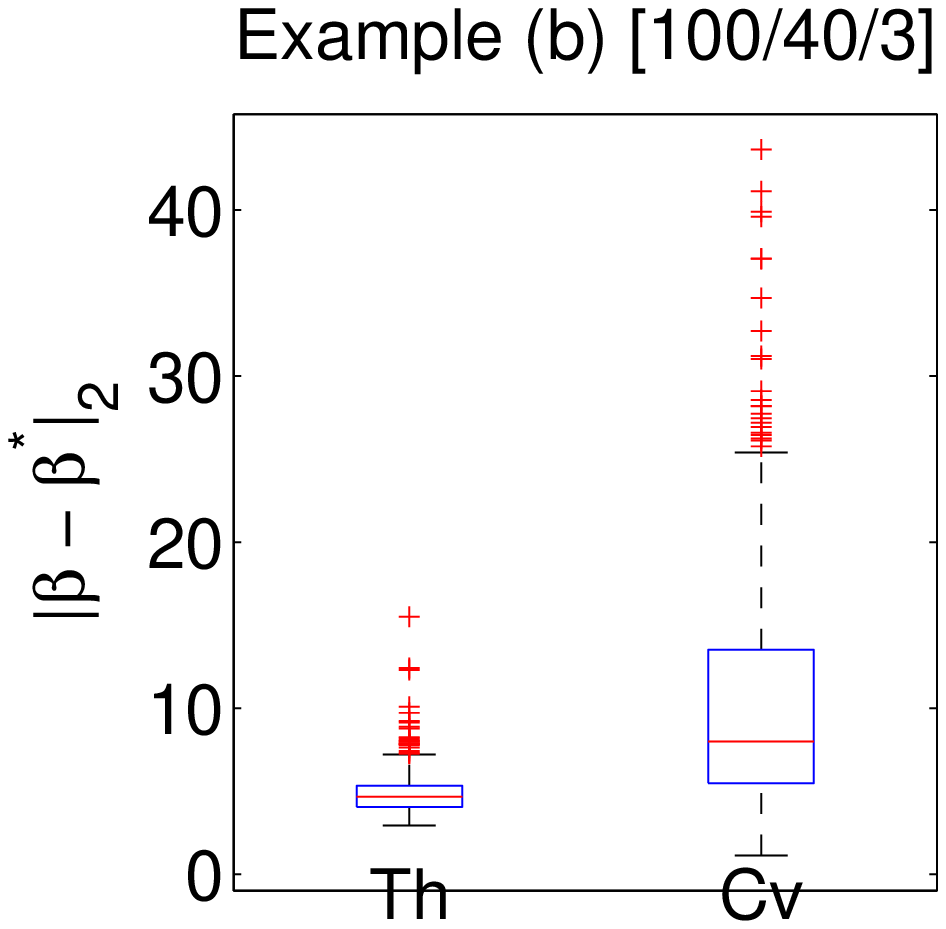}
\includegraphics[height=1.3in,width=1.5in] {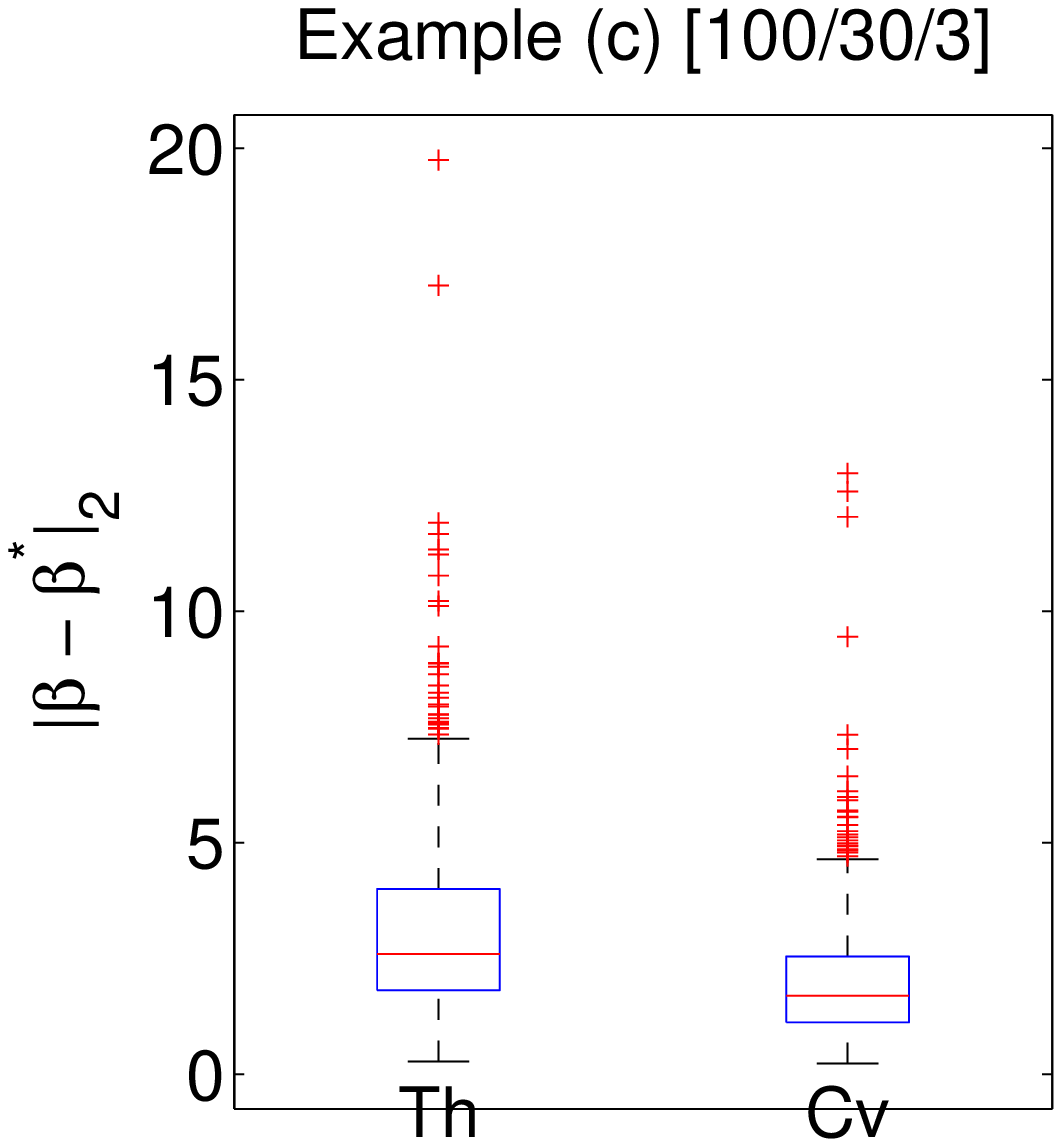}
\includegraphics[height=1.3in,width=1.5in] {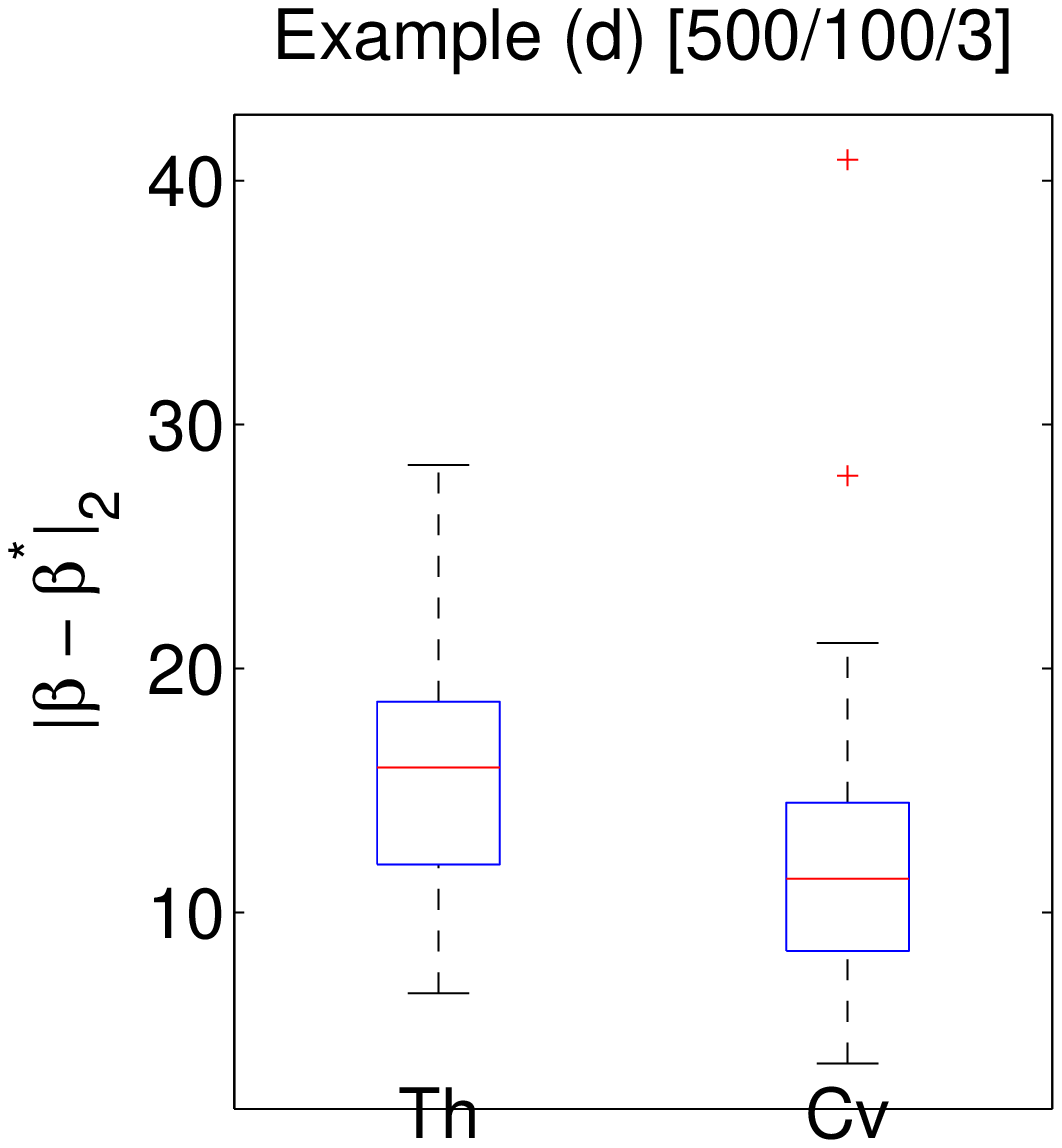}
\\
\includegraphics[height=1.3in,width=1.5in] {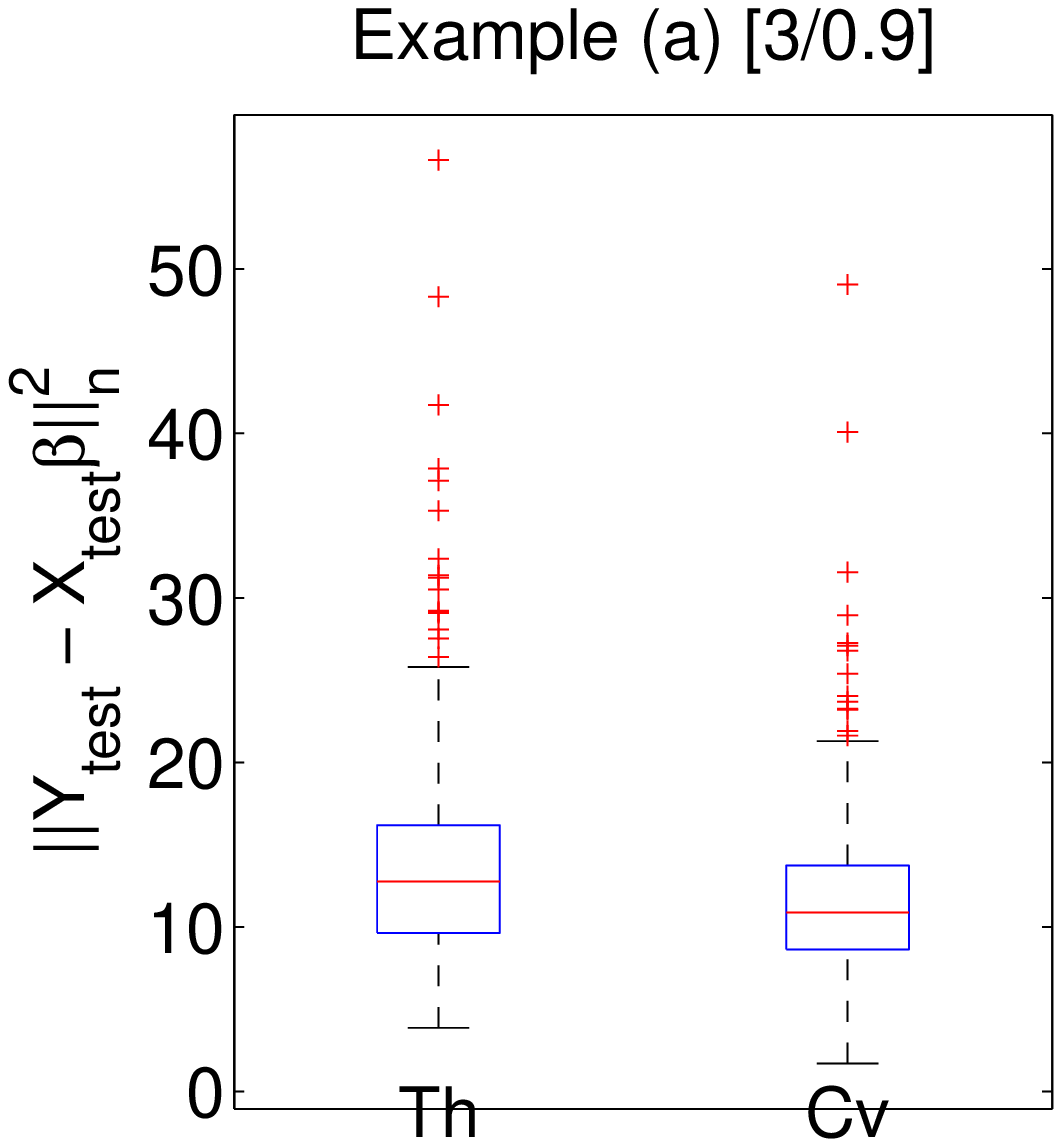}
\includegraphics[height=1.3in,width=1.5in] {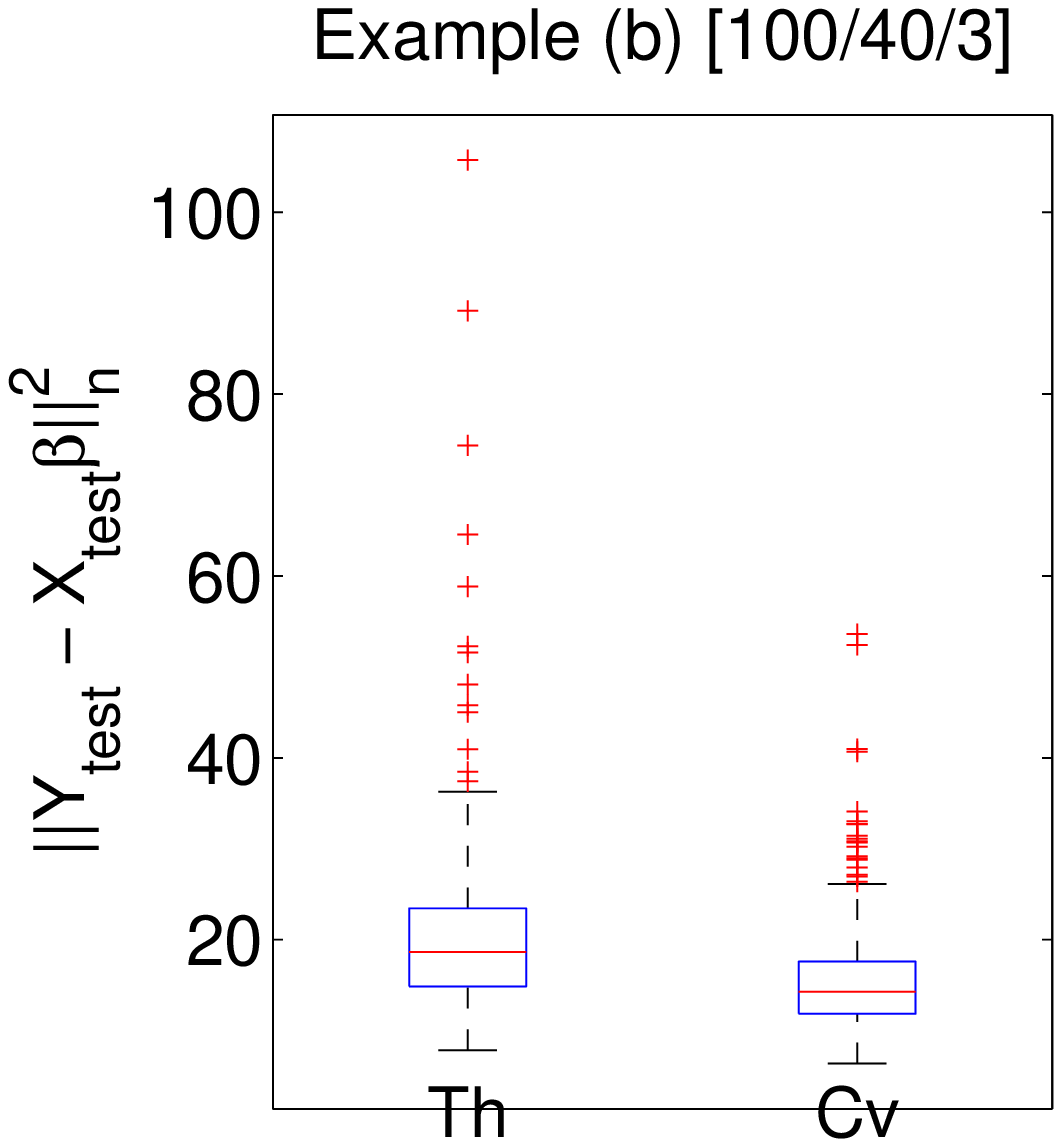}
\includegraphics[height=1.3in,width=1.5in] {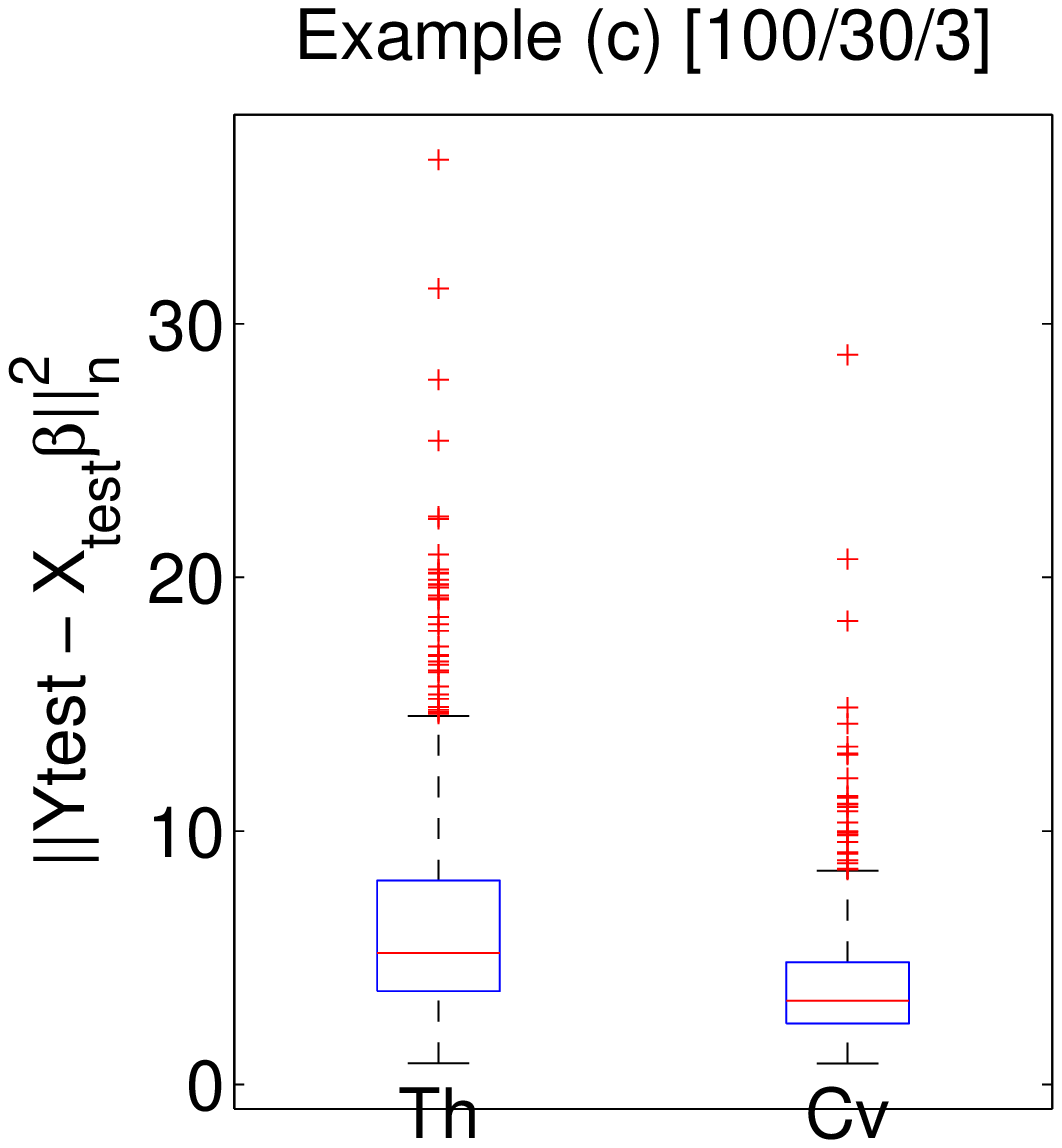}
\includegraphics[height=1.3in,width=1.5in] {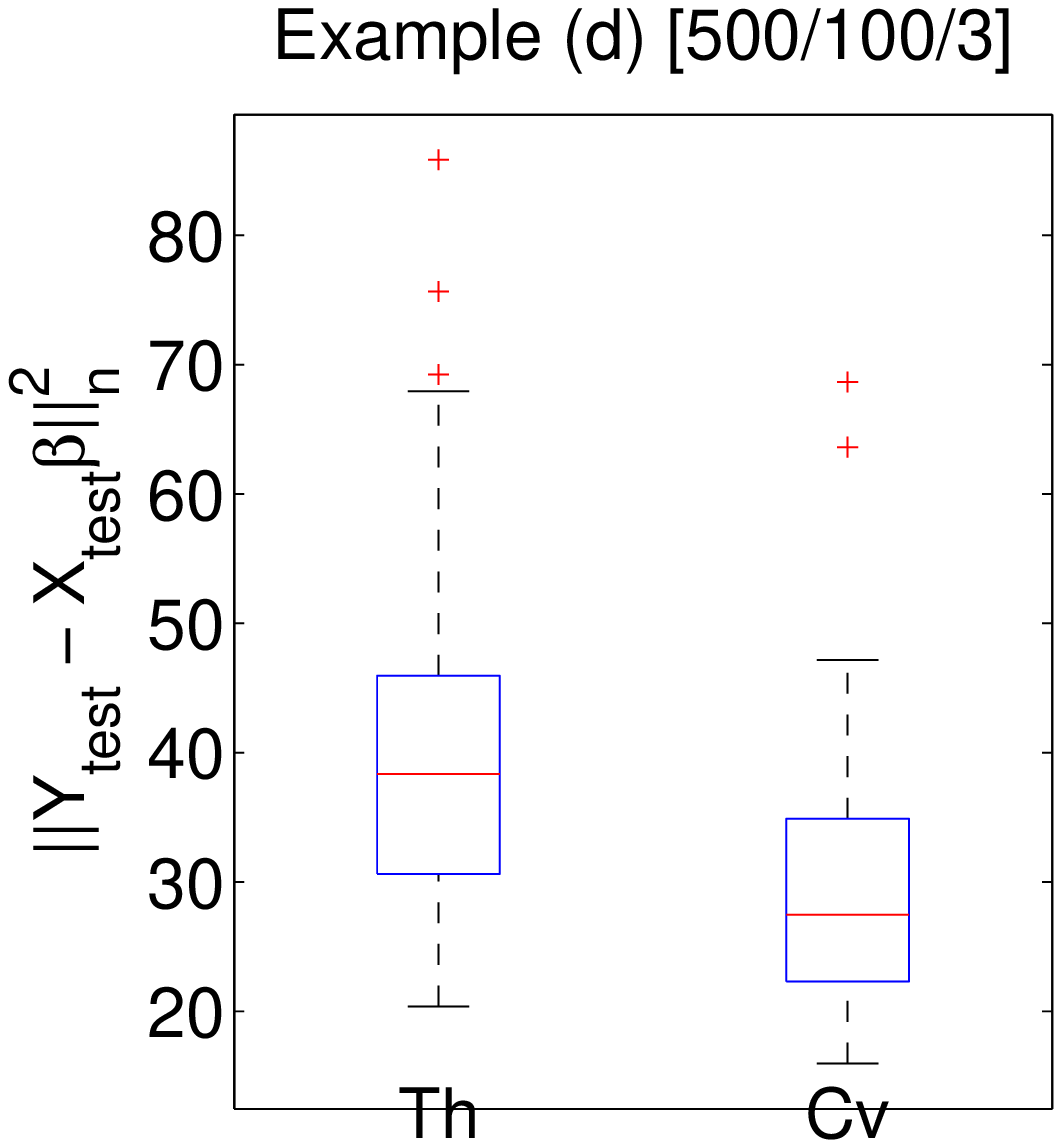}
\vglue-40pt
~
\begin{center}
\begin{minipage}[t]{0.90\textwidth}
\caption{\label{fig:BoxplotErrors5}\footnotesize Evaluation of the $\ell_2$ estimation error $|\hat{\beta}-\beta^*|_2$ (top) and the prediction error $\|Y_{test} - X_{test} \hat{\beta}\|_n^2$ (bottom) of the S-Lasso based on $500$ replications.
For each subplot: {\it Left}: The tuning parameters are chosen by $10$ fold cross validation.
{\it Right}: The tuning parameters are chosen based on the theoretical study.
We refer to Table~\ref{tab:tuningparm} for an evaluation of these tuning parameters}
\end{minipage}
\end{center}
\end{figure}

\vspace{0.2cm}

\noindent $\bullet$ second, we evaluate the values of the {\it tuning parameters} in both cases. Table~\ref{tab:tuningparm} displays the values of the tuning parameters $(\lambda,\mu)$ of the S-Lasso, when they are chosen by cross validation $(\lambda^{Cv},\mu^{Cv})$ and based on the theoretical values $(\lambda^{Th},\mu^{Th})$.
We compare them to the values of the parameters $(\lambda^{Est},\mu^{Est})$ that minimize the $\ell_2$ estimation error.

A first remark is that the values of the tuning parameters calibrated based on the theoretical study are always larger than those chosen by cross validation.
This is not surprising since the theoretical calibration of the tuning parameters is fixed to capture smoothness with a large value of $\mu_n$.
It then turns out that the theoretical considerations leads to `smoother' solutions than the cross validation.
Note however that $\lambda^{Th} > \lambda^{Cv}$ does not imply that the solution based on the theoretical issue is sparser since a larger $\mu$ usually implies that the solution is less sparse.\\
Regarding the best solution (where the tuning parameters minimize the $\ell_2$ estimation error), there are two cases.
When the true regression vector is not smooth, it seems that these `best' tuning parameters are closer to the ones chosen by cross validation.
When the true regression vector is smooth, they are closer to the tuning parameters calibrated based on theory.
To sum up, on can say that the best $\lambda$ is close to the one chosen by cross validation, whereas the best $\mu$ is closer to the one based on theory;

\vspace{0.2cm}

\noindent $\bullet$ finally, we compare both of the methods in terms of their {\it estimation accuracy of $\mathbf{J}\beta^*$}. Table~\ref{tab:EstSbeta} summarizes the results. The first four rows displays the median values of $|\mathbf{J}\hat\beta|_2$ when $\hat\beta$ denotes the S-Lasso estimator. We compare the three ways to calibrate the tuning parameters.
We observe that the S-Lasso based on cross validation (S-Lasso$^{Cv}$) provides satisfying estimations of $|\mathbf{J}\beta^*|_2$. We also note that the S-Lasso based on the theoretical values of the tuning parameters (S-Lasso$^{Th}$) is particularly good in Examples~(c) and~(d). This is not surprising since the regression vector in these examples is smooth. It behaves similarly as the best S-Lasso solution (in terms of the minimization of the $\ell_2$ estimation error).\\
Since $\lambda^{Th}$ and $\mu^{Th}$ depend on $|\mathbf{J}\beta^*|_2$ and $|\mathcal{A}^*|$ ({\it cf.} Corollary~\ref{Cor:BorneL1Sparse}), one can intent to use S-Lasso$^{Cv}$ to estimate these two quantities. In this way, one would be able to compute S-Lasso$^{Th}$ even in real dataset experiments.
However, our experiments reveal that S-Lasso$^{Cv}$ may overestimate the number of nonzero components as illustrated by the four last rows of Table~\ref{tab:EstSbeta} (this is also a well-known fact). Nevertheless, we do not exclude this approach, which can be helpful to provide closer performance to those of S-Lasso$^{Est}$.

\begin{table}[t]
\caption{Median values of the tuning parameters $(\lambda,\mu)$ of the S-Lasso for different ways of calibration: `$Cv$' for cross validation; `$Th$' for theoretical choice; `$Est$' for $\ell_2$ estimation error minimizers.
The tuning parameters displayed here correspond to the experiments illustrated in Figure~\ref{fig:BoxplotErrors5}.}
\label{tab:tuningparm}
\begin{center}
\begin{scriptsize}
\begin{sc}
\begin{tabular}{|l||c||c||c||c|}
\hline
Tuning & {\it Ex.}(a)~$[3/0.9]$ & {\it Ex.}(b)~$[100/40/3]$ & {\it Ex.}(c)~$[100/30/3]$  & {\it Ex.}(d)~$[500/100/3]$ \\
\hline
\hline
$ \lambda^{Cv} $   & $  0.4$ & $ 0.5 $  & $ 0.3$  & $1.0 $ \\
\hline
$\mu^{Cv}$   & $  0.0005$   &  $0.0003$  & $ 0.2 $   & $ 0.1  $ \\
\hline
\hline
 $ \lambda^{Th} $   & $ 2.7 $ &  $ 2.8 $ & $ 1.1$   & $ 2.1  $ \\
\hline
$\mu^{Th}$ & $ 0.5110 $ &  $ 1.3100$ & $0.4$  & $1.2 $ \\
\hline
\hline
$ \lambda^{Est} $   &    $ 0.7 $ & $1.0 $ &  $0.3$  & $1.0 $ \\
\hline
$\mu^{Est}$  & $ 0.2500 $ &  $1.2500 $ &  $ 0.3$ & $ 2 .0  $\\
\hline
\end{tabular}
\end{sc}
\end{scriptsize}
\end{center}
\vskip -0.1in
\end{table}
\begin{table}[t]
\caption{Median value of $|\mathbf{J} \hat\beta|_2$ (four first rows) and median number of nonzero components $|\hat{\mathcal{A}}|$ (four last rows) of the S-Lasso for different ways of calibration of the tuning parameters: `$Cv$' for cross validation; `$Th$' for theoretical choice; `$Est$' for $\ell_2$ estimation error minimizers. The third quantiles are displayed in brackets. The values in this table correspond to the experiments illustrated in Figure~\ref{fig:BoxplotErrors5} (and in Table~\ref{tab:tuningparm} as well).}
\label{tab:EstSbeta}
\begin{center}
\begin{scriptsize}
\begin{sc}
\begin{tabular}{|l||c||c||c||c|}
\hline
Tuning & {\it Ex.}(a)~$[3/0.9]$ & {\it Ex.}(b)~$[100/40/3]$ & {\it Ex.}(c)~$[100/30/3]$  & {\it Ex.}(d)~$[500/100/3]$ \\
\hline
\hline
$ |\mathbf{J}\beta^*|_2$   & $ 3.5 $   &  $  3$ & $  2.4 $   & $   2.8  $ \\
\hline
$ |\mathbf{J}\hat{\beta}^{Cv}|_2 $   & $   4.4\,[6.0] $ & $  4.7\,[6.0] $  & $   2.5\,[2.7]  $  & $ 4.0 \,[4.4]     $ \\
\hline
$ |\mathbf{J}\hat{\beta}^{Th}|_2$   & $ 0.9 \,[1.1]$   &  $1.8\,[1.8]$  & $  2.3\,[2.4] $   & $  2.9 \,[2.9]  $ \\
\hline
$ |\mathbf{J}\hat{\beta}^{Est}|_2$   & $  1.8\,[2.7]$   &  $ 1.8 \,[2.6]$  & $  2.3\,[2.4] $   & $ 2.7 \,[2.8]  $ \\
\hline
\hline
$ |\hat{\mathcal{A}}^{*}|   $   & $  3 $   &  $15$& $  15 $   & $   40 $ \\
\hline
$ |\hat{\mathcal{A}}^{Cv}|   $   & $  5 \,[7]$   &  $35\,[41]$& $ 29 \,[33] $   & $   74 \,[82] $ \\
\hline
$ |\hat{\mathcal{A}}^{Th}|   $   & $  6 \,[7] $ & $ 17 \,[19] $  & $ 18 \,[21]  $  & $ 53 \,[57]     $ \\
\hline
$ |\hat{\mathcal{A}}^{Est}|   $   & $  6 \,[7]$   &  $40\,[58]$  & $ 33 \,[37] $   & $  102 \,[113]  $ \\
\hline
\end{tabular}
\end{sc}
\end{scriptsize}
\end{center}
\vskip -0.1in
\end{table}

\vspace{1cm}

\noindent{\bf Conclusion of the experimental results}. The S-Lasso has good performance when the regression vector is `smooth' ({\it Examples}~(c) and (d)). Nevertheless, even in situations made in favor of the Elastic-Net and the Fused-Lasso ({\it Examples}~(b)), the S-Lasso performs similarly as the other methods when the tuning parameters are chosen based on the cross validation criterion. The S-Lasso is even better in these examples when the methods are constructed based on the theoretical considerations.\\
All the results according to the procedures for which the tuning parameters are chosen based on the theoretical perspectives is a little unfair in disfavor of the Fused-Lasso. Indeed, the rates of the tuning parameters have been calibrated based on a study made for the estimator $\hat{\beta}^{Quad}$ (the Elastic-Net and the S-Lasso are two particular cases of this estimator). For the Lasso estimator, we also used the usual rate for $\lambda$. Even if the Fused-Lasso seems to be close to the S-Lasso, it turns out that similar choices for the tuning parameters lead to worse results for the Fused-Lasso.\\
Based on results on {\it Examples}~(c) and (d) it seems that the Fused-Lasso and the Elastic-Net imply a large bias for large values of $\mu$ when the regression vector is smooth (also observed in \cite{Daye09Referee}). They do not improve significantly the performance of the Lasso estimator in such situations. Even the `corrected' Elastic-Net does not provide better results since the artificial correction seems to work for a small number of pairs ($\lambda,\mu$) that have to be chosen very carefully.\\
One can think of two-stage methods to obtain better performance for the Fused-Lasso and the Elastic-Net (and also for the S-Lasso and the Lasso), where for instance an ordinary least squares is fitted based on the estimated support.
This technique reduces of course the bias of the procedures and we refer to~\cite{Chernoz_postl1_10} for a nice theoretical study of such procedures.
However, we attempt here to examine the performance for the (one-stage) methods and observe how well the S-Lasso approaches the true regression vector.

\subsection{Pseudo-real dataset}

We apply all the methods we previously studied on artificially generated dataset from the riboflavin data. These data is about riboflavin (vitamin B2) production by Bacillus subtilis. They kindly have been provided to us by DSM Nutritional Products (Switzerland). In the original data, the real-valued response variable is the logarithm of the riboflavin production rate, and there are $p = 4088$ covariates measuring the logarithm of the expression level of $4088$ genes that cover essentially the whole genome of Bacillus subtilis. The sample size is $n = 71$.\\
Here, we are not interested in the riboflavin production, but only in the covariates matrix $X$ coming from this application. We use this design matrix to generate an artificial response vector with a `smooth' regression vector as in Equation~\eqref{ChapSLeq_depart}. Let us mention that this trick to generate pseudo-real datasets has already been used in \cite{PValues}.
In what follows, we consider two different applications based on the real covariates matrix provided by the riboflavin dataset. In the first application, say {\it Application~1}, let us define  $X$ as the $1023$ first covariates of the riboflavin dataset. Moreover, let us define the regression vector $\beta^*$ such that $\beta_j^* =  10\cdot \exp{-\frac{1}{1-((j-125)/125.1)^2}}$ for $j=1,\ldots,250$ ({\it cf.} Figure~\ref{fig:PseudoRealRegVecteur}) and the noise level $\sigma=3$. Hence, $n=71$ and $p=1023$ and then this is a high-dimensional setting with $p\gg n$ where the number of non-zero components (the sparsity index $|\mathcal{A}^*|$) is larger than the sample size $n$. According to the second application, say {\it Application~2}, we restrict $X$ to the $300$ first covariates of the riboflavin dataset. The regression vector $\beta^*$ is such that $\beta_j^* =  10\cdot \exp{-\frac{1}{1-((j-25)/25.1)^2}}$ for $j=1,\ldots,50$ ({\it cf.} Figure~\ref{fig:PseudoRealRegVecteur}), and the noise level $\sigma=3$. This is a more common high-dimensional case where the sparsity index $|\mathcal{A}^*|$ is smaller than the sample size $n$.

\begin{figure}[t]
\vskip -0.2in
\includegraphics[height=1.65in,width=1.5in] {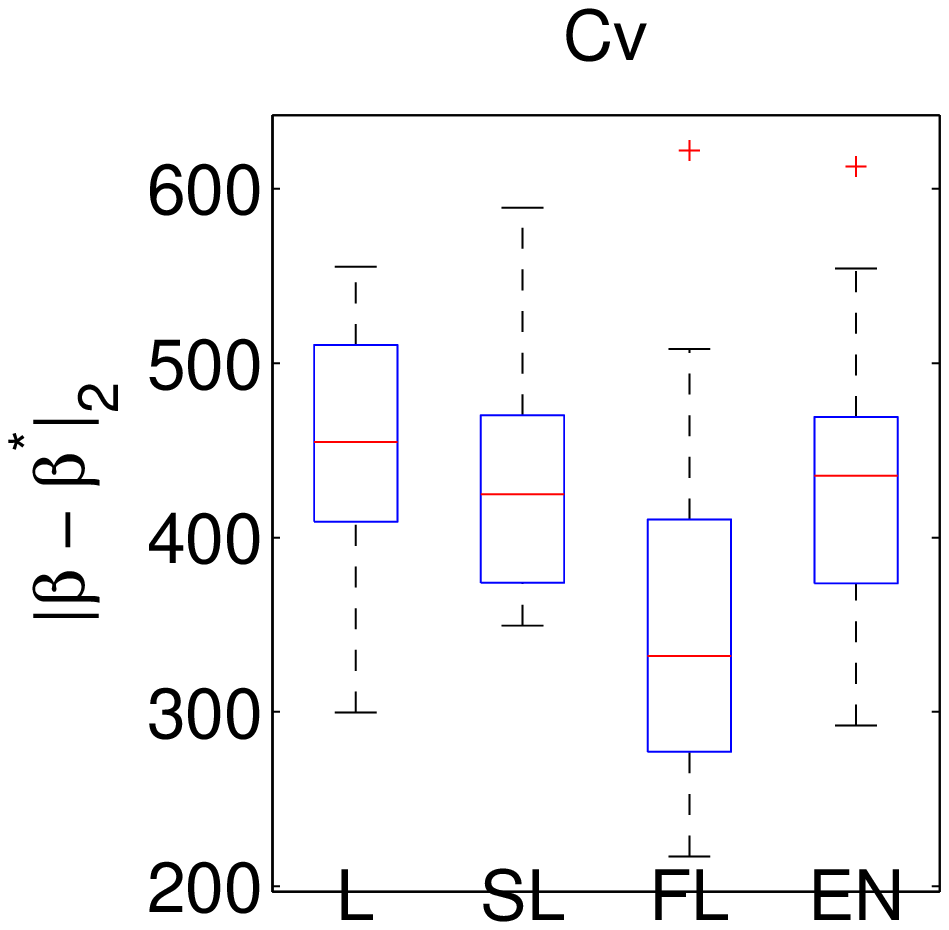}
\includegraphics[height=1.65in,width=1.5in] {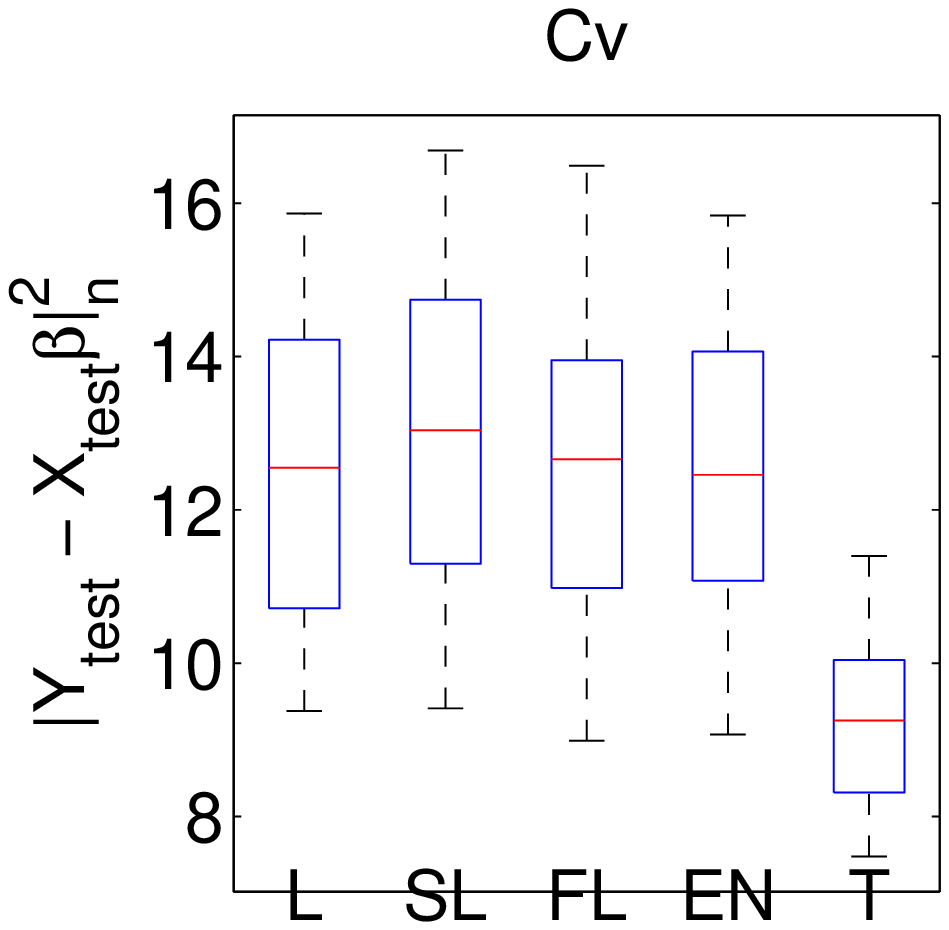}
\includegraphics[height=1.65in,width=1.5in] {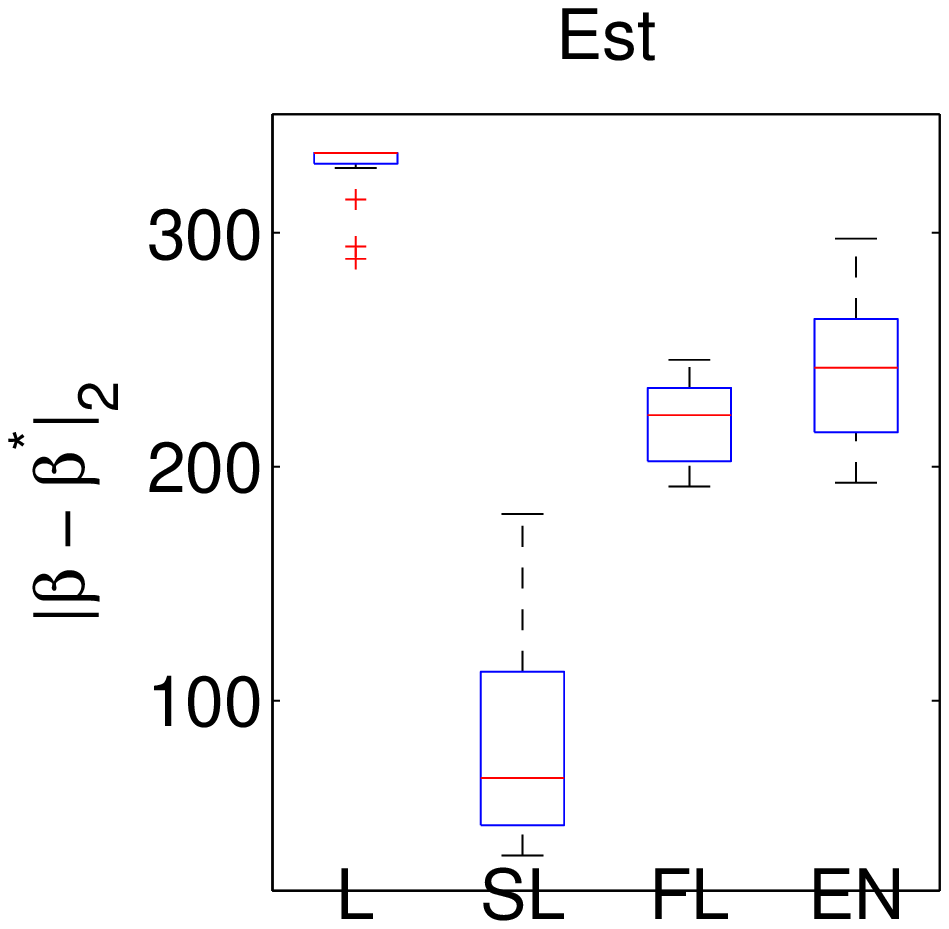}
\includegraphics[height=1.65in,width=1.5in] {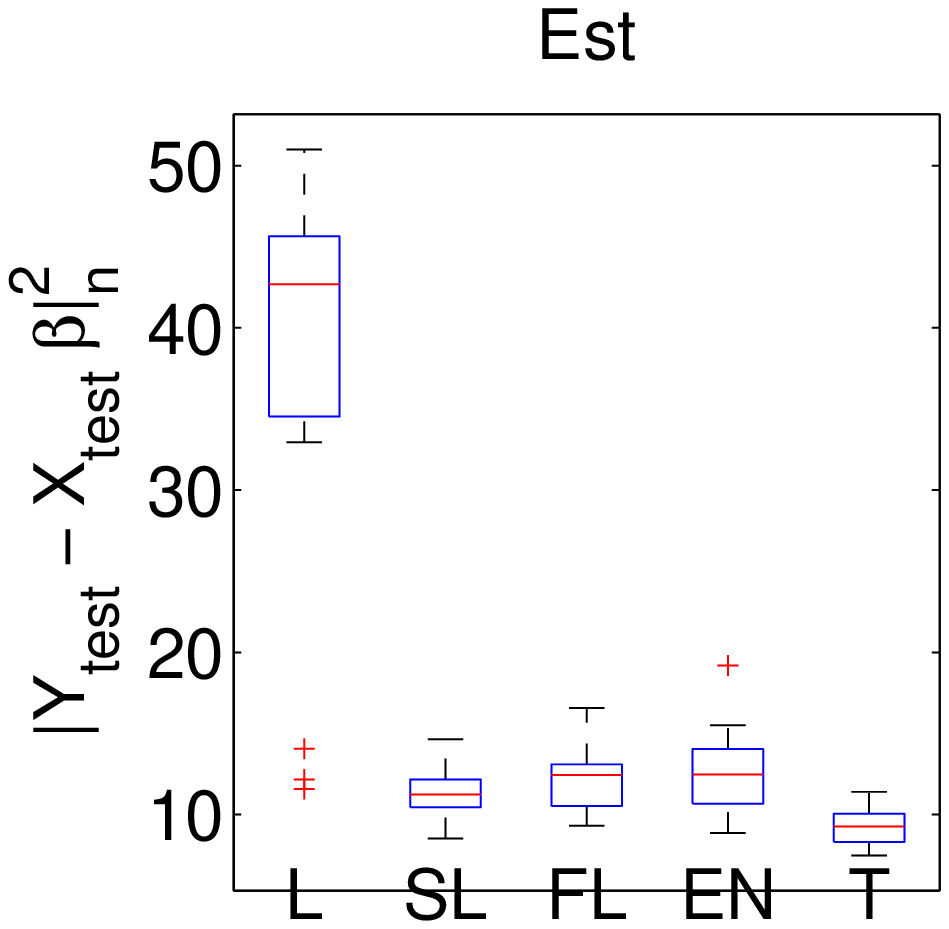}
\vglue-40pt
~
\begin{center}
\begin{minipage}[t]{0.90\textwidth}
\caption{\label{fig:PseudoRealAccuracy}\footnotesize Evaluation of the $\ell_2$ estimation error $|\hat{\beta}-\beta^*|_2$ and the prediction error $\|Y_{test} - X_{test}\hat{\beta}\|_n^2$ of the Lasso (L), the S-Lasso (SL), the Fused-Lasso (FL) and the Elastic-Net (EN) applied to the pseudo-real data, and based on $20$ replications of {\it Application~2}. {\it Left; Center-left}: The tuning parameters are chosen by $10$ fold cross validation.
{\it Center-right; Right}: The tuning parameters minimize the $\ell_2$ estimation error.}
\end{minipage}
\end{center}
\end{figure}
\begin{figure}[t]
\vskip -0.2in
\includegraphics[height=1.9in,width=3in] {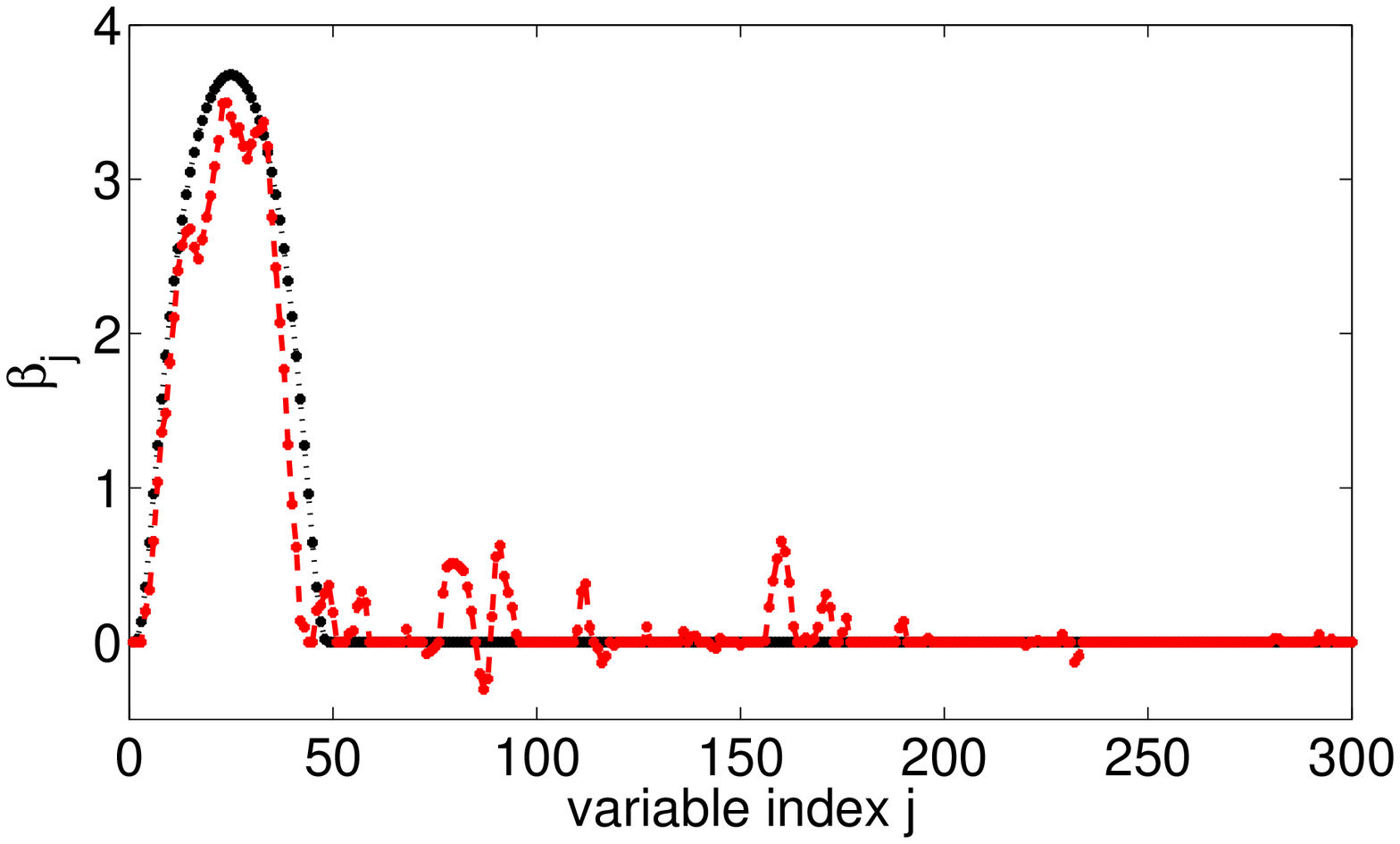}
\includegraphics[height=1.9in,width=3in] {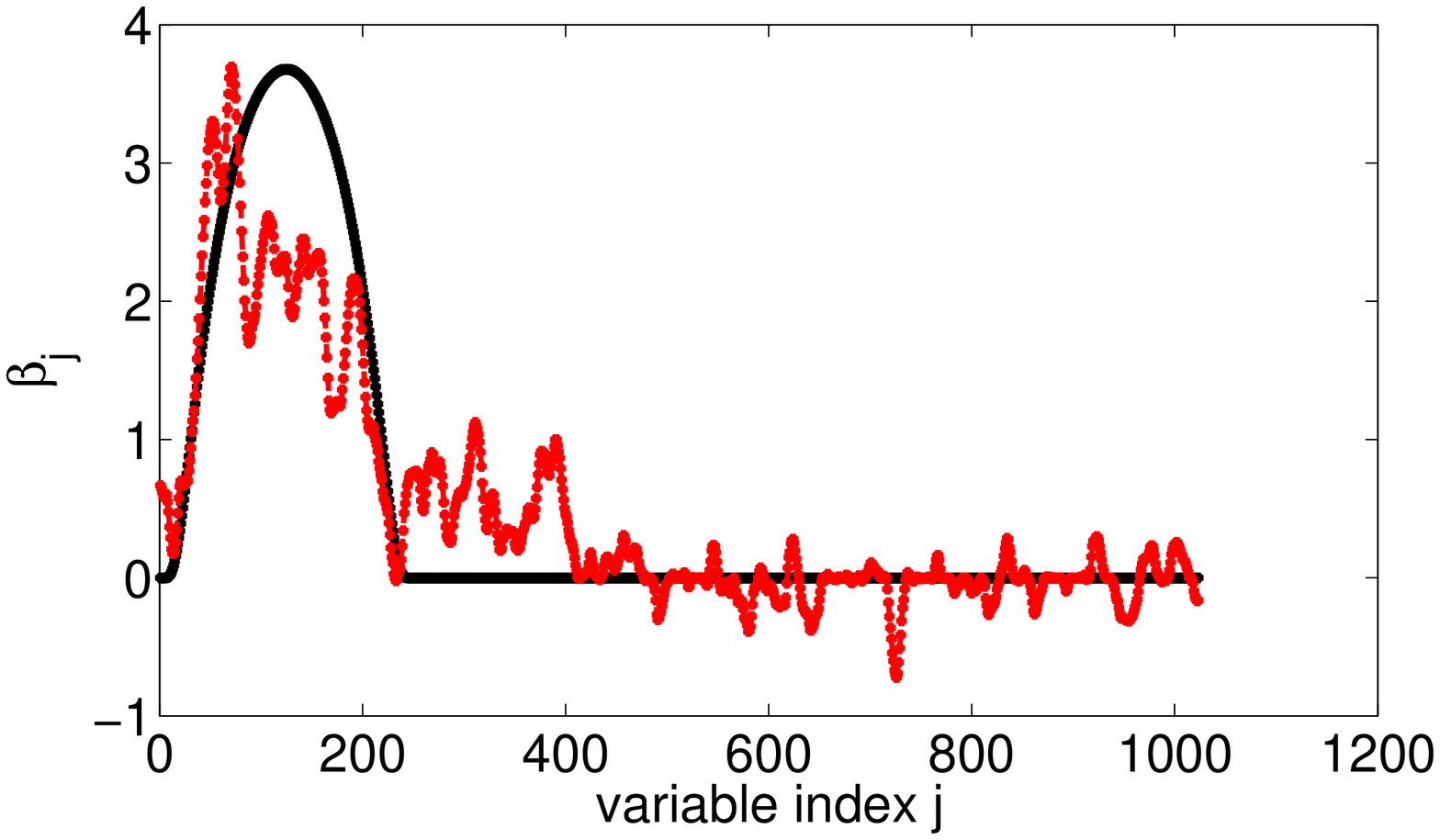}
\vglue-40pt
~
\begin{center}
\begin{minipage}[t]{0.90\textwidth}
\caption{\label{fig:PseudoRealRegVecteur}\footnotesize Best reconstitution of the regression vector $\beta^*$ (black curve) by the SL-Lasso estimator (red curve). {\it Left}: On {\it Application~2}. {\it Right}: On {\it Application~1}.}
\end{minipage}
\end{center}
\end{figure}

Let us now detail the obtained results for different experiments. First, we mention that, with the exception of the S-Lasso, all the methods provide an estimation of the regression vector which is characterized by large variations in the values of the successive components when $\mu$ is small (for the Elastic-Net and the Fused-Lasso) and by large bias when $\mu$ is large. Hence, we focus here on the S-Lasso estimator. Nevertheless, we display the comparison of all the methods in terms of accuracy in Figure~\ref{fig:PseudoRealAccuracy} when the methods are applied to {\it Application~2}. Even though the S-Lasso estimator is outperformed when the tuning parameter is chosen by cross validation (by the Fused-Lasso for the estimation error and by all the methods for the prediction; {\it cf.} Figures~\ref{fig:PseudoRealAccuracy} (left and center-left)), it turns out that we can find a S-Lasso solution which performs better than the other methods as displayed in Figures~\ref{fig:PseudoRealAccuracy} (center-right and right).
One of the best solution of the S-Lasso estimator in {\it Application~2} can also be seen in Figure~\ref{fig:PseudoRealRegVecteur} (left). We observe how the S-Lasso succeeds to reconstruct the `smooth' regression vector $\beta^*$. Before considering {\it Application~1}, we point out one more fact: in both center-right and right plots in Figure~\ref{fig:PseudoRealRegVecteur}, the tuning parameters minimize the $\ell_2$ estimation error. This can provide an explanation of such a bad performance of the Lasso when we consider the prediction error (right plot).
This also implies the big discrepancy between the Lasso based on cross validation (plot center-left) and the one corresponding to the right plot.\\
Finally, let us consider {\it Application~1}, and let us recall that the sparsity index is here larger than the sample size. Figure~\ref{fig:PseudoRealRegVecteur} (right) displays the best reconstitution of the regression vector on this very difficult problem.
We observe that the S-Lasso succeeds only partly to reconstruct the true regression vector. In the simulation study, we met a similar situation with {\it Example}~(d)~[$100/30/3$] ({\it cf.} Figure~\ref{fig:RegrVectSL}), where the S-Lasso perfectly estimated $\beta^*$. However, the situation here is even more difficult since the sparsity index is much larger than the sample size and since many high and negative correlations between the covariates appear in the riboflavin dataset.

\section{Conclusion}

In this paper, we introduced the Lasso-type estimator $\hat{\beta}^{Quad}$ which consists of two penalty terms: a $\ell_1$ penalty term which ensures sparsity and a quadratic penalty term which captures some structure in the regression vector.
We showed that this estimator satisfies good theoretical properties, specifically when the Lasso estimator might fail.
As special cases we considered the Elastic-Net and the S-Lasso.
These methods are interesting in particular when correlations between variables exist or when the regression vector is `smooth'.
We illustrated this in a certain setting and an example where $\hat{\beta}^{Quad}$ performs better than the Lasso.\\
In a concrete survey, we considered the performance of the S-Lasso estimator compared to the Lasso, the Elastic-Net and the Fused-Lasso in terms of prediction and estimation accuracy.
We found the superiority of the S-Lasso in several simulation experiments where the regression vector has a particular structure.
We also observed that the theoretical calibration of the tuning parameters and those obtained by $10$ fold cross validation provide similar performances. The methods have also been applied to {\it pseudo real} examples based on the riboflavin dataset. Finally, we pointed out in several simulation studies (see {\it Example~(d)~[$100/30/\sigma$]}) the ability of the S-Lasso to recover smooth vector even in difficult situations where the sparsity index is larger than the sample size.

\section{Proofs}
\label{sec:Proof}

We first provide two concentration results:
the first one deals with Gaussian noise and the second one concerns noise admitting finite variance.
\begin{lm}
\label{ChapSLprobalm}
	Let $\eta \in (0,1)$. Let $0<\tau \leq 1$, be a real number. Let $\Lambda_{n,p}$ be the random event defined by $\Lambda_{n,p}=\left \lbrace \max_{j=1,\ldots,p} 2|V_j| \le \tau \lambda_n \right \rbrace$ where $V_{j}= n^{-1}\sum_{i=1}^{n}x_{i,j}\varepsilon_{i}$. Let us define $\lambda_n= \frac{2 \sqrt{2} }{\tau} \sigma\sqrt{n^{-1}\log(p/\eta)}$. Then
	$$\mathbb{P}\left(\max_{j=1,\ldots,p} 2 |V_j| \le \tau \lambda_n \right)\geq 1 - \eta.$$
\end{lm}
\begin{proof}
Since $V_{j}\sim
\mathcal{N}(0,n^{-1}\sigma^2)$ for any $j\in\{1,\ldots,p\}$, an elementary Gaussian inequality gives
\begin{eqnarray*}
\mathbb{P}\left( \max_{j=1,\ldots,p} |V_j| \ge \tau \lambda_n /2 \right)
& \le  & p \max_{j=1,\ldots,p} \mathbb{P}\left( |V_j| \ge \tau \lambda_n /2 \right)\\
& \le
& p \exp\left(- \frac{n}{2\sigma^2} \left(\frac{\tau \lambda_n }{2} \right)^2 \right)
=  \eta.
\end{eqnarray*}
This ends the proof.
\end{proof}
\vspace{0.5cm}
\begin{lm}
\label{ChpSLConcentNOnGauss}
	Let $\eta \in (0,1)$. 
	Let $0<\tau \le 1$, be a real number.
	Denote also by $L$ the constant such that $n^{-1}\sum_{i=1}^n \max_{j=1,\ldots,p} x_{i,j}^2 \leq L$.
	Let $\Lambda_{n,p}$ be the random event defined by $\Lambda_{n,p}=\left \lbrace \max_{j=1,\ldots,p} 2|V_j| \le \tau \lambda_n \right \rbrace$ where $V_{j}= n^{-1}\sum_{i=1}^{n}x_{i,j}\varepsilon_{i}$ is such that for any $i = 1,\ldots,n$, $x_{i,j}^2 \leq L$ and the $\varepsilon_{i}$'s are independent random variables with zero mean and finite variance $\mathbb{E}\varepsilon_{i}^2 \leq \sigma^2$.
	Denote by $K_{Nem}$ the quantity $K_{Nem} = \inf_{q \in [2,\infty] \cap \mathbb{R}} (q-1) p^{2/q}$.
	Then for $\lambda_n=\frac{2\sigma}{\tau} \sqrt{\frac{K_{Nem} L}{n\eta} }$, we have
	$$\mathbb{P}\left(\max_{j=1,\ldots,p} 2 |V_j| \le \tau \lambda_n \right)
	\geq
	1 - \eta.$$
\end{lm}
\begin{proof}
This inequality uses an inequality on the expectation of supremum of square of sum of independent random variables that can be found in~\cite[Theorem 2.2]{DvGVW10nemirovkiInequality}.
Let us mention that $2 e \log (p) -3e < K_{Nem} < 2 e \log (p) -e$.
Markov Inequality and Theorem 2.2 in~\cite{DvGVW10nemirovkiInequality} (with $r=\infty$) imply
\begin{eqnarray}
\label{eqProof:QFDs}
	\mathbb{P}\left( \max_{j=1,\ldots,p} |V_j| \ge \tau \lambda_n /2 \right)
	& \le  & 
	\frac{4}{\tau^2 \lambda_n^2 }  \mathbb{E}\left( \max_{j=1,\ldots,p} V_j^2 \right) 
	\nonumber
	\\
	& \le & 
	\frac{4 K_{Nem}}{\tau^2 \lambda_n^2 } \sum_{i=1}^n \mathbb{E} \left( \max_{j=1,\ldots,p} n^{-2}x_{i,j}^2 \varepsilon_i^2 \right)
	\\
	& \leq &
	\frac{4 \sigma^2 K_{Nem}}{\tau^2 n\lambda_n^2 }  n^{-1}\sum_{i=1}^n \max_{j=1,\ldots,p} x_{i,j}^2 
	\leq 
	\eta,
	\nonumber 
\end{eqnarray}
where we used the definition of $\lambda_n=\frac{2\sigma}{\tau} \sqrt{\frac{K_{Nem} L}{n\eta} }$ in the last inequality. Theorem 2.2 in~\cite{DvGVW10nemirovkiInequality} is used to obtain~\eqref{eqProof:QFDs}.
\end{proof}
\vspace{0.5cm}
\begin{proof} [\it{Proof of Theorem~\ref{ChapSLThm:doubkeSparsJ}}] We provide a first result which may help the legibility of the paper. It states that the squared risk and the $\ell_1$-estimation error are controlled by the restricted $\ell_2$-estimation error $|\beta_{\mathcal{B}}^* -\hat{\beta}_{\mathcal{B}}^{Quad} |_2$.
\begin{prp}
\label{ChapSLProofProp:bornintermen}
	Let $\hat{\beta}^{Quad}$ be the estimator defined by~\eqref{ChapSLeq:penalized-risk}-\eqref{ChapSLcritere_S-lassoGneneram} with tuning parameters $\lambda_n$ and $\mu_n$. Let $0<\tau \leq  1$ be a real number. On the event $\Lambda_{n,p}=\left \lbrace \max_{j=1,\ldots,p} 2|V_j| \le \tau \lambda_{n} \right \rbrace$ with $V_{j}= n^{-1}\sum_{i=1}^{n}x_{i,j}\varepsilon_{i}$, if $\tau = 1/2$ we have
	\begin{equation}
	\label{eq:resultIntermed}
		\frac{1}{n}\left|\widetilde{X} \beta^* -\widetilde{X} \hat{\beta}^{Quad} \right|_{2}^{2}
		+
		\frac{\lambda_{n} }{2}  |\beta^* - \hat{\beta}^{Quad}|_1
		\leq
		r_n |\beta_{\mathcal{B}}^* -\hat{\beta}_{\mathcal{B}}^{Quad} |_2,
	\end{equation}
where $ r_n= 2 \lambda_{n} \sqrt{|\mathcal{A}^*|} +  2\mu_n | \widetilde{J}\beta^*|_2 $, and $\mathcal{B}$ is a set including $\mathcal{A}^*$.
\end{prp}
\begin{proof} Let first $\widetilde{X},\,\widetilde{Y}$ and $\widetilde{\varepsilon}$ be the augmented dataset defined by
$$\widetilde{X}= \begin{pmatrix}
X \\
\sqrt{n \mu_n}\mathbf{J}
\end{pmatrix}, \quad
\text{and} \quad \widetilde{Y}=
\begin{pmatrix}
Y \\
\mathbf{0}
\end{pmatrix},
 \quad
\text{and} \quad \widetilde{\varepsilon}=
\begin{pmatrix}
\varepsilon \\
- \sqrt{n \mu_n}\mathbf{J}\beta^*
\end{pmatrix},$$
where $\mathbf{0}$ is a vector of size $p$ containing only zeros and $\mathbf{J}$ is the $p\times p$ matrix given by~\eqref{ChapSLJmatrix}. Then we have $\widetilde{Y} = \widetilde{X}\beta^* + \widetilde{\varepsilon},$ and the estimator $\hat{\beta}^{Quad}$, solution of the minimization problem~\eqref{ChapSLeq:penalized-risk} with the penalty given by~\eqref{ChapSLcritere_S-lassoGneneram}, is also the minimizer of
\begin{equation*}
	\frac{1}{n}\left|\widetilde{Y}-\widetilde{X} \beta \right|_{2}^{2}+ \lambda_n | \beta |_{1}.
\end{equation*}
Hence, by definition of the estimator $\hat{\beta}^{Quad}$ we can write
\begin{eqnarray*}
	& \frac{1}{n}\left|\widetilde{Y}-\widetilde{X} \hat{\beta}^{Quad} \right|_{2}^{2}+ \lambda_n | \hat{\beta}^{Quad}|_{1} 
	\leq 
	\frac{1}{n}\left|\widetilde{Y}-\widetilde{X} \beta^* \right|_{2}^{2}+ \lambda_n | \beta^* |_{1}  \\
	\iff & \frac{1}{n}\left|\widetilde{X} \beta^* -\widetilde{X} \hat{\beta}^{Quad} +  \widetilde{\varepsilon} \right|_{2}^{2} - \frac{1}{n}\left| \widetilde{\varepsilon} \right|_{2}^{2}
	\leq
	\lambda_n | \beta^* |_{1}  - \lambda_n | \hat{\beta}^{Quad}|_{1} \\
	\iff & \frac{1}{n}\left|\widetilde{X} \beta^* -\widetilde{X} \hat{\beta}^{Quad} \right|_{2}^{2}
	\leq
	\lambda_n \left[ | \beta^* |_{1}  -  | \hat{\beta}^{Quad}|_{1} \right] + \frac{2}{n} \widetilde{\varepsilon}' \widetilde{X}  (\beta^* -\hat{\beta}^{Quad} ).
\end{eqnarray*}
Let us now consider the term $\frac{2}{n}  \widetilde{\varepsilon}' \widetilde{X} (\beta^* -\hat{\beta}^{Quad} )$. By the definition of $\widetilde{X}$ and $\widetilde{\varepsilon}$, we have the decomposition $\frac{1}{n}  \widetilde{\varepsilon}' \widetilde{X}(\beta^* -\hat{\beta}^{Quad} ) = \frac{1}{n}  \varepsilon' X(\beta^* -\hat{\beta}^{Quad} ) - \mu_n {\beta^*}' \mathbf{J}'  \mathbf{J}(\beta^* -\hat{\beta}^{Quad} )$. The first term in this decomposition is quite common in the literature and we treat it using arguments which can be found for instance in~\cite{Lasso2}. We then need to adapt those arguments in order to deals with the second term of the decomposition $\mu_n {\beta^*}' \mathbf{J} ' \mathbf{J} (\beta^* -\hat{\beta}^{Quad} )$ in the same time. Recall that $\mathcal{A}^{*} = \{j:\,\beta_j^*\neq 0\}$ and that $\mathbf{J} ' \mathbf{J} = \widetilde{J}$. Let $0<\tau \leq  1$ be a real number. Then, on the event $\Lambda_{n,p}=\left \lbrace \max_{j=1,\ldots,p} 2|V_j| \le \tau \lambda_{n} \right \rbrace$ with $V_{j}= n^{-1}\sum_{i=1}^{n}x_{i,j}\varepsilon_{i}$, we have
\begin{eqnarray}
\label{inqprelimin}
	\frac{1}{n}\left|\widetilde{X} \beta^* -\widetilde{X} \hat{\beta}^{Quad} \right|_{2}^{2}
	& \leq &
	\lambda_n \left[ | \beta^* |_{1}  -  | \hat{\beta}^{Quad}|_{1} \right] +
	\tau \lambda_{n} |\beta^* - \hat{\beta}^{Quad}|_1 \nonumber \\
	& &
	- 2 \mu_n {\beta^*}' \widetilde{J} (\beta^* -\hat{\beta}^{Quad} ).
\end{eqnarray}
The remainder of this proof is linked to the way we choose to treat the term $\mu_n {\beta^*}' \widetilde{J} (\beta^* -\hat{\beta}^{Quad} )$ and in particular in the way we choose to link the RHS of Inequality~\eqref{inqprelimin} to the quantity $|\beta_{\mathcal{A}^*}^* - \hat{\beta}_{\mathcal{A}^*}^{Quad} |_2$. We obviously can write
\begin{equation*}
	- \mu_n {\beta^*}' \widetilde{J} (\beta^* -\hat{\beta}^{Quad} )
	=
	- \mu_n {\beta_{\mathcal{B}}^*}' \widetilde{J} (\beta_{\mathcal{B}}^* -\hat{\beta}_{\mathcal{B}}^{Quad} )
	\leq
	\mu_n | \widetilde{J}\beta^*|_2  |\beta_{\mathcal{B}}^* -\hat{\beta}_{\mathcal{B}}^{Quad} |_2,
\end{equation*}
where $\mathcal{B}$ is the smallest set of indices such that the first equality holds. Note that the set $\mathcal{B}$ includes $\mathcal{A}^*$, the true sparsity set, and is not much larger due to the sparsity of $\widetilde{J}$.\\
\noindent Now let $\tau = 1/2$ in~\eqref{inqprelimin}, add $2^{-1} \lambda_{n}  |\beta^* - \hat{\beta}^{Quad}|_1$ to both sides of this inequality. We then get
\begin{eqnarray}
\label{SLeqProof:eqprirlo}
	\frac{1}{n}\left|\widetilde{X} \beta^* -\widetilde{X} \hat{\beta}^{Quad} \right|_{2}^{2}
	+
	\frac{\lambda_{n} }{2}  |\beta^* - \hat{\beta}^{Quad}|_1
	& \leq &
	\lambda_n \left[ | \beta^* |_{1}  -  | \hat{\beta}^{Quad}|_{1} + |\beta^* - \hat{\beta}^{Quad}|_1 \right] 
	\\
	& &
	+  2 \mu_n |\widetilde{J}\beta^*|_2  |\beta_{\mathcal{B}}^* -\hat{\beta}_{\mathcal{B}}^{Quad} |_2 
	\nonumber
	\\
	& \leq &	
	2 \lambda_{n} \sum_{j\in \mathcal{A}} \abs{\beta^*_{j} - \hat{\beta}_{j}^{Quad}} 
	+  2 \mu_n |\widetilde{J}\beta^*|_2  |\beta_{\mathcal{B}}^* -\hat{\beta}_{\mathcal{B}}^{Quad} |_2
	\nonumber
	\\
	& \leq &
	r_n  |\beta_{\mathcal{B}}^* -\hat{\beta}_{\mathcal{B}}^{Quad} |_2,
	\nonumber
\end{eqnarray}
where $ r_n= 2 \lambda_{n} \sqrt{|\mathcal{A}^*|} + 2 \mu_n | \widetilde{J}\beta^*|_2 $, since $|\beta_{\mathcal{A}^*}^* -\hat{\beta}_{\mathcal{A}^*}^{Quad} |_1
\leq
\sqrt{|\mathcal{A}^*|} |\beta_{\mathcal{A}^*}^* -\hat{\beta}_{\mathcal{A}^*}^{Quad} |_2
\leq
\sqrt{|\mathcal{A}^*|} |\beta_{\mathcal{B}}^* -\hat{\beta}_{\mathcal{B}}^{Quad} |_2$. In the second above inequality, we used the fact that $|\beta_{j}^{*} - \hat{\beta}_{j}^{Quad} | + |\beta_{j}^{*}| - |\hat{\beta}_{j}^{Quad} |=0 $ for any $j\notin \mathcal{A}$ and to the triangular inequality. This is the claim of Proposition~\ref{ChapSLProofProp:bornintermen} when $\widetilde{J}$ is sparse.
\end{proof}
Let us now proof the main theorem. Thanks to Inequality~\eqref{eq:resultIntermed} in Proposition~\ref{ChapSLProofProp:bornintermen}, we easily obtain that 
\begin{equation}
\label{ChapSLassoeq:ProofL2borneL1}
	|\beta^* - \hat{\beta}^{Quad}|_1
	\leq
	\varrho_n |\beta_{\mathcal{B}}^* -\hat{\beta}_{\mathcal{B}}^{Quad} |_2,
\end{equation}
where $ \varrho_n := 2 r_n/ \lambda_n=  4 \sqrt{|\mathcal{A}^*|} +  \frac{4\mu_n}{\lambda_n} | \widetilde{J}\beta^*|_2 $. Then the vector $\beta^* -\hat{\beta}^{Quad}$ is an admissible vector $\Delta$ in Assumption~$B(\mathcal{B})$. As a consequence, using this assumption in Equation~\eqref{eq:resultIntermed}, we get
on one hand
\begin{equation*}
	\frac{1}{n}\left|\widetilde{X} \beta^* -\widetilde{X} \hat{\beta}^{Quad} \right|_{2}^{2}
	\leq
	\frac{r_n }{\sqrt{\phi_{\mu_n}}}	 \sqrt{\frac{1}{n}}\left|\widetilde{X} \beta^* -\widetilde{X} \hat{\beta}^{Quad} \right|_{2},
\end{equation*}
and a simple simplification leads to the first part of the result
\begin{equation}
\label{ChapSlassoeq:proofBorneSurPredicfin}
	\frac{1}{n}\left|\widetilde{X} \beta^* -\widetilde{X} \hat{\beta}^{Quad} \right|_{2}^{2}
	\leq
	 \phi_{\mu_n}^{-1} ( 2 \lambda_n \sqrt{|\mathcal{A}^*| }  + 2 \mu_n |\widetilde{J}\beta^*|_2 )^2  .
\end{equation}
On the other hand, Inequality~\eqref{ChapSLassoeq:ProofL2borneL1}, combined to Assumption~$B(\mathcal{B})$ and Inequality~\eqref{ChapSlassoeq:proofBorneSurPredicfin}, implies
\begin{equation*}
	|\beta^* - \hat{\beta}^{Quad}|_1
	\leq
	2\phi_{\mu_n}^{-1} \frac{( 2 \lambda_n \sqrt{|\mathcal{A}^*| }  + 2 \mu_n |\widetilde{J}\beta^*|_2 )^2}{  \lambda_n},
\end{equation*}
which is the desired bound on the $\ell_1$ estimation error given in Theorem~\ref{ChapSLThm:doubkeSparsJ}. The proof is completed when we use Lemma~\ref{ChapSLprobalm} with $\tau = 1/2$ to control the probability of the event $\Lambda_{n,p}$.
\end{proof}
\vspace{0.5cm}
\begin{proof} [\it{Proof of Proposition~\ref{ChapSLprop:SupNormm}}]
We first provide a bound on $|\beta_{\Theta}^* - \hat{\beta}_{\Theta}^{Quad}|_2$ for  $\Theta = \mathcal{B}\cup\mathcal{C}$.
Theorem~\ref{ChapSLThm:doubkeSparsJ} states a bounds on the prediction error and on the $\ell_1$ estimation error under Assumption~$B(\mathcal{B})$.
Here we do not care about the $\ell_1$ estimation error. Then one can observe that in the intermediate step between~\eqref{inqprelimin} and~\eqref{SLeqProof:eqprirlo} in the previous proof, one can avoid the addition of the term $\lambda_n/2 |\beta^{Quad} - \beta^*|_1$. As a consequence, we obtain~\eqref{ChapSlassoeq:proofBorneSurPredicfin} but with $\tau = 1$ instead of $1/2$ in~\eqref{inqprelimin}. Apart from this value of $\tau$ everything remains the same.

More particularly, thanks to~\eqref{ChapSLassoeq:ProofL2borneL1} we can use Assumption~$B'(\mathcal{B}\cup \mathcal{C})$, which directly implies that the following inequality holds
$|\beta_{\Theta}^* - \hat{\beta}_{\Theta}^{Quad}|_2 \leq  \sqrt{\phi_{\mu_n}^{-1}} \sqrt{ \frac{1}{n} }\left|\widetilde{X} \beta^* -\widetilde{X} \hat{\beta}^{Quad} \right|_{2} $, with $\Theta = \mathcal{B}\cup\mathcal{C}$.
Combining this inequality with~\eqref{ChapSlassoeq:proofBorneSurPredicfin}, we easily get
\begin{equation}
\label{ChapSLProofEq:ControlEstinErrorSurSupport2}
	|\beta_{\Theta}^* - \hat{\beta}_{\Theta}^{Quad}|_2
	\leq
	2 \phi_{\mu_n}^{-1} (  \lambda_n \sqrt{|\mathcal{A}^*| } 
	+ \mu_n |\widetilde{J}\beta^*|_2 ),
\end{equation}
with $\Theta = \mathcal{B}\cup\mathcal{C}$.
Now, we consider the term $|\beta_{\Theta^c}^* - \hat{\beta}_{\Theta^c}^{Quad}|_2$.
Denote by $\delta$ the vector $\delta = \beta^* - \hat{\beta}^{Quad}$ for shorten.
For any $p$-dimensional vector $a$, let $a_{(1)}\leq a_{(2)}\leq \ldots a_{(p)} $ be the corresponding ranked sequence.
Given this new notation, note that for any $j \in [1,\ldots,p]$, the inequality $|\delta_{\mathcal{B}^c}|_{(j)} \leq |\delta_{\mathcal{B}^c}|_1 \times j^{-1}$ holds.
As a consequence
\begin{equation*}
	|\delta_{\Theta^c}|_2^2 \leq |\delta_{\mathcal{B}^c}|_1^2 \sum_{j\geq m+1} j^{-2} 
				\leq
				m^{-1} |\delta_{\mathcal{B}^c}|_1^2 ,
\end{equation*}
where we recall that $\Theta = \mathcal{B}\cup\mathcal{C}$, with $|\mathcal{B}|=m$.
Then using the last display with~\eqref{ChapSLassoeq:ProofL2borneL1} yields to
\begin{equation*}
	|\delta_{\Theta^c}|_2 \leq \frac{\varrho_n}{\sqrt{m}}  |\delta_{\mathcal{B}}|_2 
			      \leq 
				\frac{\varrho_n}{\sqrt{m}}  |\delta_{\Theta}|_2, 
\end{equation*}
where $\varrho_n = 4 \sqrt{|\mathcal{A}|} + \frac{4\mu_n}{\lambda_n} |\widetilde{J}\beta^*|_2$.
Combine this last inequality with~\eqref{ChapSLProofEq:ControlEstinErrorSurSupport2} implies
\begin{equation*}
	|\delta|_2 \leq (1+ \frac{\varrho_n}{\sqrt{m}} ) |\delta_{\Theta}|_2 
	\leq
	 2 \phi_{\mu_n}^{-1} (1+ \frac{\varrho_n}{\sqrt{m}} )  ( \lambda_n \sqrt{|\mathcal{A}^*| } 
	+ \mu_n |\widetilde{J}\beta^*|_2 ).
\end{equation*}
Since $|\delta|_{\infty} \leq |\delta|_2$, we obtained the desired control on the sup-norm of $ \beta^* - \hat{\beta}^{Quad} $.
\end{proof}
\vspace{0.5cm}
\begin{proof}[\it{Proof of Theorem~\ref{ChapSLthm:VarSelec}}]
This result is quite natural since it is a direct consequence of Proposition~\ref{ChapSLprop:SupNormm}.
We refer the reader to the proof of Theorem~2 in~\cite{KarimNormSup} for instance.
\end{proof}
\vspace{0.5cm}
\begin{proof} [\it{Proof of Theorem~\ref{ChapSLThm:doubke}}]

We consider now the case of general matrices $\widetilde{J}$. Most of the proof is similar to the sparse case (Proof of Theorem~\ref{ChapSLThm:doubkeSparsJ} above). The same reasoning leads to~\eqref{inqprelimin} and the only different occurs when we deal with the term $- \mu_n {\beta^*}' \widetilde{J} (\beta^* -\hat{\beta}^{Quad} )$. We have here
\begin{equation*}
	- \mu_n {\beta^*}' \widetilde{J}(\beta^* -\hat{\beta}^{Quad} )
	\leq
	\mu_n | \widetilde{J} \beta^*|_{\infty}  |\beta^* -\hat{\beta}^{Quad} |_1.
\end{equation*}
Then, if we set $\tau = \frac{1}{4}$ and the tuning parameter $\mu_n = \frac{\lambda_n}{8 | \widetilde{J}\beta^*|_{\infty} }$, Inequality~\eqref{inqprelimin} becomes
\begin{equation*}
	\frac{1}{n}\left|\widetilde{X} \beta^* -\widetilde{X} \hat{\beta}^{Quad} \right|_{2}^{2}
	\leq
	\lambda_n \left[ | \beta^* |_{1}  -  | \hat{\beta}^{Quad}|_{1} \right] +
	\frac{\lambda_{n}}{2} |\beta^* - \hat{\beta}^{Quad}|_1.
\end{equation*}
Add $2^{-1} \lambda_{n}  |\beta^* - \hat{\beta}^{Quad}|_1$ to both sides of the previous inequality and then thanks to the fact that $|\beta_{j}^{*} - \hat{\beta}_{j}^{Quad} | + |\beta_{j}^{*}| - |\hat{\beta}_{j}^{Quad} |=0 $ for any $j\notin \mathcal{A}^*$ and to the triangular inequality, the above inequality implies that (we refer to the proof of Proposition~\ref{ChapSLProofProp:bornintermen} for similar arguments).
\begin{equation*}
	\frac{1}{n}\left|\widetilde{X} \beta^* -\widetilde{X} \hat{\beta}^{Quad} \right|_{2}^{2}
	+
	\frac{\lambda_{n} }{2}  |\beta^* - \hat{\beta}^{Quad}|_1
	\leq
	2 \lambda_{n} \sqrt{|\mathcal{A}^*|}  |\beta_{\mathcal{A}^*}^* -\hat{\beta}_{\mathcal{A}^*}^{Quad} |_2.
\end{equation*}
%
%
%
%
This above intermediate result is the analogous of Proposition~\ref{ChapSLProofProp:bornintermen} in the case where $\widetilde{J}$ is general. That is, we get a similar bound but depending on $|\beta_{\mathcal{A}^*}^* -\hat{\beta}_{\mathcal{A}^*}^{Quad} |_2$ instead of $|\beta_{\mathcal{B}}^* -\hat{\beta}_{\mathcal{B}}^{Quad} |_2$ and with $r_n =2 \lambda_{n} \sqrt{|\mathcal{A}^*|}$.  Note also that~\eqref{ChapSLassoeq:ProofL2borneL1} is replaced by the following linear inequality
$|\beta^* - \hat{\beta}^{Quad}|_1
	\leq
	4 |\beta_{\mathcal{A}}^* -\hat{\beta}_{\mathcal{A}}^{Quad} |_1$. Taking into account this changing, we use can use Assumption~RE instead of Assumption~$B(\mathcal{B})$ and then a similar reasoning as in the proof of Theorem~\ref{ChapSLThm:doubkeSparsJ} leads to the desired results.
\end{proof}
\vspace{0.5cm}
\begin{proof} [\it{Proof of Proposition~\ref{ChapSLprp:supNormParl2}}]
Using exactly the same reasoning as in the proof of Proposition~\ref{ChapSLprop:SupNormm} but based on Theorem~\ref{ChapSLThm:doubke} instead of Theorem~\ref{ChapSLThm:doubkeSparsJ} we obtain with probability at least $1- \eta$
\begin{equation}
\label{ChapSLProofEq:ControlEstinErrorSurSupport}
	|\beta_{\mathcal{A}^*}^* - \hat{\beta}_{\mathcal{A}^*}^{Quad}|_2
	\leq
	2 \phi_{\mu_n}^{-1}    \lambda_n  \sqrt{|\mathcal{A}^*| } ,
\end{equation}
since here $\tau$ becomes equal to $ 1/2$ in Lemma~\ref{ChapSLprobalm}. This completes the proof of the first part of the Proposition.
\begin{equation*}
	|\hat{\beta}_{\mathcal{A}^*}^{Quad} -\beta_{\mathcal{A}^*}^* |_{\infty} \leq U \quad
	 \Leftrightarrow \quad
	\beta_j^* -U \leq \hat{\beta}_j^{Quad} \leq \beta_j^* + U  \quad \forall j\in \mathcal{A}^* .
\end{equation*}
Note that by assumption, we have $|\beta_j^*| > U,\,\forall j\in \mathcal{A}^*$. Then if we distinguish the case $\beta_j^* > 0$ and the case $\beta_j^* < 0$, we easily conclude that $\beta_j^* > 0$ implies $\hat{\beta}_j^{Quad} > 0$ and $\beta_j^* < 0$ implies $\hat{\beta}_j^{Quad} < 0$. This ables us to write
\begin{equation*}
	\mathbb{P}(\mathop{\rm Sgn}(\hat{\beta}_{\mathcal{A}^*}^{Quad}) = \mathop{\rm Sgn}(\beta_{\mathcal{A}^*}^*)) 
	\geq
	\mathbb{P}(	|\hat{\beta}_{\mathcal{A}^*}^{Quad} -\beta_{\mathcal{A}^*}^* |_{\infty} \leq U )
	\geq 
	1 - \eta,
\end{equation*}
and this naturally implies the that $\mathcal{A}^* \subset \hat{\mathcal{A}}$ with high probability.
\end{proof}
\vspace{0.5cm}
\begin{proof} [\it{Proof of Theorem~\ref{ChapSLprp:ConsistAvecSign}}]
We now show that $ \hat{\mathcal{A}} \subset \mathcal{A}^* $ with high probability. This proof is quite inspired by the one by Bunea~\cite{BuneaEN}. First of all, note that we can write the KKT conditions of the minimization problem~\eqref{ChapSlassEq:CritAugme} as
\begin{equation}
\label{ChapSLeq:KKTq2}
	|K_n(\hat{\beta}^{Quad} - \beta^{*}) -\frac{X'\varepsilon}{n} + \mu_n \widetilde{J} \beta^* |_{\infty}
	\leq
	\frac{\lambda_n}{2}.
\end{equation}
Then all the solutions of the criterion~\eqref{ChapSlassEq:CritAugme} share the same active set
\begin{equation*}
	\hat{\mathcal{A}}
	=
	\left\{
		j\in\{1,\ldots,p\}:\,\, |(K_n(\hat{\beta}^{Quad} - \beta^{*}))_j - \frac{X_j'\varepsilon}{n} + \mu_n (\widetilde{J} \beta^*)_j |
		=
		\frac{\lambda_n}{2}
	\right\}.
\end{equation*}
That is, all these solutions have non-zero components at the same positions. We now use this property to show that the estimator $\hat{\beta}^{Quad}$ has non-zero components at the same positions as a well-controlled (but uncomputable) estimator on an event which occurs with high probability. For this purpose, let us consider the criterion
\begin{equation*}
	F(b)
	=
	\Vert Y- \sum_{j\in \mathcal{A}^*} X_j b_j \Vert_n^2
		+
	\lambda_n \sum_{j\in \mathcal{A}^*} | b_j |
		+
	\mu_n b_{\mathcal{A}^*}' \mathbf{J}_{\mathcal{A}^*}' \mathbf{J}_{\mathcal{A}^*} b_{\mathcal{A}^*},
\end{equation*}
where recall that for any $p$-dimensional vector $a$ and any set $\Theta \subset \{1,\ldots,p\}$, the notation $a_{\Theta}$ means that $(a_{\Theta})_j = a_j, \forall j\in \Theta $ and $0$ otherwise. Moreover, $\mathbf{J}_{\mathcal{A}^*}$ is such that $(\mathbf{J}_{\mathcal{A}^*})_{j,k} = \mathbf{J}_{j,k}$ if ${j,k}\in  \mathcal{A}^*$ and $0$ otherwise. Define the estimator
\begin{equation*}
	\hat{b}
	=
	\underset{b \in \mathbb{R}^{p} : \, b_{(\mathcal{A}^*)^c} = \mathbf{0}_p} {\argmin} F(b),
\end{equation*}
where $\mathbf{0}_p$ is the zero in $\mathbb{R}^p$. Since we restricted $\hat{b}$ to be zero when $\beta^*$ is zero and that this is an information we do not have access to, we mention that the vector is not computable. Let us denote by $\Omega$ the following event
\begin{equation*}
	\Omega
	=
	\bigcap_{k\notin \mathcal{A}^*}
	\left\{
		\left|
			\sum_{j\in \mathcal{A}^*} (K_n)_{j,k} (\hat{b}_j - \beta_j^{*}) - \frac{X_k'\varepsilon}{n} + \mu_n \sum_{j\in \mathcal{A}^*} \widetilde{J}_{j,k} \beta_j^*
		\right|
		<
		\frac{\lambda_n}{2}
	\right\}.
\end{equation*}
Observe how the event $\Omega$ is inspired by the KKT conditions~\eqref{ChapSLeq:KKTq2}. Actually, on the event $\Omega$, the components $\hat{b}_k$ with $k \notin\mathcal{A}^* $ equals zero as they do not saturate KKT conditions. This makes the minimization of $F(b)$ over $b \in \mathbb{R}^{p} : \, b_{(\mathcal{A}^*)^c} = \mathbf{0}_p$ coincide with the minimization of the criterion~\eqref{ChapSlassEq:CritAugme} on $\Omega$. That is, the estimator $\hat{b}$ turns out to be also solution of the original criterion~\eqref{ChapSlassEq:CritAugme} on $\Omega$. But $\hat{\beta}^{Quad}$ is also solution of~\eqref{ChapSlassEq:CritAugme} and then, as we already pointed, this implies that on $\Omega$, both of $\hat{\beta}^{Quad}$ and $\hat{b}$ have non-zero components at the same positions and then, $\hat{b}$ has non-zero components at components $j\in \hat{\mathcal{A}}$. Add the fact that by construction $\hat{b}_{(\mathcal{A}^*)^c} = \mathbf{0}_p$, then $\hat{\mathcal{A}} \subset\mathcal{A}^* $ on the event $\Omega$. It then remains to prove that the event $\Omega$ occurs with high probability. We have
\begin{eqnarray}
	\mathbb{P} (\hat{\mathcal{A}} \nsubseteq \mathcal{A}^*)
	& \leq &
	\mathbb{P} (\Omega^c)
\nonumber
	\\
	& \leq &
	\sum_{k\notin \mathcal{A}^*} \mathbb{P} \left(
		\left|
			\sum_{j\in \mathcal{A}^*} (K_n)_{j,k} (\hat{b}_j - \beta_j^{*}) - \frac{X_k'\varepsilon}{n} + \mu_n \sum_{j\in \mathcal{A}^*} \widetilde{J}_{j,k} \beta_j^*
		\right|
		\geq
		\frac{\lambda_n}{2}
	\right)
\nonumber
	\\
	& \leq &
	\sum_{k\notin \mathcal{A}^*} \mathbb{P} \left(
		\left|
			\sum_{j\in \mathcal{A}^*} (K_n)_{j,k} (\hat{b}_j - \beta_j^{*}) - \frac{X_k'\varepsilon}{n}
		\right|
		\geq
		\frac{\lambda_n}{2} -  \mu_n |\widetilde{J}\beta^*|_{\infty}
	\right)
\nonumber
	\\
	& \leq &
	\sum_{k\notin \mathcal{A}^*} \mathbb{P} \left(
		\left|
			\sum_{j\in \mathcal{A}^*} (K_n)_{j,k} (\hat{b}_j - \beta_j^{*}) - \frac{X_k'\varepsilon}{n}
		\right|
		\geq
		\frac{\lambda_n}{4}
	\right)
\nonumber
	\\
	& \leq &
	\sum_{k\notin \mathcal{A}^*} \mathbb{P} \left(
		|\sum_{j\in \mathcal{A}^*} (K_n)_{j,k} (\hat{b}_j - \beta_j^{*})|
		\geq
		\frac{\lambda_n}{8}
	\right)
	+
	\sum_{k\notin \mathcal{A}^*} \mathbb{P} \left(
		|\frac{X_k'\varepsilon}{n}|
		\geq
		\frac{\lambda_n}{8}
	\right)
\label{ChapSLeq:ProofProbLongl}
\end{eqnarray}
where we used the fact that for real number $a$ and $b$, we have $|a|+|b| \geq|a+b|$ in the third inequality and the fact that $\mu_n = \frac{\lambda_n}{4 |\widetilde{J}\beta^*|_{\infty}}$ in the forth one. Let us consider the last two terms in the last display separately. i) First, since $\lambda_n = 16 \sigma \sqrt{ \frac{\log ( p / \sqrt{\eta p / (1+p)} )}{n} }$, and using close arguments to those employed in Lemma~\ref{ChapSLprobalm}, we obtain $\sum_{k\notin \mathcal{A}^*} \mathbb{P} \left(
		|\frac{X_k'\varepsilon}{n}|
		\geq
		\frac{\lambda_n}{8}
	\right) 
\leq 
\eta \frac{1}{1+p}$; ii) according to $\sum_{k\notin \mathcal{A}^*} \mathbb{P} \left(
		|\sum_{j\in \mathcal{A}^*} (K_n)_{j,k} (\hat{b}_j - \beta_j^{*})|
		\geq
		\frac{\lambda_n}{8}
	\right)$, we need to control $|\sum_{j\in \mathcal{A}^*} (K_n)_{j,k} (\hat{b}_j - \beta_j^{*})|$ for every $k\notin \mathcal{A}^*$. On one hand, Assumption~$D$ implies that
\begin{equation}
\label{ChapSLeq:ProofPtitINge}
	\forall \, k\notin \mathcal{A}^* \quad \quad |\sum_{j\in \mathcal{A}^*} (K_n)_{j,k} (\hat{b}_j - \beta_j^{*})|
	\leq
	\sum_{j\in \mathcal{A}^*} |\hat{b}_j - \beta_j^{*}| t/| \mathcal{A}^*|.
\end{equation}
By definition of $\hat{b}$, we just have to repeat the proof of Theorem~\ref{ChapSLThm:doubke} but with $\hat{b}$ instead of $\hat{\beta}^{Quad}$ and only on the true sparsity set $\mathcal{A}^*$. We get that on the event $\Lambda_{n,\mathcal{A}^*} = \left \lbrace \max_{j\in\mathcal{A}^*} |X_j'\varepsilon | \le \lambda_{n}/ 8 \right \rbrace$, which is the same that $\Lambda_{n,p}$ but using $\mathcal{A}^*$ instead of $\{1,\ldots,p\}$,
\begin{equation*}
	\sum_{j\in \mathcal{A}^*} |\hat{b}_j - \beta_j^{*}|
	\leq
	8 \phi_{\mu_n}^{-1} \lambda_n |\mathcal{A}^*|.
\end{equation*}
Moreover, similar reasoning as in Lemma~\ref{ChapSLprobalm} leads to $\mathbb{P} \left(\Lambda_{n,\mathcal{A}^*}^c \right) \leq \eta \frac{1}{1+p}$. Combine this result with~\eqref{ChapSLeq:ProofPtitINge} and get
\begin{eqnarray*}
	\sum_{k\notin \mathcal{A}^*} \mathbb{P}\left(
		|\sum_{j\in \mathcal{A}^*} (K_n)_{j,k} (\hat{b}_j - \beta_j^{*})|
		\geq
		\frac{\lambda_n}{8}
	\right)
	& \leq &
	p \, \mathbb{P}\left(
		\sum_{j\in \mathcal{A}^*} |\hat{b}_j - \beta_j^{*}|
		\geq
		\frac{| \mathcal{A}^*|\lambda_n}{8t}
	\right)
	\\
	& \leq &
	p \, \mathbb{P}\left(
		\sum_{j\in \mathcal{A}^*} |\hat{b}_j - \beta_j^{*}|
		\geq
		8 \phi_{\mu_n}^{-1} \lambda_n |\mathcal{A}^*|
	\right) 
	\\
	& \leq &
	p \, \mathbb{P}\left( \Lambda_{n,\mathcal{A}^*}^c \right) 
	\leq
	\eta \frac{p}{1+p},
\end{eqnarray*}
provided that $t \leq \frac{ \phi_{\mu_n}}{64}$. We finally conclude by this last inequality and~\eqref{ChapSLeq:ProofProbLongl} that $\mathbb{P} (\hat{\mathcal{A}} \nsubseteq \mathcal{A}^*) \leq \eta (\frac{1}{1+p} + \frac{p}{1+p}) \leq \eta.$ Then we get the desired result.
\end{proof}
\vspace{0.5cm}
\begin{proof}[\it{Proof of Theorem~\ref{ChapSLthm:BruitNonGaussian}}]
This proof is almost the same as the one of Theorem~\ref{ChapSLThm:doubkeSparsJ}.
The only difference is the way to control the event $\Lambda_{n,p}=\left \lbrace \max_{j=1,\ldots,p} 2|V_j| \le \tau \lambda_n \right \rbrace$ where $V_{j}= n^{-1}\sum_{i=1}^{n}x_{i,j}\varepsilon_{i}$ when the noise admits only zero mean and finite variance.
Then we do not use the concentration inequality provided in Lemma~\ref{ChapSLprobalm} for the Gaussian noise but an analog concentration inequality more adapted to this type of noise.
This concentration inequality is given by Lemma~\ref{ChpSLConcentNOnGauss} and we get
$$
	\mathbb{P}\left(\max_{j=1,\ldots,p} 2 |V_j| \le \tau \lambda_n \right)
	\geq
	1 - \eta,
$$
for a value of $\lambda_n=\frac{2\sigma}{\tau} \sqrt{\frac{K_{Nem} L}{n\eta} }$. Then we set $\tau=1/2$ and we plug this new value of the tuning parameter $\lambda_n$ instead to the one used to establish the previous results into Theorem~\ref{ChapSLThm:doubkeSparsJ}.
We just finish the proof by using the fact that $\mu_n = \frac{\lambda_n \sqrt{|\mathcal{A}^*|}}{2|\widetilde{J}\beta^*|_2}$ and we obtain the analogous of Corollary~\ref{Cor:BorneL1Sparse}.
\end{proof}

\vspace{4mm}

\noindent{\bf Acknowledgements.} We would like to thank the Referees for their helpful comments and suggestions. They help us to improve significantly this revised version of the paper.
We also would like to thank Johannes Lederer for insightful comments.

\bibliographystyle{plain}
\bibliography{slasso2}

\begin{thebibliography}{10}

\bibitem{BachGpLasso}
F.~Bach.
\newblock Consistency of the group lasso and multiple kernel learning.
\newblock {\em J. Mach. Learn. Res.}, 9:1179--1225, 2008.

\bibitem{Chernoz_postl1_10}
A.~Belloni and V.~Chernozhukov.
\newblock {P}ost-$\ell_1$-{P}enalized estimation in high dimensional sparse
  linear regression models.
\newblock Submitted, 2010.

\bibitem{Lasso3}
P.~Bickel, Y.~Ritov, and A.~Tsybakov.
\newblock Simultaneous analysis of lasso and {D}antzig selector.
\newblock {\em Ann. Statist.}, 37(4):1705--1732, 2009.

\bibitem{Bunea_consist}
F.~Bunea.
\newblock Consistent selection via the lasso for high dimensional approximating
  regression models.
\newblock 2008.
\newblock IMS Collections, B. Clarke and S. Ghosal Editors.

\bibitem{BuneaEN}
F.~Bunea.
\newblock Honest variable selection in linear and logistic regression models
  via {$\ell_1$} and {$\ell_1+\ell_2$} penalization.
\newblock {\em Electron. J. Stat.}, 2:1153--1194, 2008.

\bibitem{BTWAggSOI}
F.~Bunea, A.~Tsybakov, and M.~Wegkamp.
\newblock Aggregation for {G}aussian regression.
\newblock {\em Ann. Statist.}, 35(4):1674--1697, 2007.

\bibitem{Lasso2}
F.~Bunea, A.~B. Tsybakov, and M.~H. Wegkamp.
\newblock Sparsity oracle inequalities for the {L}asso.
\newblock {\em Electron. J. Stat.}, 1:169--194, 2007.

\bibitem{ChriMo7GpLass}
C.~Chesneau and M.~Hebiri.
\newblock Some theoretical results on the grouped variables lasso.
\newblock {\em Math. Methods Statist.}, 17(4):317--326, 2008.

\bibitem{ArnakTsyb}
A.~S. Dalalyan and A.~B. Tsybakov.
\newblock Aggregation by exponential weighting and sharp oracle inequalities.
\newblock In {\em Learning theory}, volume 4539 of {\em Lecture Notes in
  Comput. Sci.}, pages 97--111. Springer, Berlin, 2007.

\bibitem{Daye09Referee}
Z.~John Daye and X.~Jessie Jeng.
\newblock Shrinkage and model selection with correlated variables via weighted
  fusion.
\newblock {\em Computational Statistics \& Data Analysis}, 53(4):1284--1298,
  2009.

\bibitem{DvGVW10nemirovkiInequality}
Lutz D{\"u}mbgen, Sara~A. van~de Geer, Mark~C. Veraar, and Jon~A. Wellner.
\newblock Nemirovski's inequalities revisited.
\newblock {\em Amer. Math. Monthly}, 117(2):138--160, 2010.

\bibitem{Efron-LARS}
B.~Efron, T.~Hastie, I.~Johnstone, and R.~Tibshirani.
\newblock Least angle regression.
\newblock {\em Ann. Statist.}, 32(2):407--499, 2004.
\newblock With discussion, and a rejoinder by the authors.

\bibitem{FanLiScad}
J.~Fan and R.~Li.
\newblock Variable selection via nonconcave penalized likelihood and its oracle
  properties.
\newblock {\em J. Amer. Statist. Assoc.}, 96(456):1348--1360, 2001.

\bibitem{Mo7SLasso}
M.~Hebiri.
\newblock Regularization with the smooth-lasso procedure.
\newblock Preprint Laboratoire de Probabilit{\'e}s et Mod{\`e}les
  Al{\'e}atoires, 2008.

\bibitem{JY_EN_08}
J.~Jia and B.~Yu.
\newblock On model selection consistency of elastic net when $p\gg n$.
\newblock Tech. Report 756, Statistics, UC Berkeley, 2008.

\bibitem{KimKBoyd_trending_09}
S.~Kim, K.~Koh, S.~Boyd, and D.~Gorinevsky.
\newblock {$l_1$} trend filtering.
\newblock {\em SIAM Rev.}, 51(2):339--360, 2009.

\bibitem{Fusion}
S.~R. Land and J.~H. Friedman.
\newblock Variable fusion: a new method of adaptive signal regression.
\newblock {\em Manuscript}, 1996.

\bibitem{KarimNormSup}
K.~Lounici.
\newblock Sup-norm convergence rate and sign concentration property of {L}asso
  and {D}antzig estimators.
\newblock {\em Electron. J. Stat.}, 2:90--102, 2008.

\bibitem{GpLassLogistic}
L.~Meier, S.~van~de Geer, and P.~B{\"u}hlmann.
\newblock The group {L}asso for logistic regression.
\newblock {\em J. R. Stat. Soc. Ser. B Stat. Methodol.}, 70(1):53--71, 2008.

\bibitem{RelaxedLasso}
N.~Meinshausen.
\newblock Relaxed {L}asso.
\newblock {\em Comput. Statist. Data Anal.}, 52(1):374--393, 2007.

\bibitem{MeinshBulhmConsistLasso}
N.~Meinshausen and P.~B{\"u}hlmann.
\newblock High-dimensional graphs and variable selection with the lasso.
\newblock {\em Ann. Statist.}, 34(3):1436--1462, 2006.

\bibitem{PValues}
N.~Meinshausen, L.~Meier, and P.~B{\"u}hlmann.
\newblock p-values for high-dimensional regression.
\newblock {\em J. Amer. Statist. Assoc.}, 104:1671--1681, 2009.

\bibitem{MeinYuSelect}
N.~Meinshausen and B.~Yu.
\newblock Lasso-type recovery of sparse representations for high-dimensional
  data.
\newblock {\em Ann. Statist.}, 37(1):246--270, 2009.

\bibitem{NesterovAlgo07}
Yu. Nesterov.
\newblock Gradient methods for minimizing composite objective function.
\newblock CORE Discussion Papers 2007076, Universit\'e catholique de Louvain,
  Center for Operations Research and Econometrics (CORE), Sep 2007.

\bibitem{RWY09}
G.~Raskutti, M.~Wainwright, and B.~Yu.
\newblock Minimax rates of estimation for high-dimensional linear regression
  over $\ell_q$-balls.
\newblock Technical report.
\newblock submitted.

\bibitem{RT10}
P.~Rigollet and A.~Tsybakov.
\newblock Exponential {S}creening and optimal rates of sparse estimation.
\newblock Technical report, Jul 2010.
\newblock submitted.

\bibitem{Rinaldo08FusedAdaptive}
A.~Rinaldo.
\newblock Properties and refinements of the fused lasso.
\newblock {\em Ann. Statist.}, 37(5B):2922--2952, 2009.

\bibitem{Rosset-PeacewiseLin}
S.~Rosset and J.~Zhu.
\newblock Piecewise linear regularized solution paths.
\newblock {\em Ann. Statist.}, 35(3):1012--1030, 2007.

\bibitem{Tibshirani-LASSO}
R.~Tibshirani.
\newblock Regression shrinkage and selection via the lasso.
\newblock {\em J. Roy. Statist. Soc. Ser. B}, 58(1):267--288, 1996.

\bibitem{Rosset-Fused}
R.~Tibshirani, M.~Saunders, S.~Rosset, J.~Zhu, and K.~Knight.
\newblock Sparsity and smoothness via the fused lasso.
\newblock {\em J. R. Stat. Soc. Ser. B Stat. Methodol.}, 67(1):91--108, 2005.

\bibitem{TibJunior_Smooth_10}
R.J. Tibshirani and J.~Taylor.
\newblock Regularization paths for least squares problems with generalized
  $\ell_1$ penalties.
\newblock Submitted, 2010.

\bibitem{Tsyb-VanDeGeer-SquareRoot}
A.~B. Tsybakov and S.~A. van~de Geer.
\newblock Square root penalty: adaptation to the margin in classification and
  in edge estimation.
\newblock {\em Ann. Statist.}, 33(3):1203--1224, 2005.

\bibitem{VandeGeerConditionLasso09}
S.~van~de Geer and P.~B{\"u}hlmann.
\newblock On the conditions used to prove oracle results for the lasso.
\newblock {\em Elect. Journ. Statist.}, 3:1360--1392, 2009.

\bibitem{WainSelection}
M.~Wainwright.
\newblock Sharp thresholds for noisy and high-dimensional recovery of sparsity
  using l1-constrained quadratic programming.
\newblock Manuscript, 2006.

\bibitem{YeZhang10}
F.~Ye and C.~Zhang.
\newblock Rate {M}inimaxity of the {L}asso and {D}antzig {S}elector for the
  $\ell_q$ {L}oss in $\ell_r$ {B}alls.
\newblock {\em J. Mach. Learn. Res.}, 11:3519--3540, 2010.

\bibitem{GpLasso1}
M.~Yuan and Y.~Lin.
\newblock Model selection and estimation in regression with grouped variables.
\newblock {\em J. R. Stat. Soc. Ser. B Stat. Methodol.}, 68(1):49--67, 2006.

\bibitem{NonNegativeGarotte}
M.~Yuan and Y.~Lin.
\newblock On the non-negative garrote estimator.
\newblock {\em J. R. Stat. Soc. Ser. B Stat. Methodol.}, 69(2):143--161, 2007.

\bibitem{YuehValeurPropre}
W-C. Yueh.
\newblock Eigenvalues of several tridiagonal matrices.
\newblock {\em Appl. Math. E-Notes}, 5:66--74 (electronic), 2005.

\bibitem{ZhangHuangConsist}
C-H. Zhang and J.~Huang.
\newblock The sparsity and bias of the {LASSO} selection in high-dimensional
  linear regression.
\newblock {\em Ann. Statist.}, 36(4):1567--1594, 2008.

\bibitem{BiYuConsistLasso}
P.~Zhao and B.~Yu.
\newblock On model selection consistency of {L}asso.
\newblock {\em J. Mach. Learn. Res.}, 7:2541--2563, 2006.

\bibitem{AdapLassoZou}
H.~Zou.
\newblock The adaptive lasso and its oracle properties.
\newblock {\em J. Amer. Statist. Assoc.}, 101(476):1418--1429, 2006.

\bibitem{Zou-E-Net}
H.~Zou and T.~Hastie.
\newblock Regularization and variable selection via the elastic net.
\newblock {\em J. R. Stat. Soc. Ser. B Stat. Methodol.}, 67(2):301--320, 2005.

\bibitem{ZZ_AdaptEN_09}
H.~Zou and H.~Zhang.
\newblock On the adaptive elastic-net with a diverging number of parameters.
\newblock {\em Ann. Statist.}, 37(4):1733--1751, 2009.

\end{thebibliography}

\end{document}